\documentclass[10pt,a4paper]{article}

\usepackage[utf8]{inputenc}
\RequirePackage[l2tabu, orthodox]{nag}
\usepackage[american]{babel}
\usepackage[
			hmargin=2cm,
			vmargin=2cm,
			twoside,
			a4paper,			
			nomarginpar]
			{geometry}
\usepackage{amssymb,amsthm}
\usepackage{amsmath}
\usepackage{xspace} 
\usepackage[dvipsnames]{xcolor}
\usepackage{dsfont}
\usepackage{microtype}
\usepackage{appendix}
\DisableLigatures[f]{encoding = *}
\usepackage{ellipsis} 
\usepackage{bm}
\usepackage{enumerate}
\usepackage{algpseudocode}
\usepackage[ruled,linesnumbered]{algorithm2e}
	\algnewcommand\And{\textbf{and}}
\usepackage{multicol,multirow,booktabs}

\definecolor{dred}{RGB}{220,0,0}
\definecolor{halfgray}{gray}{0.55}
\definecolor{webgreen}{rgb}{0,.5,0}
\definecolor{webbrown}{rgb}{.6,0,0}
\usepackage[hyperfootnotes=false]{hyperref}
\hypersetup{%
    colorlinks=true, linktocpage=true, 
    breaklinks=true, 
    plainpages=false, bookmarksnumbered, bookmarksopen=true, bookmarksopenlevel=1,%
    hypertexnames=true, 
    urlcolor=webbrown, linkcolor=RoyalBlue, citecolor=webgreen, 
    pdftitle={},%
    pdfauthor={Alessandro Zocca},%
}  

\usepackage{tabularx} 
\usepackage{caption}
\usepackage[caption=false]{subfig}  
\usepackage{placeins}
\usepackage{psfrag}
\usepackage{graphicx}
\usepackage{pgfplots}
	\pgfplotsset{width=7cm,compat=1.8}
\usepackage{array}
    \newcolumntype{C}[1]{>{\raggedleft\arraybackslash}p{#1}}
    
\usepackage{soul}
\setstcolor{red}
\usepackage{booktabs}
\usepackage{diagbox}

\newcommand{\set}[1]{\left\{#1\right\}}

\newcommand{\normm}[1]{{\left\vert\kern-0.25ex\left\vert\kern-0.25ex\left\vert #1 \right\vert\kern-0.25ex\right\vert\kern-0.25ex\right\vert}}

\DeclareMathOperator{\tr}{tr}

\newcommand{\R}{\mathbb R}
\newcommand{\N}{\mathbb N}

\newcommand{\bs}{\backslash}

\newcommand{\hi}{s}
\newcommand{\hj}{t}
\newcommand{\rt}{R_{\mathrm{tot}}}
\newcommand{\ff}{\bm{f}}
\newcommand{\pp}{\bm{p}}

\newcommand{\bmo}{\bm{1}}
\newcommand{\bt}{\bm{\theta}}
\renewcommand{\a}{\bm{\alpha}}
\newcommand{\ee}{\bm{e}}

\renewcommand{\l}{\lambda}
\renewcommand{\b}{\beta}

\newcommand{\im}{b_{\calI}}

\newcommand{\calB}{\mathcal{B}}
\newcommand{\calBB}{\calB\calB}

\newcommand{\calE}{\mathcal{E}}

\newcommand{\calG}{\mathcal{G}}

\newcommand{\calI}{\mathcal{I}}

\newcommand{\calL}{\mathcal{L}}

\newcommand{\calP}{\mathcal{P}}

\newcommand{\calV}{\mathcal{V}}

\newtheorem{thm}{Theorem}[section]
\newtheorem{cor}[thm]{Corollary}
\newtheorem{lem}[thm]{Lemma}
\newtheorem{prop}[thm]{Proposition}

\title{A Spectral Representation of Power Systems with Applications to Adaptive Grid Partitioning and Cascading Failure Localization} 
\date{\today}
\author{Alessandro Zocca\footnote{Department of Mathematics, Vrije Universiteit Amsterdam, 1081HV, Netherlands, \texttt{a.zocca@vu.nl}} \and Chen Liang, Linqi Guo, Steven H.~Low, Adam Wierman \footnote{Department of Computing and Mathematical Sciences, California Institute of Technology, Pasadena, CA 91125, \texttt{\{cliang2, lguo, slow, adamw\}@caltech.edu}}}

\begin{document}
\maketitle

\begin{abstract}
Transmission line failures in power systems propagate and cascade non-locally. This well-known yet counter-intuitive feature makes it even more challenging to optimally and reliably operate these complex networks. In this work we present a comprehensive framework based on spectral graph theory that fully and rigorously captures how multiple simultaneous line failures propagate, distinguishing between non-cut and cut set outages.
Using this spectral representation of power systems, we identify the crucial graph sub-structure that ensures line failure localization -- the network bridge-block decomposition.
Leveraging this theory, we propose an adaptive network topology reconfiguration paradigm that uses a two-stage algorithm where the first stage aims to identify optimal clusters using the notion of network modularity and the second stage refines the clusters by means of optimal line switching actions. Our proposed methodology is illustrated using extensive numerical examples on standard IEEE networks and we discussed several extensions and variants of the proposed algorithm.
\end{abstract}


\section{Introduction}
\label{sec:intro}
Electrical power grids are among the most complex and critical networks in modern-day society, reliably bringing power from generators to end users, cities and industries, often very far away geographically from each other. 
Traditionally, power systems have been designed as one-directional networks, where electric energy travels over high-voltage transmission lines from big conventional and controllable generators to distribution networks and eventually to consumers. This paradigm is changing rapidly in recent years due to many concurrent trends. Firstly, more distributed energy resources are coming online and consumers are slowly becoming prosumers, shifting the conventional one-directional power flow paradigm to bi-directional. There are massive investments in renewable energy sources, whose power generation is geographically correlated, more volatile, and less controllable than traditional sources of generation. Moreover, we are electrifying our transportation systems, which is extremely important for reducing our greenhouse gas emissions. 
However, electric vehicles have high energy demands and the current grid design and operations will not be able to sustain a high penetration of EVs, especially in view of correlated charging patterns. 
Managing increasing and correlated loads while having more volatile and less controllable generation may seem to be an impossible task, especially when aiming to keep the same operational and reliability standards. 

Given these trends, power grids are becoming increasingly stressed and have less margin for maneuver, making failures more likely and harder to contain, increasing the chance of blackouts. The growing complexity and increasing stochasticity of these networks challenge the classical reliability analyses and strategies. For instance, preventing line failures and mitigating their non-local propagation will become ever harder, having to deal with a broader range of more variable power injection configurations and more bidirectional flows on transmission lines.

To respond to these increasing challenges,  power systems infrastructure is becoming more adaptive and responsive. Power systems have been traditionally looked at as a static network, but in fact they are not, since many of the transmission lines can be remotely taken online/offline. Rapid control mechanisms and corrective switching actions are increasingly being used to improve network reliability~\cite{shao2005corrective} and reduce operational costs~\cite{fisher2008optimal,hedman2009optimal}, especially since it is possible to quickly and efficiently estimate the current status of the grid using new monitoring devices and data processing strategies~\cite{huang2012state}.

However, new transmission infrastructure can be expensive and its placement will not necessarily increase reliability. Therefore, we need to make optimal use of the existing system and its adaptability. Optimally and dynamically switching transmission lines is an inexpensive and promising option because it uses existing hardware to achieve increased grid robustness. The ability of the operator to adaptively change the topology of the grid depending on the current network configuration offers a great potential. However, even with perfect information about the system, finding the best switching actions in real time is not an easy task. The flows on the lines are fully determined by power flow physics and network topology (cf.~Kirchhoff's and Ohm's laws), which means that every switching action causes a global power flow redistribution. In view of the combinatorial nature of this problem and of the large scale of the network, it is clear that a brute force approach will fail. 



\subsection{Contributions of the paper}

This paper aims to tackle the aforementioned challenges in power systems operations by taking full advantage of transmission line flexibility.  To accomplish this requires new mathematical tools to first understand and then optimally and reliably operate these increasingly complex systems. In particular, a new understanding of the role of the network topology, especially when it comes to non-local failure propagation, is needed. 

To this end, we introduce new spectral graph theory tools to analyze and optimize power systems. In particular, we propose a new spectral representation of power systems that effectively captures complex interactions within power networks, for instance inter-dependencies between infrastructure components and failure propagation events.
The key observation underlying this representation is the fact that, under the DC power flow approximation, the linear relation between power injections and line flows can be expressed using the weighted graph Laplacian matrix, where the line susceptances are used as edge weights. The eigenvalues and eigenvectors of this matrix thus contain rich information about the topology of the network as well as on its physical and electrical properties. Starting from this observation, Section~\ref{sec:spectral} illustrates how the graph Laplacian and its pseudo-inverse can be used to characterize islands of the power network and the power balance conditions of each island. 

When using conventional representations, the impact of line failures is discontinuous and notoriously difficult to characterize or even approximate, but our spectral representation provides a simple and exact characterization. More specifically, in Section~\ref{sec:distfactor} we show how the impact of network topology changes on line flows, quantified using the so-called \textit{distribution factors}, can be exactly described using spectral quantities. The results, most of which already appeared in~\cite{TPS1,TPS2} cover not only the case of single line outage, but also the more general case of the simultaneous outage of multiple lines. The analysis distinguishes two fundamentally different scenarios depending on whether the network remains connected or not, namely \textit{cut set outages} and \textit{non-cut set outages}.

Leveraging the spectral representation and the distribution factors analysis, in Section~\ref{sec:localization} we identify which graph sub-structures affect power flow redistribution and fully characterize the line failure propagation. More specifically, our analysis reveals that the network line failure propagation patterns can be fully understood by focusing on network block and bridge-block decompositions, which are related, respectively, to the cut vertices and cut edges of the power network. 

Our failure localization results suggest the robustness of the network is often improved by reducing the redundancy of the network, which indirectly makes the aforementioned decomposition finer. Using this insight, we then suggest in Section~\ref{sec:optimization} a novel design principle for power networks and propose algorithms that leverage their pre-existing flexibility to increase their robustness against failure propagation. More specifically, we propose a new procedure that, by temporarily switching off transmission lines, can be used to optimally modify the network topology in order to refine the bridge-block decomposition. We accomplish this through solving a novel two-stage optimization problem that adaptively modifies the power network structure, using the current power injection configuration as input. 

In Subsection~\ref{sub:mod} we introduce the first step of this procedure, which identifies a target number of clusters in the power network by solving an optimization problem whose objective function is a weighted version of the classical \textit{network modularity} problem. Being an NP-hard problem to solve exactly, we show how spectral clustering methods can be used to quickly and effectively obtain good approximated solutions, i.e., nearly-optimal clusters. 
The second step of our procedure, presented in Subsection~\ref{sub:cut}, identifies line switching actions that transform the identified clusters into bridge-blocks, refining the bridge-block decomposition of the network. The optimal subset of line switches are selected using a combinatorial optimization problem which aims to minimize the congestion of transmission lines of the network while achieving the target network bridge-block decomposition. 

This two-step procedure, which we refer to as the one-shot algorithm, is summarized in Subsection~\ref{sub:pra}, where we also introduce a faster recursive variant that iteratively refines the bridge-block decomposition by bi-partitioning the biggest bridge-block. The numerical implementation of both algorithms is discussed in Section \ref{sec:numerics}, where we test their performance on a large family of IEEE test networks of heterogeneous size. As revealed by our extensive analysis, almost all these power networks have a trivial bridge-block decomposition consisting of a single giant bridge-block, suggesting that their robustness could be greatly improved by our procedure.

Most of the present paper focuses on reliability issues and, in particular, line failure propagation, but we conclude in Section~\ref{sec:conclusion} by highlighting some other promising applications of the network block and bridge-block decompositions. Specifically, we demonstrate how a finer-grained block decomposition of a power network can be leveraged to (a) design a real-time failure localization and mitigation scheme that provably prevents and localize successive failures; (b) accelerate the standard security constrained OPF by decomposing and decoupling the problem into smaller versions in each block; (c) enable more efficient distributed solvers for AC OPF by transforming a globally coupled constraint into their counterparts for each block; and (d) tractably quantify the market price manipulation power of aggregators.

\subsection{Related literature}
\label{sub:literature}
The deep interplay between the topology of a given power system and power flow physics is for some aspects unique in network theory and a large body of literature has been devoted to the challenge of finding effective representations and approximations for these complex networked systems. These are instrumental not only to understand and analyze the behavior of these networks, but also to develop fast and effective algorithms and optimization strategies. In this subsection we briefly review the existing related literature on this topic.

Our work uses tools from spectral graph theory, an established and vast field in which many good books are available, see e.g.~\cite{brouwer2011spectra,VanMieghem2011} and reference therein. Particularly relevant for our analysis is~\cite{RanjanZhangBoley2014}, which focuses on the efficient computation of the pseudo-inverse of the graph Laplacian and explores its intimate connection with effective resistances. 
Graph spectra-based methods have been extensively used in the context of power systems, in particular in the study of phase angles frequency dynamics~\cite{Dorfler2013,Dorfler2014,Dorfler2018,Pirani2017,Koeth2019,Koeth2019b}, but also for power system restoration~\cite{QuirosTortos2013} and to analyze of the geometry of power flows~\cite{Retire2019,caputo2019spectral}. A spectral characterization for the network bridges, which play a crucial role in our analysis, was already given in~\cite{Cetinay2018,Soltan2017b}.

Understanding the underlying structure of given graph using spectral graph methods (e.g., Cheeger's inequalities) is a classical problem that received great attention in both discrete math and computer science literature. A canonical problem is the minimum $k$-cut problem, which aims to find the best partition in $k$ clusters of a given graph~\cite{StoerWagner1997}. The same problem has then been rediscovered in other domains, e.g., computer vision~\cite{ShiMalik2000}, with more emphasis on how to quickly find approximated solutions for NP-hard $k$-way partitioning problems. In this context, several clustering techniques based on spectral methods have been proposed~\cite{ShiMalik2000,NgJordanWeiss2002}, whose properties have later been studied analytically in~\cite{vonLuxburg2007}. At the same time, there was an increasing interest in the physics literature to study large networks arising in various domain (among which the world wide web and epidemiology) and unveil their underlying community structure, see~\cite{Fortunato2016} for a review. It is in this context that the concept of network modularity~\cite{Newman2004b,Newman2004} that we use in this paper has been introduced. Besides network modularity, many other techniques developed for complex networks have been applied to power systems, in particular to capture their main topological features and assess their robustness~\cite{PaganiAiello2013,Cuadra2015}.

The approach we propose in this paper crucially exploits the fact that the network topology of power systems can be changed by means of line switching actions. The classical \textit{Optimal Transmission Switching} (OTS) problem leverages the same existing flexibility, but with different goals. This optimization problem was first introduced in~\cite{koglin1980overload,Glavitsch1985,mazi1986corrective} as a corrective strategy to alleviate line overloading due to contingencies. Afterwards, transmission switching has been explored in the literature as a control method for various problems such as voltage security~\cite{shao2006,khanabadi2012optimal}, line overloads~\cite{granelli2006,shao2006}, loss and/or cost reduction~\cite{bacher1988,Fliscounakis2007,schnyder1990,Han2016,fisher2008optimal,hedman2010coopt}, clearing contingencies~\cite{khanabadi2012optimal,korad2013robust}, improving system security~\cite{schnyder1988}, or a combination of these objectives \cite{bacher1986network,rolim1999,shao2005corrective}. The interested readers can also read~\cite{Hedman2011} for a comprehensive review of the OTS literature.

Power grids are naturally divided in control regions~\cite{Panciatici2018}, raising the issue of how to optimally cluster and operate these networks. For this reason, clustering recently received a lot of attention in the power systems literature.  A very diverse set of methodologies have been considered, ranging from spectral clustering~\cite{Bialek2014} to various heuristics based on the modularity score~\cite{Guerrero2017,Lin2017}. A substantial effort has been devoted to expand and augment classical clustering methods in order to account for specific features of power systems, for instance using the notion of electrical distance~\cite{Blumsack2009,Cotilla2013} or conductances~\cite{LuWen2013,LuWen2015}. 
In the context of cascading failure analysis, clustering and community detection methods have also been used on the abstract interaction graphs~\cite{Nakarmi2018,Nakarmi2019} rather than on the physical network topologies.

Many ad-hoc clustering algorithms have been developed in the power systems literature particularly in the context of \textit{intentional controlled islanding} (ICI), an extreme security mechanism against cascading failures in which the network is disconnected into several self-sustained ``islands'' to prevent further contingencies. 
The goal of the standard ICI problem is to find the optimal partition of the network into islands while including additional constraints to ensure generator coherency and minimum power imbalance. This inevitably makes the clustering problem even harder and nontrivial trade-offs arise, which is the reason why several heuristics and approximations methods have been developed for ICI in the recent literature, see~\cite{HassaniAhangar2020,Liu2020} for an overview.


Intentional controlled islanding is a rather extreme response to large-scale cascading failures. In~\cite{Bialek2021} the author propose as an alternative emergency measure beside precisely on the core properties of the bridge-block decomposition. Several other mitigation strategies less drastic than ICI could be adopted, but they often require more detailed cascade models. However, modeling cascading failures mathematically in power systems is rather complex due to the underlying power flow physics. The book~\cite{Bienstock2015} gives a comprehensive overview of the various models that have been introduced in the literature to describe cascading failures in power systems as well as the various optimization approaches that have been devised to improve the robustness of these systems. Most cascading failure models usually consider line or generator failures as trigger events, but correlated stochastic fluctuations of the power injections have also been considered in~\cite{Nesti2018}.

Contingency analysis is a very commonly used tool for reliable operations of a power system that assesses the impact of either generator and transmission line outages~\cite{Wood2013}. Such an impact can be assessed by solving the AC power flow equations that describe the network after each contingency.  Due to the large number of contingencies that must be assessed in order to satisfy $N-k$ security for $k \geq 1$, it is a common practice to first use DC power flow models to quickly screen contingencies and select a much smaller subset that result in voltage or line limit violations for more detailed analysis using AC power flow models. Such a contingency screening uses the \textit{power transfer distribution factors} (PTDFs) and \textit{line outage distribution factors} (LODFs) for a DC power flow model, see~\cite[Chapter 7]{Wood2013} for more details. The main advantage of using these factors is that the impact of generator and transmission outages on the post contingency networks can be analyzed using the common pre-contingency topology across contingency scenarios.
LODFs for multi-line outages were first considered in~\cite{EnnsQuadaSackett1982} to study the impact of network changes on line currents using the network Laplacian matrix, obtained from the admittance matrix after the linearization of AC power flow equations. However, the refined proof approach using the matrix inversion lemma was introduced later in~\cite{AlsacStottTinney1983,StottAlsacAlvarado1985}. Note that the paper~\cite{StottAlsacAlvarado1985} allows for more general outages (e.g., generator outages) and proposes a methodology to quickly rank contingencies in security analysis.
Probably unaware of the results in~\cite{AlsacStottTinney1983, StottAlsacAlvarado1985}, some of the formulas for generalized LODF (GLODF) have been re-discovered in \cite{ GulerGross2007,GulerGrossLiu2007,GuoShahidehpour2009} using different approaches based on the Power Transfer Distribution Factors (PTDF). More specifically, line outages are emulated through changes in injections on the pre-contingency network by judiciously choosing injections at the endpoints of each outaged/disconnected line using PTDF. 
Distribution factors for linear flows have been studied extensively in the recent literature~\cite{Horsch2018,Kaiser2020c,Ronellenfitsch2017,Ronellenfitsch2017b,Strake2019}.
Some recent work~\cite{TPS1,TPS2,Guo2018,Guo2019,Kaiser2020a,Kaiser2020b} rigorously proves that the presence of specific graph sub-structures guarantees that some of these distribution factors are zero, hence showing the potential for localizing outages and avoiding a global propagation by means of optimal network reconfiguration. The localization effects of these graph substructures can be effectively visualized by means of an ad-hoc version of the so-called \textit{influence} (or \textit{interaction}) \textit{graph}~\cite{Hines2013,Hines2017,Nakarmi2018,Nakarmi2019}. In~\cite{Guo2019,Liang2021,TPS3} the authors also explore how carefully designed network substructures synergize well with other control mechanisms that provide congestion management in real time. 
LODFs are also studied more recently as a tool to quantify network robustness and flow rerouting \cite{Strake2019}. While PTDF and LODF determine the sensitivity of power flow solutions to parameter changes, one can also study the sensitivity of optimal power flow solutions to parameter changes; see, e.g., \cite{Gribik1990, Hauswirth2018}.

\section{A spectral representation of power systems}
\label{sec:spectral}
A power transmission network can be described as a graph $G=(V,E)$, where $V=\set{1,\ldots, n}$ is the set of vertices (modeling \textit{buses}) and $E\subset V \times V$ is the set of edges (modeling \textit{transmission lines}). Denote by $n=|V|$ the number of vertices and by $m=|E|$ the number of edges of the network $G$. 
In this section we describe power network model that we consider throughout the paper and highlight an intimate connection between spectral graph theory and power flow physics. We first review in Subsection~\ref{sub:basic} some key notions from graph theory and some classical network decompositions that are crucial for our analysis. Then, in Subsection~\ref{sub:dcmodel}, we introduce the power flow model that we will focus on in this paper in spectral terms and present some immediate results that follow from this representation.


\subsection{Network decompositions and substructures}
\label{sub:basic}
Due to power flow physics, the topology of a power transmission network plays a central role in the way power flows on its lines. In fact, the power flows are intrinsically determined by the network physical structure via Kirchhoff's laws once that the power injections are fixed. This also means that the power flow redistribution after a contingency and possible cascading outages are intimately related to the network structure. It comes as no surprise that to study the robustness of a power network makes uses of advance graph theory notions and algorithms. We now briefly review some useful network decompositions and substructures that will play a crucial role in our analysis in the next sections. 

In this subsection we look at a transmission power network in purely topological terms as an undirected and unweighted graph $G$, thus temporarily ignoring the physical and electrical properties of its transmission lines. The $k$-\textit{partition} of a network $G=(V,E)$ is a finite collection $\calP=\set{\calV_1,\calV_2,\cdots, \calV_k}$ of nonempty and disjoint subsets of $V$ such that $\bigcup_{i=1}^k \calV_i=V$. 
Denote by $\Pi_k(G)$ the collection of $k$-partitions of $G$ and by $\Pi(G) := \bigcup_{i=1}^{|V|} \Pi_i(G)$ the collection of all its partitions.
Given a partition $\calP$, each edge of the network $G$ is either a \textit{cross edge} if its two endpoints belong to different clusters of $\calP$ and \textit{internal edge} otherwise. We denote the collections of cross edges and internal edges as $E_c(\calP)$ and $E_i(\calP)$, respectively.
There is a natural partial order $\succeq$ on the collection $\Pi(G)$ of partitions, defined as follows: given two partitions $\calP^1=\set{\calV_1^1,\calV_2^1,\cdots,\calV_{k_1}^1}$ and $\calP^2=\set{\calV_1^2,\calV_2^2,\cdots,\calV_{k_2}^2}$, we say that $\calP^1$ is \emph{finer} than $\calP^2$, denoted as $\calP^1\succeq \calP^2$, if each subset in $\calP^1$ is contained in some subset in $\calP^2$. More precisely, $\calP^1\succeq \calP^2$ if for every $i=1,2,\ldots, k_1$, there exists some $j(i)\in \set{1,2,\ldots, k_2}$ such that $\calV_i^1\subseteq \calV_{j(i)}^2$. 

Each partition $\calP$ induces a \textit{quotient graph} $G/\calP$, which is the undirected graph whose vertices are the clusters in $\calP$ and two distinct clusters $\calV_i, \calV_j \in \calP$, are adjacent if and only if there exists at least one cross edge in the original graph $G$ whose endpoints belong one to $\calV_i$ and one to $\calV_j$. 
A slight variation of the quotient graph is the \textit{reduced graph} $G_\calP$, which is defined as the (multi)graph whose vertices are still the clusters in $\calP$, but in which we draw an edge between two distinct clusters for every cross edge between them. The reduced graph $G_\calP$ can thus have multiple (parallel) edges, and coincides with the quotient graph only when it is simple.

It is well known that any equivalence relation on the set of vertices of the network $G$ induces a partition. We now briefly review three canonical decompositions of an undirected graph that have been introduced in the graph-theory literature (see, e.g.,~\cite{Harary1971,Westbrook1992}) using specific equivalence relations. 

First of all, consider the partition $\calP_\mathrm{path}=\{\calI_1,\dots, \calI_k\}$ in which two vertices are in the same subset if and only if there is a path connecting them. In this case $E_c(\calP_\mathrm{path})=\emptyset$ by definition, the quotient and reduce graphs trivially coincides and and have no edges. 
The subgraphs $G_1,\dots,G_k$ induced by $\calP_\mathrm{path}$ are the \textit{connected components} of $G$, also referred to as \textit{islands} in the power systems literature. A graph is \textit{connected} if it consists of a single connected component, i.e.~$|\calP_\mathrm{path}|=1$. In general, a power network $G$ may not be connected at all times, either due to operational choice or as a result of severe contingencies).

The notion of graph connectedness allows to introduce two more notions that will be crucial for our network analysis. A subset $\calE \subset E$ is a \textit{cut set} of $G$ if the deletion of all the edges in $\calE$ disconnects $G$. If the cut set consists of a single edge, we refer to it as \textit{bridge} or \textit{cut edge}. A \textit{cut vertex} of $G$ is any vertex whose removal (together with its incident edges) increases the number of its connected components.

A second network decomposition can be obtained by looking at the network circuits. Recall that a circuit is a path from a vertex to itself with no repeated edges. We consider the partition $\calP_\mathrm{circuit}(G)$ such that two vertices are in the same class if and only if there is a circuit in $G$ containing both of them.
The subgraphs induced by the subset of nodes in each of the equivalence classes of $\calP_\mathrm{circuit}(G)$ are precisely the connected components of the graph obtained from $G$ by deleting all the bridges and, for this reason, they are called \textit{bridge-connected components} or \textit{bridge-blocks}. We will refer to the partition $\calP_\mathrm{circuit}(G)$ as \textit{bridge-block decomposition} of $G$ and denote it by $\mathcal{B}\mathcal{B}(G)$.

The next lemma summarizes the intuitive fact that the bridge-block decomposition of a graph becomes finer if some of its edges are removed. This result is the cornerstone of the strategy to improve the network reliability that is presented in Section~\ref{sec:optimization}.
\begin{lem}[Removing edges makes bridge-block decomposition finer]\label{lem:finer}
For any graph $G=(V,E)$ and subset of edges $\calE \subset E$, the bridge-block decomposition of the graph $G^\calE:=(V,E\setminus \calE)$ obtained from $G$ by removing the edges in $\calE$ is always finer than that of $G$, i.e.,
$$\calBB(G^\calE) \succeq  \calBB(G).$$
\end{lem}
\begin{proof}
Consider the partition $\calBB(G)=\calP_{\mathrm{circuit}}(G)$ and recall that two vertices belong to the same equivalence class (bridge-block) if and only if there exists a circuit that contains both of of them. By removing a subset of edges $\calE$ from $G$, each of the equivalence classes either remains unchanged or gets partitioned in smaller equivalence classes. This readily implies that for every new bridge-block $B \in \calBB(G^\calE)$ obtained after the edge removal there exists a bridge-block $B'$ of the original network, i.e., $B' \in \calBB(G)$, such that $B \subseteq B'$. 
\end{proof}

A third network decomposition can be obtained by considering the network cycles, which are the circuits in which the only repeated vertices are the first and last ones. More specifically, we consider the equivalence relation for edges of being contained in a cycle and denote by $\{E_1,E_2,\dots,E_k\}$ to be the resulting partition of $E$. Let $R_i=(V_i,E_i)$ be the corresponding subgraph of $G$, consisting of the edges in $E_i$ and the vertices $V_i$ that are the endpoints of these edges. We refer to each subgraph $R_i$ either as \textit{block} and to $\mathcal{B}(G)=\{V_1,\dots,V_k\}$ as the \textit{block decomposition} of $G$. Every block is a maximal biconnected (or 2-connected) subgraph, since the removal of any of its vertices does not disconnect it. A vertex could appear in more than one block and, since its removal disconnects $G$, this is an equivalent characterization of cut vertices. The block decomposition $\mathcal{B}(G)$ thus is not a partition of the vertex set $V$, but nonetheless for any pair of vertices there is either no block or exactly one block containing them both. We distinguish two types of blocks, \textit{trivial blocks} and \textit{nontrivial blocks}, depending on whether they consist of a single edge or not. A trivial block consists of an edge that is not contained in any cycle; such an edge must then be a bridge since its removal disconnects the graph, cf.~\cite[Chapter 3]{Harary1971}. Two nontrivial blocks either share a cut vertex or are connected by a bridge, i.e., a trivial block. 

We remark that a statement similar to Lemma~\ref{lem:finer} holds also for the edge partition $\{E_1,E_2,\dots,E_k\}$ corresponding to the block decomposition of a graph $G$ and for the block decomposition $\calB(G)$ itself. However, for the strategy to improve network reliability that we present in Section~\ref{sec:optimization}, it turns out that $\calB(G)$ is less convenient to work with than the bridge-block decomposition $\calBB(G)$ since the block decomposition $\calB(G)$ is not a proper vertex partition (recall that the cut vertices always appear in two or more blocks).


Both the block and bridge-block decompositions of a graph $G$ are unique. We can informally say that the block decomposition is finer than the bridge-block decomposition, since each bridge corresponds to a trivial block, while every bridge-block consists of several non-trivial blocks connected by cut vertices. Such a nested structure of the block and bridge-block decompositions is illustrated for a small graph in Fig.~\ref{fig:networkdec}. Another visualization of these network decompositions for an actual power network can be found later in Section~\ref{sec:localization}, see Fig.~\ref{fig:influencegraph_pre}.

A graph naturally decomposes into a tree/forest of blocks called its \textit{block-cut tree/forest}, depending on whether the original graph is connected or not. More specifically, in such a block-cut tree/forest there is a vertex for each block and for each cut vertex, and there is an edge between any block and cut vertex that belongs to that block. Similarly, also the bridge-blocks and bridges of $G$ have a natural tree/forest structure, called \textit{bridge-block tree/forest}, depending if $G$ is connected or not. Such a tree/forest is the graph with a vertex for each bridge-block and with an edge for every bridge. We refer the reader to~\cite[Chapter 3]{Harary1971} for further details.

\begin{figure}[!h]
\centering
\subfloat[]{\includegraphics[width=.26\textwidth]{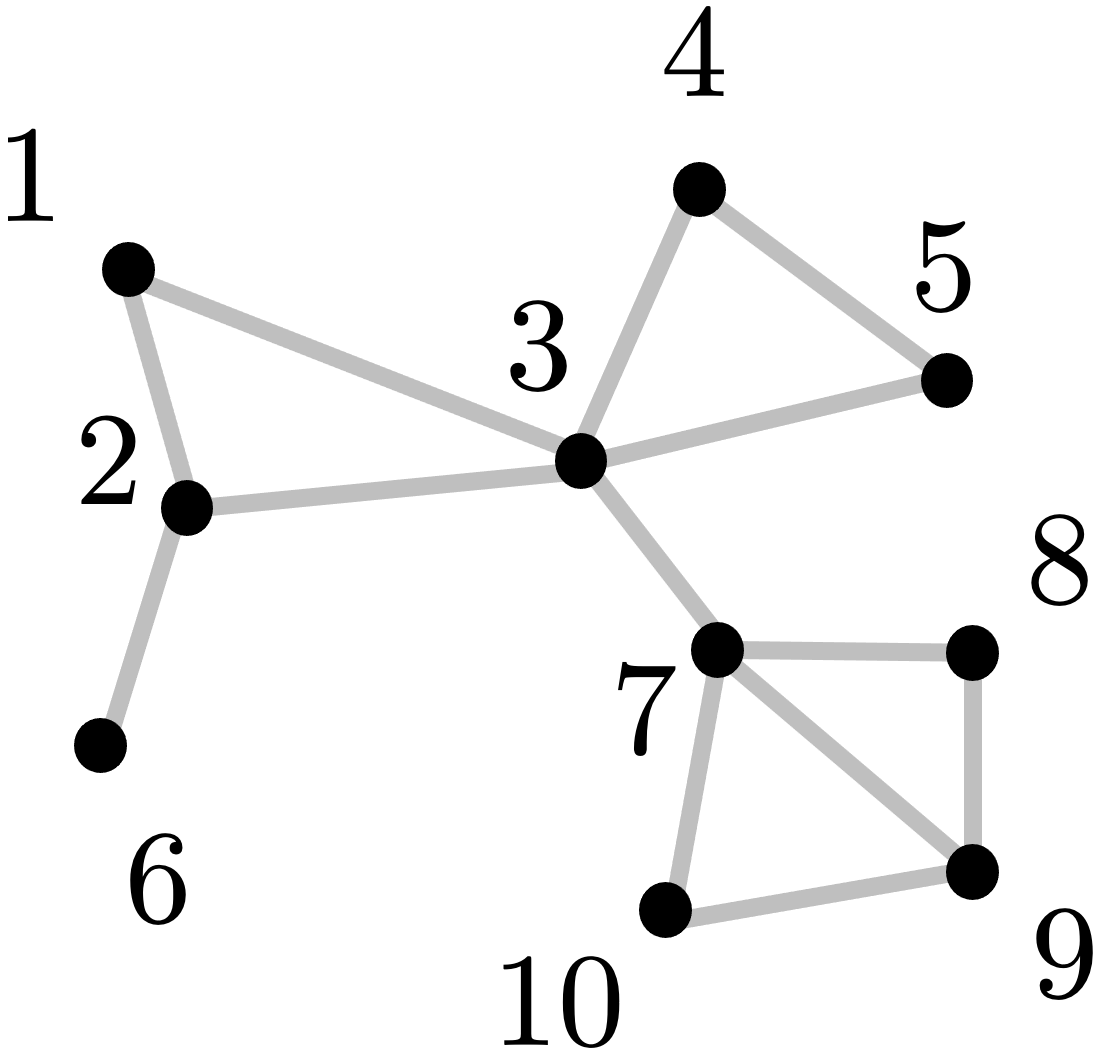}}
\hspace{1cm}
\subfloat[]{\includegraphics[width=.26\textwidth]{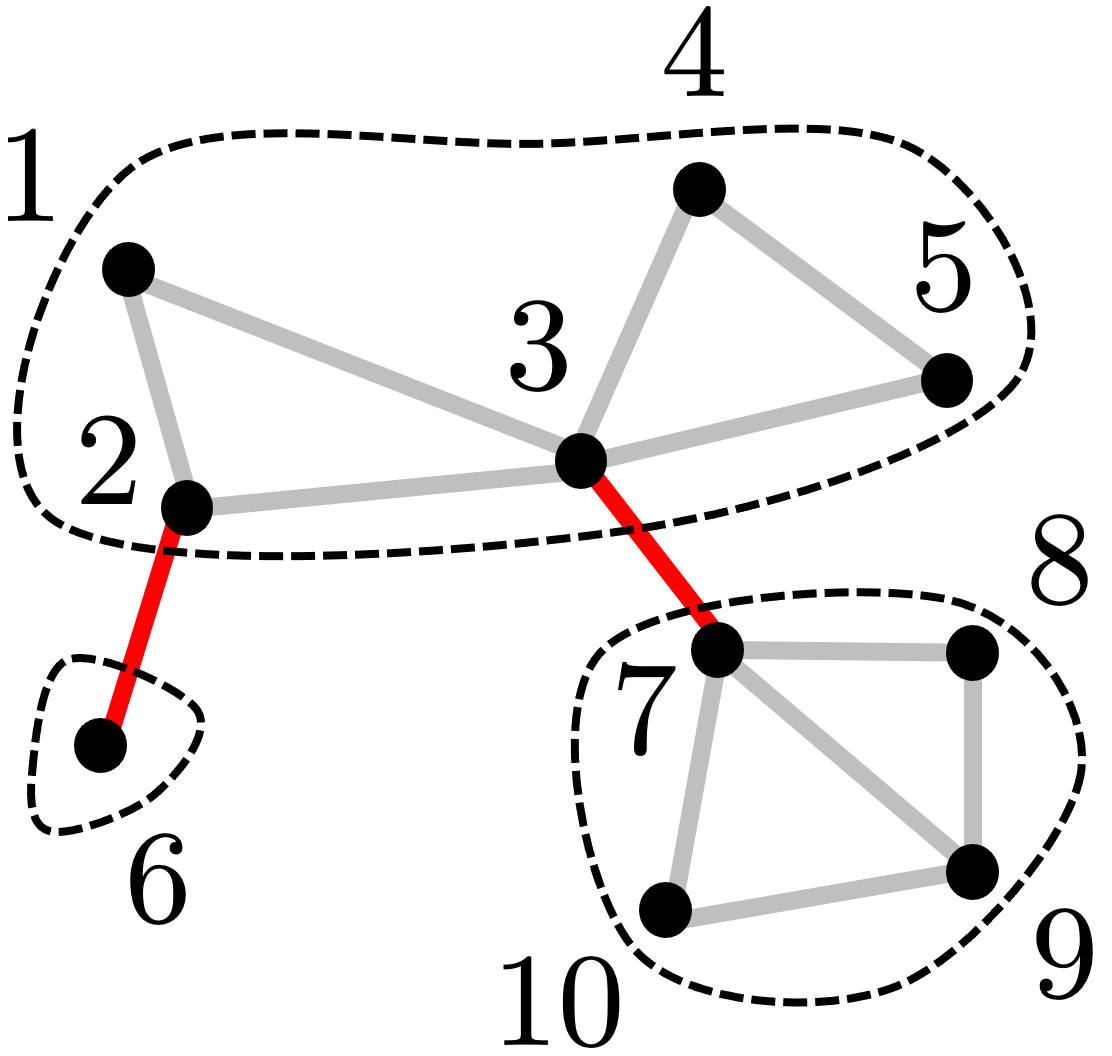}}
\hspace{1cm}
\subfloat[]{\includegraphics[width=.3\textwidth]{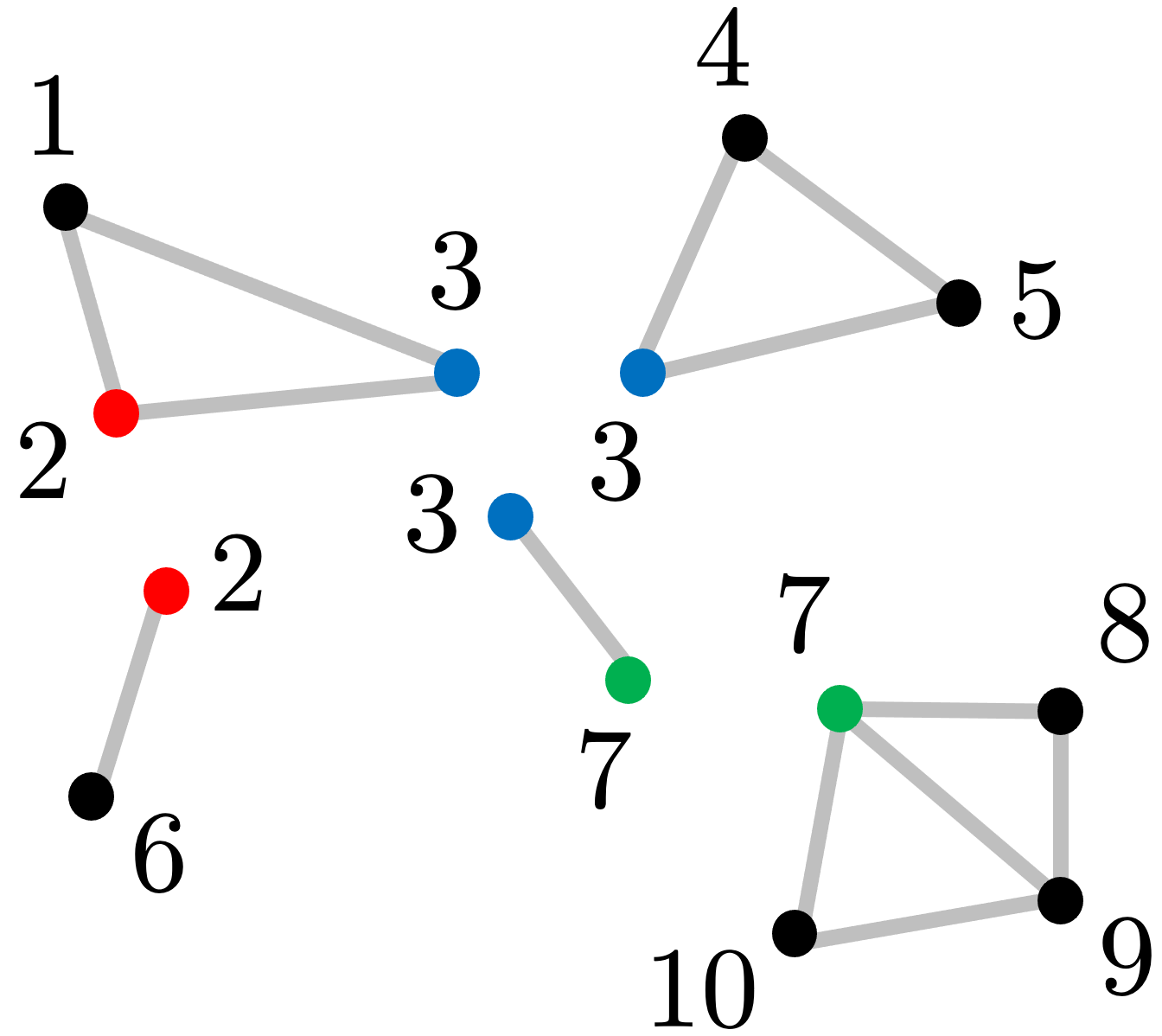}}
\caption{An undirected graph $G$ in (a) with its bridge-block decomposition in (b) and with its block decomposition in (c). Edges $(2, 6)$ and $(3, 7)$ are bridges. Vertices $2, 3, 7$ are the cut vertices of $G$ and appear in multiple blocks.}
\label{fig:networkdec}
\end{figure}
\FloatBarrier

The problems of finding the blocks and bridge-blocks of a given graph are well understood: sequential algorithms that run in time $O(n + m)$ were given in~\cite{Hopcroft1973,Tarjan1972} and logarithmic-time parallel algorithms are given in~\cite{Awerbuch1987,Tarjan1985}.
Later in Section~\ref{sub:PTDF} we also give a precise spectral characterization for bridges (cf. Corollary \ref{ch:lm; corollary:PTDF.3}).

Another complementary powerful tool to study the topology of power networks is using its spanning trees, which will play a crucial role in the study of distribution factors in Section~\ref{sec:distfactor}. A spanning tree of $G$ is a subgraph that is a tree with a minimum possible number of edges that includes all of the vertices of $G$. Let $T_\calE$ be the collection of all spanning trees of $G$ consisting only of edges in $\calE\subseteq E$ and $T = T_E$ the set of all spanning trees of $G$. We further denote by $T\!\setminus\! \calE := T_{E\!\setminus\! \calE}$ the collection of all spanning trees of $G$ consisting of edges in $E\!\setminus\! \calE$. For any pair of disjoint subsets of vertices $V_1, V_2\subset V$, we denote by $T(V_1, V_2)$ the collection of spanning forests of $G$ consisting of exactly two trees necessarily disjoint, one containing the vertices in $V_1$ and the other one containing those in $V_2$. Note that $T_\calE$ can be empty if the subset $\calE$ is too small and that $T(V_1, V_2)$ is empty by definition whenever $V_1$ and $V_2$ are not disjoint.

\subsection{Power flow model and its spectral representation}
\label{sub:dcmodel}
In this subsection, we present the classical DC power flow model and a few key spectral properties that are intertwined with the power flow dynamics.

In order to accurately model a power system, it is important to capture the physical properties of its transmission lines. When necessary, we thus look at $G=(V,E)$ as weighted graph, in which each edge $\ell=(i,j) \in E$ has a weight $b_\ell=b_{ij}=b_{ji}>0$ that models the \textit{susceptance} of the corresponding line. To stress the difference with its unweighted counterpart, we will refer to the weighted graph $G=(V,E,\bm{b})$ as \textit{power network} and to its edges as lines.

It is crucial for our analysis to capture the ``electrical properties'' of specific parts or components of the power network, in particular of its spanning trees. For this reason we introduce a nonnegative weight for any subset $\calE \subseteq E$ of lines by defining a function $\beta : 2^E\rightarrow \mathbb R_+$ as follows: 
\[
	\beta(\calE) := \prod_{\ell\in \calE}\, b_\ell, \qquad \calE \subseteq E.
\]
In other words, $\beta(\calE)$ is the product of the susceptances of all the lines in the subset $\calE$. For consistency we set $\beta(\calE) = 0$ if $\calE = \emptyset$ and $E\neq\emptyset$, but $\beta(\calE) = 1$ when the power network $G$ consists of a single vertex, i.e., $\calE = E = \emptyset$.

Let $B:=\mathrm{diag}(b_1,\dots,b_m) \in \R^{m \times m}$ be the diagonal matrix with the line susceptances (i.e., the edge weights) $b_1,\dots,b_m$ as entries, assuming a fixed order has been chosen for the edge set $E$. Aiming to describe line flows, it is necessary to introduce an (arbitrary) orientation for the edges, which is captured by the vertex-edge incidence matrix $C \in \{-1,0,+1\}^{n\times m}$ defined as
$$
C_{i \ell}=\begin{cases}
  1 & \text{if vertex }i\text{ is the source of }\ell,\\
  -1 & \text{if vertex }i\text{ is the target of }\ell,\\
  0 &\text{otherwise.}
\end{cases}
$$
The spectral representation of power systems relies on the \textit{weighted Laplacian matrix} of the power network $G$, that is the symmetric $n \times n$ matrix $L$ defined as $L:=C B C^T$ or, equivalently in terms of the susceptances, as
\begin{equation}
\label{eq:Laplaciandef}
    L_{i,j} :=
    \begin{cases}
    		\sum_{k \neq i} b_{i k} 		& \text{ if } i=j,\\
		- b_{ij}  							& \text{ if } i \neq j.
	\end{cases}
\end{equation}
Denote by $\l_1, \l_2 , \dots, \l_{n}$ its eigenvalues and by $\bm{v}_1,\bm{v}_2,\dots,\bm{v}_n$ the corresponding eigenvectors, which we take to be of unit norm and pairwise orthogonal. The next proposition summarizes some crucial properties of the weighted Laplacian matrix $L$ and its \textit{Moore-Penrose pseudo-inverse} $L^\dag$. The proof can be found in the appendix and we refer the reader to~\cite{RanjanZhangBoley2014,VanMieghem2011} for further spectral properties of graphs.

\begin{prop}[Laplacian matrix $L$ and its pseudo-inverse]
\label{ch:lm; eq:Laplacian-properties}
The following properties hold for the weighted Laplacian matrix of a power network $G = (V, E, \bm{b})$:
\begin{itemize}
\item[(i)] For every $\bm{x} \in\R^n$ we have $\bm{x} ^T L \bm{x}  = \sum_{(i,j)\in E}\, b_{ij} (x_i - x_j)^2 \geq 0$.
\item[(ii)] $L$ is a symmetric positive semidefinite matrix and, hence, has a real non-negative spectrum, i.e.,~$0 \leq \l_1 \leq \l_2 \leq \dots \leq \l_{n}$.
\item[(iii)] All the rows (and columns) of $L$ sum to zero and thus the matrix $L$ is singular.
\item[(iv)] The matrix $L$ has rank $n-k$, where $k\geq 1$ is the number of connected components $\calI_1,\dots, \calI_k$ of $G$. The eigenvalue $0$ has multiplicity $k$ and $\mathrm{ker}(L)=\mathrm{ker}(L^\dag)=\mathrm{span}(\bm{1}_{\calI_1},\dots,\bm{1}_{\calI_k})$, where $\bm{1}_{\calI_j} \in \{0,1\}^n$ is the vector which is $1$ on $\calI_j$ and $0$ elsewhere. Furthermore, the pseudo-inverse of $L$ can be calculated as
	\[
	    L^\dag = \sum_{j=k+1}^n \frac{1}{\lambda_j} \bm{v}_j \bm{v}_j^T = \left( L + \sum_{j=1}^k \frac{1}{|\calI_j|} \bmo_{\calI_j} \bmo_{\calI_j}^T \right)^{-1}  - \sum_{j=1}^k \frac{1}{|\calI_j|} \bmo_{\calI_j} \bmo_{\calI_j}^T.   
	\]
In particular, in the special case in which $G$ is connected, i.e., $k=1$, $\mathrm{ker}(L)=\mathrm{ker}(L^\dag) = \mathrm{span}(\bmo)$, where $\bm{1} \in \R^n$ is the vector with all unit entries and
	\[
		L^\dag = \sum_{j=2}^{n} \frac{1}{\l_j} \, \bm{v}_j \bm{v}_j^T = \left( L + \frac{1}{n} \bmo \bmo^T \right)^{-1} \ - \ \frac{1}{n} \bmo \bmo^T.
	\]
\item[(v)] $L^\dag$ is a real, symmetric, positive semi-definite matrix with zero row (and hence column) sums and its nonzero eigenvalues are $0 < \l_n^{-1} \leq \dots \leq \l_{k+1}^{-1}$. 
\end{itemize}
\end{prop}

Denote by $\pp \in \R^{n}$ the vector of power injections, where the entry $p_i$ models the power generated (if $p_i>0$) or consumed (if $p_i<0$) at the bus corresponding to vertex $i$. We say that a power injection vector $\pp \in \R^n$ is \textit{balanced} if $\pp \in \mathrm{im}(L)$. Since $ \mathrm{im}(L^\dag) = \mathrm{ker}(L^\dag)^\bot=\mathrm{span}(\bm{1}_{\calI_1},\dots,\bm{1}_{\calI_k})^\bot$, a power injection vector $\pp$ is balanced if and only if the net power injection is equal to zero in every island $\calI_j$ of the network, i.e., $\sum_{i \in \calI_j} p_i = 0$.

Transmission power systems are operated using alternating current (AC), but the equations describing the underlying power flow physics are non-linear leading to many computational challenges in solving them~\cite{Wood2013}. Therefore, linearization techniques are often used to approximate the power flow equations. We now introduce the so-called \textit{DC power flow model}, which is a first-order approximation of the AC equations that is commonly used to model high-voltage transmission system~\cite{Wood2013,Stott2009,Dorfler2013b}.

Given an (oriented) edge $\ell=(i,j) \in E$, we denote interchangeably as $f_\ell$ or $f_{ij}$ the power flow on that line. We set $f_{ji}=-f_{ij}$, with the convention that a power flow is negative if the power flows in the opposite direction with respect to the edge orientation. The so-called \textit{DC power flow model} is the system of equations
\begin{subequations}\label{eq:dc_model}
\begin{align}
& \pp = C \ff, \label{eq:conservation}\\
& \ff = BC^T \bt, \label{eq:kirchhoff}
\end{align}
\end{subequations}
where $\bt \in \R^n$ is the vector of \textit{phase angles} at the network vertices and $\ff\in \R^m$ is the vector of line power flows. Equation~\eqref{eq:conservation} captures the flow-conservation constraint, while equation~\eqref{eq:kirchhoff} describes Kirchhoff's laws.
The DC model \eqref{eq:dc_model} has a unique solution $\ff$ for each balanced injection vector $\pp$. Indeed, from~\eqref{eq:dc_model} we can deduce that $\pp=CBC^T\bt = L \bt$, and, ultimately, obtain the following spectral representation for the power flows
\begin{equation}
\label{eq:flowsdc_model}
	\ff =BC^T L^\dag \pp.
\end{equation}
The power flows are thus uniquely determined by the injections and physical properties of the network. 
In view of~\eqref{eq:flowsdc_model}, it is clear that the network operator does not directly control the power routing, which can be only indirectly changed by modifying the power injections or the network topology. This peculiar feature is one of the reasons why it is challenging to study and improve the reliability of power systems and their robustness against failures.

The interactions between components in a power network depend, not only on its topology, but also on the physical properties of its components and the way they are coupled by power flow physics. Such an ``electrical structure'' can be effectively unveiled and studied using spectral methods, more specifically looking at the so-called \textit{effective resistance}, which is formally defined as follows. The effective resistance $R_{i,j}$ between a pair of vertices $i$ and $j$ of the network $G$ is the non-negative quantity
\begin{equation}
\label{eq:effr}
	R_{ij} := (\mathbf{e}_i - \mathbf{e}_j)^T L^\dag (\mathbf{e}_i - \mathbf{e}_j) = L^\dag_{ii} + L^\dag_{jj} - 2 L^\dag_{ij},
\end{equation}
where $\mathbf{e}_i$ denotes the vector with a $1$ in the $i$-th coordinate and $0$'s elsewhere. $R_{ij}$ quantifies how ``close'' vertices $i$ and $j$ are in the power network $G$ and takes small value in the presence of many paths with high susceptances between them. The effective resistance is a proper distance on $V\times V$ and for this reason it is often referred to as \textit{resistance distance}~\cite{Klein1993}. 
The \textit{total effective resistance} of a graph $G$ is defined as $\rt(G)=\frac{1}{2} \sum_{i,j=1}^n R_{ij}$ and it is intimately related to the spectrum of $G$ by the following identity proved by~\cite{Klein1993}
\[
    \rt(G) = n \cdot \tr(L^\dag) = n \cdot \sum_{j=2}^n \frac{1}{\lambda_j}.
\]
The same quantity is also known as \textit{Kirchhoff index} in the special case when all the edges of the network $G$ have unit weights, i.e., $b_\ell=1$, $\ell=1,\dots,m$. The total effective resistance is a key quantity that measures how well connected the network is and for this reason has been extensively studied and rediscovered in various contexts, such as complex network analysis~\cite{Ellens2011} and probability theory~\cite{Doyle2000}.
The notion of effective resistances have been first introduced in the context of power networks analysis by~\cite{Cotilla-Sanchez2012} and~\cite{Cotilla2013} and since then have been extensively used to study their robustness, see e.g.,~\cite{Cetinay2018,Dorfler2018,Koc2014,Wang2014,Wang2015}, as well as to devise algorithm to improve their reliability, see e.g.~\cite{Ghosh2008,Johnson2010,Zocca2017}.

\FloatBarrier

\section{Distribution factors and power flow redistribution}
\label{sec:distfactor}
In this section we introduce two families of sensitivity factors of power networks, also known as \textit{distribution factors}. The first class is that of \textit{power transfer distribution factors} (PTDF), which describe how changes in power injections impact line flows. The second type of factors are the so-called \textit{generalized line outage distribution factors} (GLODF), which capture how line removal/outages impact power flows on the surviving lines. As illustrated in the previous section, if the power injections or the network topology change, one can always recompute the new line flows by solving the power flow equations~\eqref{eq:flowsdc_model}. The distribution factors are based on the DC approximation for power flows and provide fast estimates of the new line flows without solving again the AC power flow equations. For this reason, they have been widely used for security and contingency analysis when power flow solutions are computationally expensive, see~\cite[Chapter 7]{Wood2013}. 

In this paper, we use distribution factors to analyze structural properties of power flow solutions and to design failure localization and mitigation mechanisms to reduce the risk of large-scale blackouts from cascading failures. In this section we unveil the deep connection between the distribution factors and specific substructures of the power network topology, complementing~\cite{Ronellenfitsch2017,Strake2019}. We henceforth assume the power network to be connected. The rest of the section is structured as follows. We first focus on the impact of power injection changes on line flows in Subsection~\ref{sub:PTDF}, introducing the PTDF factors and matrix and showing how they can be calculated in terms of spectral quantities. We then illustrate the effect of network topology changes on power flows. In our analysis we distinguish two very different cases. In Subsection~\ref{sub:noncut} we look at the scenario in which the set of outaged/removed lines does not disconnect the power network, obtaining the close-form expression for the GLODFs. In Subsection~\ref{sub:cut} we consider the more involved case in which the outaged lines disconnect the power network into multiple islands and show how the impact of topology change can be clearly distinguished from that of the balancing rule that needs to be invoked to rebalance the power in each of these newly created islands.

\subsection{Power transfer distribution factors (PTDF)}
\label{sub:PTDF}
The goal of this subsection is to analyze the effect of power injections change on line flows. We will use $\tilde{\cdot}$ to denote variables after such a change (possibly a contingency, but not exclusively).

Consider two nodes $s,t \in V$ (not necessarily adjacent) and a scalar $\delta_{\hi\hj}$ describing an injection change, which can be either positive or negative. Suppose the injection at node $\hi$ is increased from $p_{\hi}$ to $\tilde p_{\hi} := p_{\hi }+ \delta_{\hi\hj}$ and that at node $\hj$ is reduced from $p_{\hj}$ to $\tilde p_{\hj} := p_{\hj} - \delta_{\hi\hj}$, while all other injections remain unchanged $\tilde p_i = p_i$, $i \neq \hi, \hj$, so that
\[
	\tilde{\pp} =  \pp + \delta_{\hi \hj} (\mathbf{e}_{\hi} - \mathbf{e}_{\hj})
\]
remains balanced, i.e., $\sum_{i\in V} \tilde p_i = \sum_{i\in V} p_i = 0$. Note that this is an equivalent condition a power injection for being balanced since we assumed that $G$ is connected and thus $\mathrm{ker}(L)=\bm{1}$ (cf.~Proposition~\ref{ch:lm; eq:Laplacian-properties}).

The \textit{power transfer distribution factor} (PTDF) $D_{\ell, \hi\hj}$, also called the \textit{generation shift distribution factor}, is defined to be the resulting change $\Delta f_{\ell} := \tilde f_{\ell} - f_{\ell}$ in the line flow on line $\ell$ normalized by $\delta_{\hi\hj}$:
\[
	D_{\ell, \hi\hj} := \frac{\Delta f_{\ell}}{\delta_{\hi\hj}} =  \frac{\tilde f_{\ell} - f_\ell }{\delta_{\hi\hj}}.
\]
For convenience $D_{\ell,\hi\hj}$ is defined to be $0$ if $i=j$ or $\hi = \hj$. The next proposition shows how this factor can be computed from the pseudo-inverse $L^\dag$ of the graph Laplacian matrix $L$ or, alternatively, in terms of the spanning forests of $G$. 
\begin{prop}[PTDF $D_{\ell,\hi \hj}$]
\label{ch:lm; prop:PTDF.1}
Given a connected power network $G = (V, E, \bm{b})$, the following identities hold for every line $\ell = (i, j)\in E$ and any pair of nodes $\hi, \hj \in V$:
%
\begin{equation}
\label{eq:Dlst}
    D_{\ell, \hi\hj} \ = \ b_{\ell} \left( L^\dag_{i\hi} +  L^\dag_{j\hj} -  L^\dag_{i\hj} -  L^\dag_{j\hi} \right) \ = \ b_{\ell} \ \frac{\sum_{\calE\in T(\{i,\hi\}, \{j,\hj\})}\, \beta(\calE) \ - \ \sum_{\calE\in T(\{i,\hj\}, \{j,\hi\})}\, \beta(\calE)}{\sum_{\calE\in T}\, \beta(\calE)}.
\end{equation}
\end{prop}
Recall that $T(\{i,\hi\}, \{j,\hj\})$ has been defined in Section~\ref{sec:spectral} as the collection of spanning forests of $G$ consisting of exactly two disjoint trees, one containing nodes $i$ and $\hi$ and the other one containing nodes $j$ and $\hj$ (the definition of $T(\{i,\hj\}, \{j,\hi\})$) is analogous).
The proof easily follows combining~\cite[Theorem 4]{TPS1} and the properties of the pseudo-inverse $L^\dag$ derived in~\cite{Guo2017}.
Besides the precise spectral relation with pseudo-inverse of the graph Laplacian, the previous proposition shows that the PTDFs depend \textit{only on the topology and line susceptances} and are thus independent of the injections $\pp$. We remark that the insensitivity of the PTDFs to injections is, however, specific to the DC power flow approximation, which yields a linear relation between power injections and line flows.

As noted above, $D_{\ell,\hi \hj}$ is the change in power flow on line $\ell$ when a unit of power is injected at node $\hi$ and withdrawn at node $\hj$ where nodes $\hi$ and $\hj$ need not be adjacent.  When they are adjacent, i.e., $\hat \ell := (\hi, \hj)$ is a line in $E$, then the PTDFs define a $m\times m$ matrix $D := ( D_{\ell\hat \ell}, \, \ell, \hat \ell\in E)$, to which we refer as the \textit{PTDF matrix}. 
The following proposition summarizes some properties of the PTDF matrix that already appeared without proof in~\cite[Corollary 5]{TPS1}. A detailed proof is thus provided in the Appendix.
\begin{prop}[PTDF matrix]
\label{ch:lm; prop:PTDF.2}
If $G = (V, E, \bm{b})$ is a connected power network, then the following properties hold for the PTDF matrix $D= ( D_{\ell\hat \ell}, \, \ell, \hat \ell\in E)$:
\begin{itemize}
\item[(i)] $D \ = \  B C^T  L^\dag  C$.

\item[(ii)] For every line $\ell \in E$ the corresponding diagonal entry of the PTDF matrix $D$ is given by
	\[
	D_{\ell \ell} = 1 \ - \ \frac{ \sum_{\calE\in T\!\setminus\! \{\ell\}} \beta(\calE) } { \sum_{\calE\in T} \beta(\calE) }.
	\]
	Hence, in particular, $0<D_{\ell\ell} \leq 1$.
\item[(iii)] For every line $\ell = (i,j) \in E$ the corresponding diagonal entry of the PTDF matrix $D$ is given by
	\[
	D_{\ell \ell} = b_{\ell} \left( L^\dag_{ii} +  L^\dag_{jj} -  2 L^\dag_{ij} \right) = b_{\ell} R_{ij}.
	\]
\end{itemize}
\end{prop}
Identity (ii) reveals that the more ways to connect every node without going through $\ell = (i,j)$, the smaller $D_{\ell\ell}$ is, i.e., the smaller it is the impact on the power flow on line $\ell$ when a unit of power is injected and withdrawn at nodes $i$ and $j$ respectively.

Recall that $T\!\setminus\! \{\ell\}$ is the collection of all spanning trees of $G$ consisting of edges in $E\!\setminus\! \{\ell\}$. Using~\eqref{eq:effr} and the fact that if $\ell$ is a bridge, then $T\!\setminus\! \{\ell\} = \emptyset$, we immediately obtain the following spectral characterization for the bridges of the network.
\newpage
\begin{cor}[Spectral characterization of network bridges]
\label{ch:lm; corollary:PTDF.3}
If $G = (V, E, \bm{b})$ is a connected power network, the following three statements are equivalent:
\begin{itemize}
\item Line $\ell$ is a bridge;
\item $D_{\ell\ell} = 1$;
\item $R_{ij} = L^\dag_{ii} +  L^\dag_{jj} -  2 L^\dag_{ij} = b_{\ell}^{-1}$.
\end{itemize}
\end{cor}
The proof immediately follows from the statements (ii) and (iii) of Proposition~\ref{ch:lm; prop:PTDF.2} noticing that line $\ell$ is a bridge if and only if it belongs to \textit{any} spanning tree of the power network $G$, or equivalently $T\!\setminus\! \{\ell\}=\emptyset$. 

Lastly, we present a sufficient condition for having zero PTDF in purely topological terms. Its proof is given in the Appendix and relies on the representation of the PTDF in terms of spanning forests given in~\eqref{eq:Dlst}.
\begin{thm}[Simple Cycle Criterion for PTDF]
\label{thm:zeroPTDF}
If there is no simple cycle in a power network $G = (V, E, \bm{b})$ that contains both lines $\ell$ and $\hat \ell$, then $D_{\ell \hat \ell} = 0$.
\end{thm}

\subsection{GLODF for non-cut set}
\label{sub:noncut}
Consider a connected power network $G=(V,E, \bm{b})$ with balanced injections $\pp$. When a subset of lines $\calE \subset E$ trips or is removed from service, the network topology changes and, as a consequence, the power flow redistributes on the \textit{surviving network} $G^\calE :=(V,E \setminus \calE)$ as prescribed by power flow physics.
To fully understand the power flow redistribution and unveil the impact of network topology on the flow changes, we distinguish two scenarios, depending on whether $\calE$ is a cut set or not for the power network $G$, and analyzed them separately. 

Even if the disconnection of the subset of lines $\calE \subset E$ can be deliberate, for instance due to maintenance or network optimization (as it will be the case in the Section~\ref{sec:optimization}), in most cases it is due to a \textit{contingency}, e.g., a line or node (bus) failure. For this reason and for conciseness, we present the next family of distribution factors using the standard power system terminology, referring to the lines in $\calE$ as \textit{outaged lines} and distinguishing between \textit{pre-contingency} and \textit{post-contingency line flows}.

Let $\calE\subsetneq E$ be any non-cut set of $k:=|\calE|$ lines that are simultaneously removed from the power network. The resulting surviving network $G^\calE=(V,E\setminus \calE)$ is still connected. Assuming that the injections $\pp$ remain unchanged (and thus balanced), the new power flows can be calculated directly from~\eqref{eq:flowsdc_model} after having updated the matrices $B$, $C$, and $L$ to reflect the structure of the surviving network $G^{\calE}=(V,E\setminus \calE)$. We now review this calculations in more detail and show how it leads to the notion of \textit{generalized line outage distribution factors (GLODFs}. These factors quantify the impact of the removal of each line in $\calE$ to each of the survival lines and are thus essential for any power network contingency analysis.

We first introduce some auxiliary notation. Let $\ff_\calE$ be the vector of the pre-contingency power flows on the outaged lines in $\calE$ and let $\ff_{-\calE}$ and $\tilde \ff_{-\calE}$ be the pre and post-contingency power flows respectively on the surviving lines in $-\calE := E\!\setminus\! \calE$. 
Partition the susceptance matrix $B$ and the incident matrix $ C$ into submatrices corresponding to surviving lines in $-\calE$ and outaged lines in $\calE$:
\begin{equation}
	B \ =: \ \textup{diag} (B_{-\calE}, \, B_{\calE}), \qquad  C \ =: \ [ C_{-\calE} \ \  C_{\calE}].
\label{ch:cf; subsec:LODF; eq:BCpartitions.1}
\end{equation}
Assuming that the injections $\pp$ remain unchanged, we can calculate the post-contingency network flows by solving the DC power flow equations~\eqref{eq:flowsdc_model}
$
	\tilde \ff_{-\calE} \ = \ B_{-\calE}  C_{-\calE}^T  L_{-\calE}^\dag  \pp,
$
expressed in terms of the Laplacian matrix $ L_{-\calE} :=  C_{-\calE} B_{-\calE}  C_{-\calE}^T$ of the surviving network. 

The main result of this section shows that for a non-cut set $\calE \subsetneq E$ the post-contingency flow net changes $\Delta \ff_{-\calE} := \tilde \ff_{-\calE} - \ff_{-\calE}$ depend linearly on the pre-contingency power flows $\ff_\calE$ on the outaged lines. Their sensitivity to $\ff_\calE$ defines a $|E \bs \calE| \times |\calE|$ matrix $K^\calE := (K^\calE_{\ell\hat \ell}, \ell\in E\setminus\! \calE, \hat \ell \in \calE)$, called the \textit{Generalized Line Outage Distribution Factor} (GLODF) matrix, through
\[
	\Delta \ff_{-\calE} = K^\calE \, \ff_\calE,
\]
i.e., $\Delta f_\ell = \tilde f_\ell - f_\ell \ = \ \sum_{\hat \ell\in \calE}\, K^\calE_{\ell\hat \ell} f_{\hat \ell}$, for any $\ell\in E\setminus\! \calE$. Like the PTDFs, also the GLODFs depend solely on the network topology and susceptances, as illustrated by the next theorem, which shows how the GLODF matrix $K^{\calE}$ can be calculated using spectral quantities.

To state this result, we first need some additional notation. Partition the PTDF matrix $D$ into submatrices corresponding to non-outaged lines in $-\calE$ and outaged lines in $\calE$, possibly after permutations of rows and columns. Since $D \ = \  B C^T  L^\dag  C$ from Proposition \ref{ch:lm; prop:PTDF.2}(i), these different blocks of $D$ can be written explicitly in terms of the submatrices of $B$ and $C$ introduced in~\eqref{ch:cf; subsec:LODF; eq:BCpartitions.1}:
\begin{equation}
	D = \begin{bmatrix} D_{-\calE, -\calE} & D_{-\calE\calE} \\ D_{\calE, -\calE} & D_{\calE\calE} \end{bmatrix} = 
	\begin{bmatrix} 
	B_{-\calE} C_{-\calE}^T  L^\dag  C_{-\calE} & B_{-\calE} C_{-\calE}^T  L^\dag  C_{\calE} \\ 
	B_{\calE} C_{\calE}^T  L^\dag  C_{-\calE} & B_{\calE} C_{\calE}^T  L^\dag  C_{\calE} \\
	\end{bmatrix}.
	\label{ch:cf; subset:glodf; eq:decomposeD.1}
\end{equation}

We now express the GLODF matrix $K^\calE$ explicitly in terms of the PTDF submatrices introduced in~\eqref{ch:cf; subset:glodf; eq:decomposeD.1}.

\begin{thm} [GLODF $K^\calE$ for non-cut set outage]
\label{ch:cf; thm:LODF.3}
Let $\calE\subsetneq E$ be a non-cut set outage for the power network $G=(V,E, \bm{b})$. 
Then, the matrix $I - D_{\calE\calE} \ = \ I - B_{\calE}  C_{\calE}^T  L^\dag  C_\calE$ is invertible and the net line flow changes $\Delta \ff_{-\calE}$ are given by
\begin{equation}
	\Delta \ff_{-\calE} = K^\calE \, \ff_\calE,
\label{ch:cf; subsec:lodf; eq:defKlF}
\end{equation}
where the GLODF matrix can be calculated as
\begin{equation}
	K^\calE  =  D_{-\calE\calE}\, \left( I - D_{\calE\calE} \right)^{-1} = B_{-\calE}  C_{-\calE}^T  L^\dag  C_\calE \left( I - B_{\calE}  C_{\calE}^T  L^\dag  C_\calE \right)^{-1}.
\label{ch:cf; eq:GLODF.1}
\end{equation}%
\end{thm}
The proof is immediate using~\cite[Theorem 7]{TPS1}, after reformulating the results appearing there in terms of the inverse of reduced Laplacian matrix $A$ using the pseudo-inverse $L^\dag$ using the identities proved in~\cite{Guo2017}. 

This formula of $K^\calE$ is derived by considering the pre-contingency network with changes $\Delta \pp$ in power injections that are judiciously chosen to emulate the effect of simultaneous line outages in $E$. The reference \cite{AlsacStottTinney1983} seems to be the first to introduce the use of matrix inversion lemma to study the impact of network changes on line currents in power systems. This method is also used in \cite{StottAlsacAlvarado1985} to derive the GLODF for ranking contingencies in security analysis. The GLODF has also been derived earlier, e.g., in \cite{EnnsQuadaSackett1982}, and re-derived recently in \cite{GulerGrossLiu2007, GulerGross2007}, without the simplification of the matrix inversion lemma.

\subsubsection{Non-bridge outages}
\label{sub:lodfnonbridge}
We now briefly consider the special case in which the subset $\calE$ of outaged lines is still a non-cut set, but consists of a single line, i.e., $\calE = \{ \hat \ell\}$. Since the singleton $\{\hat \ell\}$ is a cut set if and only if $\hat \ell$ is a bridge, we are thus focusing on \textit{non-bridge line outages}.

The GLODF matrix $K^\calE$ introduced earlier rewrites now as a vector $(K_{\ell\hat \ell}, \ell \in E\setminus\! \{\hat \ell\})$, where, since there is no ambiguity, we suppressed the superscript ${\{\hat \ell\}}$ for compactness. We recover in this way the so-called \textit{line outage distribution factor} (LODF) $K_{\ell\hat \ell}$, which is formally defined to be the change $\Delta f_\ell := \tilde f_\ell - f_\ell$ in power flow on surviving line $\ell \neq \hat \ell$ when a \textit{single non-bridge} line $\hat \ell$ trips, normalized by the pre-contingency line flow $f_{\hat \ell}$:
\[
	K_{\ell\hat \ell}  :=  \frac{\Delta f_\ell}{f_{\hat \ell}}, \qquad \ell \neq \hat \ell \in E, \ \text{for non-bridge line }\hat \ell
\]
assuming that the injections $ \pp$ remain unchanged, cf.~\cite{Wood2013} and \cite{Guo2017}. The LODFs for non-bridge outages inherit all the properties of GLODFs for non-cut sets and, in particular, they are also independent of $\pp$. The next proposition reformulates the results in Theorem~\ref{ch:cf; thm:LODF.3} in the case of a single non-bridge failure, and relates the LODF to the weighted spanning forests of pre-contingency network; see~\cite[Theorem 6]{TPS1} for the proof of this latter fact.
\begin{prop}[LODF $K_{\ell \hat \ell}$ for non-bridge outage]
\label{ch:lm; thm:LODF.1}
Let $\hat \ell = (\hi, \hj)$ be a non-bridge outage that does not disconnect the power network $G=(V,E,\bm{b})$.
For any surviving line $\ell = (i, j) \neq \hat \ell$ the corresponding LODF is given by
\begin{equation}
K_{\ell \hat \ell} = \frac{D_{\ell\hat \ell}} {1 - D_{\hat \ell \hat \ell}} \ \ = \ \ 
	\frac{ b_{\ell} \left(  L^\dag_{i\hi} +  L^\dag_{j\hj} -  L^\dag_{i\hj} -  L^\dag_{j\hi} \right) }{ 1 \ - \ b_{\hat \ell} \left(  L^\dag_{\hi\hi} +  L^\dag_{\hj\hj} - 2 L^\dag_{\hi\hj} \right) }
	  = \frac{ b_{\ell} \left( R_{i\hj}-R_{i\hi}+R_{j\hi}-R_{j\hj}\right)}{2(1-b_{\hat \ell} R_{st})}
\label{ch:lm; eq:LODF.1a}
\end{equation}
and can be equivalently be expressed in terms of the spanning forests of $G$ as
\begin{equation}
	K_{\ell \hat \ell} = \frac{b_{\ell}} {\sum_{\calE\in T\!\setminus\! \{\hat \ell\} }\, \beta(\calE)} \Big( \sum_{\calE\in T(\{i,\hi\}, \{j,\hj\})} \beta(\calE) - \sum_{\calE\in T(\{i,\hj\}, \{j,\hi\})}\, \beta(\calE) \Big).
\label{ch:cf; subsec:lodf; eq:Kllhat.1}
\end{equation}
\end{prop}
Identity~\eqref{ch:lm; eq:LODF.1a} can be equivalently derived using a rank-$1$ update of the pseudo-inverse $L^\dag$ as done by~\cite{Soltan2017b}. Each LODF accounts for all the alternative paths on which the power originally carried by line $\hat \ell$ can flow in the new topology and for how many of these the line $\ell$ belongs to. This dependence is made explicit in~\eqref{ch:cf; subsec:lodf; eq:Kllhat.1} in terms of the weighted spanning forests.

Recall that Proposition \ref{ch:lm; prop:PTDF.2} shows that $D_{\hat \ell \hat \ell} = 1$ if and only if line $\hat \ell$ is a bridge, which corroborates the fact that \eqref{ch:lm; eq:LODF.1a} is valid only for non-bridge failures.
We remark that, differently from the diagonal entries $(D_{\ell \ell}, \ell \in E)$ of the PTDF matrix, the factor $D_{\ell\hat \ell}$ can take any sign and thus so can $K_{\ell\hat \ell}$.
From~\eqref{ch:cf; subsec:lodf; eq:Kllhat.1} we immediately deduce that $K_{\ell\hat \ell}>0$ if $\ell$ and $\hat \ell$ share either the source node ($i = \hi$) or the terminal node ($j=\hj$), since in this case $T(\{i,\hj\}, \{j,\hi\}) = \emptyset$. 

We can also express LODF in \eqref{ch:lm; eq:LODF.1a} in matrix form and relate it to the PTDF matrix $D$ defined in Subsection~\ref{sub:PTDF}. Partition the edge set $E$ in two subsets, $E^\textup{b}$ and $E^\textup{n}$, by distinguishing between bridges and non-bridge lines and set $m_\textup{b}:=|E^\textup{b}|$ and $m_\textup{n}:=|E^\textup{n}|$. Consider the $m \times m_\textup{n}$ matrix $K^\textup{n} := \big( K_{\ell\hat \ell}\, , \, \ell \in E, \hat \ell \in E^\textup{n} \big)$ and $D^\textup{n} := \big( D_{\ell\hat \ell}\, , \, \ell \in E, \hat \ell \in E^\textup{n} \big)$ and rewrite them into two submatrices, possibly after permutations of rows and columns:
\[
	K^\textup{n} \ \ = \ \ \begin{bmatrix} K_\textup{bn} \\ K_\textup{nn} \end{bmatrix},  \qquad 
	D^\textup{n} \ \ = \ \ \begin{bmatrix} D_\textup{bn} \\ D_\textup{nn}  \end{bmatrix}
\]
%
Note we implicitly extended the LODF definition to the case in which $\ell=\hat\ell \in E^\textup{n}$ as follows: the diagonal entries of the submatrix $K_\textup{nn}$ are defined as
\[
	K^\textup{n}_{\hat \ell \hat \ell} := \frac{D_{\hat \ell\hat \ell} }{ 1 - D_{\hat \ell \hat \ell} } \qquad \text{ for non-bridge line }\hat \ell \in E^\textup{n}
\]
and represents the change in power flow on line $\hat \ell$ if $\hat \ell = (\hi, \hj)$ is \emph{not} outaged but injections $p_{\hi}$ and $p_{\hj}$ are changed by one unit. 

%
To express the LODF \eqref{ch:cf; subsec:lodf; eq:Kllhat.1} in matrix form in terms of $L^\dag$, partition the matrices $B$ and $C$ into two submatrices corresponding to the non-bridge lines and bridge lines (possibly after permutations of rows and columns), i.e., $B =: \ \textup{diag} \left(B_\textup{b}, \, B_\textup{n} \right)$ and $C \ =: \ [  C_\textup{b} \ \  C_\textup{n} ]$. Let $\textup{diag} (D_\textup{nn})$ be the $m_\textup{n}\times m_\textup{n}$ diagonal matrix whose diagonal  entries are $D_{\hat \ell \hat \ell}$ for non-bridge lines $\hat \ell$. 
\begin{cor}
\label{ch:cf; corollary:LODF.2}
With the submatrices defined above we have $D^\textup{n} \ = \ B  C^T L^\dag  C_\textup{n}$ and
\[
	K^\textup{n} \ = \ D^\textup{n} \left( I - \mathrm{diag}(D_\textup{nn}) \right)^{-1} \ = \ B  C^T  L^\dag  C_\textup{n} \left( I - \mathrm{diag} (B_\textup{n}  C_\textup{n}^T  L^\dag  C_\textup{n} ) \right)^{-1}.
\]
\end{cor}
Given a non-cut set $\calE$ of lines, let $-\calE := E\!\setminus\! \calE$. Consider the submatrix $K_{-\calE\calE} := \big( K_{\ell\hat \ell}, \, \ell\in -\calE, \hat \ell\in \calE \big)$ of $K^\textup{n}$ and the submatrices $D_{-\calE\calE} := \big( D_{\ell\hat \ell}, \, \ell\in -\calE, \hat \ell\in \calE \big)$ and  $D_{\calE\calE} := \big( D_{\hat \ell\hat \ell}, \, \hat \ell\in \calE \big)$ of $D^\textup{n}$. Corollary \ref{ch:cf; corollary:LODF.2} allows us to rewrite the line outage distribution factors in~\eqref{ch:lm; eq:LODF.1a} in matrix form as
\begin{equation}
	K_{-\calE\calE}  =  D_{-\calE\calE} \left( I - \textup{diag} (D_{\calE\calE}) \right)^{-1}.
\label{ch:cf; sec:lodf; eq:K_{-FF}D_{-FF}}
\end{equation}
Equation~\eqref{ch:cf; sec:lodf; eq:K_{-FF}D_{-FF}}, despite looking similar to the formula~\eqref{ch:cf; eq:GLODF.1} for the GLODF matrix $K^\calE$ in Theorem~\ref{ch:cf; thm:LODF.3}, is fundamentally different since only the diagonal elements of $D_{\calE\calE}$ are used, while the full matrix $D_{\calE\calE}$ appears in~\eqref{ch:cf; eq:GLODF.1}.

It is important to stress that the LODF submatrix $K_{-\calE\calE}$ is \textit{not} the GLODF matrix $K^\calE$. Indeed the matrix $K_{-\calE\calE}$ is a compact way to write all the LODF factors $K_{\ell \hat \ell}$ corresponding to the \textit{individual} line outages $\hat \ell \in \calE$ and does \emph{not} give the sensitivity of power flows in the post-contingency network to a set $\calE$ of \textit{simultaneous} line outages. For this reason, each column of $K_{-\calE\calE}$ must be interpreted \emph{separately}: column $\hat \ell$ gives the power flow changes on each surviving line $\ell\in -\calE$ due to the outage of a \emph{single} non-bridge line $\hat \ell \in \calE$. Nonetheless, combining~\eqref{ch:cf; eq:GLODF.1} and~\eqref{ch:cf; sec:lodf; eq:K_{-FF}D_{-FF}}, we can show that in the case of a non-cut set $\calE$, the GLODF matrix $K^\calE$ and the LODF submatrix $K_{-\calE\calE}$ are still related as follows:
\begin{equation}
	K^\calE = K_{-\calE\calE} \left( I - \mathrm{diag}(D_{\calE\calE}) \right) \left( I - D_{\calE\calE} \right)^{-1}.
\label{ch:cf; eq:GLODF.1c}
\end{equation}
This expression shows that if $\calE = \{\hat \ell\}$ is a singleton, then the two matrices coincides, but in general $K^\calE \neq K_{-\calE\calE}$. 

\subsection{Islanding and GLODF for cut sets}
\label{sub:cut}

In this subsection we study the impact of cut set outages, extending the results of the previous subsection to the scenario where a set of lines trip simultaneously disconnecting the network into two or more connected components or islands. We first focus on the case of a general cut set outage and then later, in Subsection~\ref{sub:bridge}, specify our results in the case in which the cut set consist of a single line, i.e., the case of a bridge outage.

Consider a subset of lines $\calE \subset E$ that is a cut set for $G=(V,E)$, which means that the surviving network $G^\calE:=(V,E\bs\calE)$ consists of two or more islands $\calI_1,\dots, \calI_k$. A \textit{power balancing rule} needs to be invoked to rebalance power on each island after the contingency. We assume such operations on the islands are decoupled, so that we can study the power flow redistribution in each of them separately. Therefore, for the purpose of presentation, let us focus on one of these islands, say $\calI$. 

We can distinguish three types of failed lines and partition $\calE = \calE_\calI \cup \calE_{\partial \calI} \cup \calE_{\calI^c}$ accordingly as follows:
\begin{itemize}
	\item $\calE_\textrm{int} = \calE_\calI:= \set{(i,j)\in \calE ~:~ i\in\calI \text{ and } j\in\calI }$ is the subset of failed \textit{internal lines}, i.e., those both endpoints of which belong to $\calI$;
	\item $\calE_\textrm{tie} = \calE_{\partial \calI}:=\set{(i,j)\in \calE ~:~ i\notin\calI\text{ and }j\in\calI }$ is the subset of failed \textit{tie lines}, i.e., those that have exactly one endpoint in $\calI$; 
	\item $\calE_{\calI^c}$ is the subset of failed lines neither endpoints of which belongs to $\calI$.
\end{itemize}
Any outaged line in $\calE_{\calI^c}$ does not have a direct impact on the island $\calI$ and can thus be ignored and thus we can assume without loss of generality that $\calE = \calE_\textup{tie} \cup \calE_\textup{int}$. 

A main difference between a bridge outage and a cut-set outage is that in the former case $\calE = \calE_\textup{tie}$ consists of a single tie line, whereas in the latter case $\calE$ may contain internal lines in $\calE_\textup{int}$ as well. Since $\calI$ is connected, $\calE_\textup{int}$ is a non-cut set. The power flows on surviving lines in $\calI$ are impacted by both internal line outages in $\calE_\textup{int}$ and by tie line outages in $\calE_\textup{tie}$. 

Denote by $\pp_I := (p_i, i \in \calI)$ the pre-contingency injections in $\calI$. The power injections $\pp_I$ may not be balanced inside the island, as there could be a non-zero total power imbalance 
\[
	\im := \bm{1}^T \pp_\calI = \sum_{i \in \calI} p_i  = -\sum_{\hat \ell\in \calE_\textup{tie}} f_{\hat \ell}.%
\]
If $\im < 0$, then it means that pre-contingecy the island was importing power, while in the opposite case, the island was exporting power. In either case, if $\im \neq 0$, the network operator has to intervene with a \textit{balancing rule} that, by adjusting injections (generators and/or loads) in some of the nodes, ensures that the power balance is restored in $\calI$.

A popular balancing rule, named \emph{proportional control}, is to share the imbalance among a set of participating nodes with a fixed proportion. Specifically a proportional control $\mathbb G_{\a}$ is defined by a nonnegative vector $\a := (\alpha_i, i \in \calI)$ such that $\sum_{i \in \calI}\, \alpha_i = 1$, with the interpretation that, post contingency, each node $i \in \calI$ adjusts its injection by the amount $\Delta p_i := -\alpha_i \im$. The injection adjustments vector is then
\begin{equation}
	\Delta \pp_{\a} := ( \Delta p_i, i \in \calI) = -\im \sum_{i \in \calI} \alpha_i \ee_i,
\label{ch:cf; subsec:bridgeoutage; eq:Galpha.1}
\end{equation}
where $\ee_i$ is the standard unit vector of size $n_\calI$. We refer to each node $i \in \calI$ such that $\alpha_i>0$ as \textit{participating node}. All participating nodes adjust their injections in the same direction, whose sign depends on the sign of $\im$.
The injections are therefore changed from $\pp$ to a post-contingency vector $ \pp + \Delta \pp_\alpha$. The proportional control rebalances the power on $\calI$: the assumption $\sum_{i \in \calI}\, \alpha_i = 1$ guarantees that the post-contingency injections are zero-sum, as
\[
	\bm{1}^T ( \pp + \Delta \pp_\alpha) =\bm{1}^T ( \pp - \im \sum_{i \in \calI} \alpha_i \ee_i ) = \im -\im = 0.
\]

We now derive the impact of the cut set $\calE$ outage on the surviving lines in island $\calI$.
Let $\ff$ be the pre-contingecy power flows on the lines in $\calI$ and further partition $\ff = (\ff_{-\textup{int}}, \ff_\textup{int})$, where $\ff_\textup{int} := (f_\ell, \ell\in \calE_\textup{int})$ denote the pre-contingency power flows on the outaged internal lines, $\ff_{-\textup{int}} := (f_\ell, \ell\in E_\calI \!\setminus\! \calE_\textup{int})$ those on the surviving lines in $\calI$.
Let $B, C, L^\dagger$ denote the susceptance matrix, the incidence matrix and the graph Laplacian pseudo-inverse associated with the island $\calI$. We also partition the matrices $B, C$ according to surviving lines and outaged internal lines as $B =:  \textup{diag} \left(B_{-\textup{int}}, B_\textup{int} \right)$ and $C=:\begin{bmatrix} C_{-\textup{int}}, & C_\textup{int} \end{bmatrix}$.

When multiple lines are tripped from the grid simultaneously disconnecting the network, the aggregate impact in general is different from the aggregate effect of tripping the lines separately. Nonetheless, as the following result shows, the impact from balancing rule turns out to be separable from the rest.

\begin{thm}[GLODF for cut set with proportional control $G_{\a}$]
The net change on the surviving lines of the island $\calI$ due to a cut set outage $\calE=\calE_\textup{tie} \cup \calE_\textup{int}$ is given by
\begin{equation}
\Delta \ff_{-\textup{int}} =  K^{\calE_\textup{int}} \, \ff_{\textup{int}} +  \sum_{\hat \ell\in \calE_\textup{tie}} f_{\hat \ell}  \sum_{k: \alpha_k>0} \alpha_k\, \left( D_{-\textup{int}, k j(\hat \ell)} + K^{\calE_\textup{int}}\, D_{\textup{int}, k j(\hat \ell)} \right),
\label{ch:cf; subsec:bridgeoutage; eq:DeltaP.1}
\end{equation}
where $j(\hat \ell)$ is the endpoint of tie line $\hat \ell$ that lies in $\calI$.
\end{thm}
For a proof, the reader can look at~\cite[Theorem 5]{TPS2} and eq.~(12) therein. The RHS of~\eqref{ch:cf; subsec:bridgeoutage; eq:DeltaP.1} consists of two terms:
\[
\Delta \ff_{-\textup{int}} =  \underbrace{ K^{\calE_\textup{int}} \, \ff_{\textup{int}} }_{ \textup{int. line $\calE_\textup{int}$ outage} }
	\ \ + \ \ 	
	\underbrace{ \sum_{\hat \ell\in \calE_\textrm{tie}} f_{\hat \ell}  \sum_{k: \alpha_k>0} \alpha_k\, \left( D_{-\textrm{int}, k j(\hat \ell)} + K^{\calE_\textrm{int}}\, D_{\textrm{int}, k j(\hat \ell)} \right)}_{ \text{tie line $\calE_\textrm{tie}$ outage} }.
\]
The first term on the RHS accounts for the impact of the outage of the non-cut set $\calE_\textup{int}$ of internal lines in $\calI$ through the GLODF $K^{\calE_\textup{int}} := B_{-\textup{int}} \, C_{-\textup{int}}^T \, L \, C_{\textup{int}} (I - B_{\textup{int}} \, C_{\textup{int}}^T \, L^\dagger \, C_{\textup{int}})^{-1}$ (cf.~Theorem \ref{ch:cf; thm:LODF.3}). The second term captures the impact of the proportional control $\mathbb G_\alpha$ in response to tie line outages in $\calE_\textup{tie}$, which also accounts for the mimicked power injection change due to the outaged tie lines $\hat \ell \in \calE_\textup{tie}$ at their endpoints $j(\hat \ell)$ in $\calI$. We stress that the impact of the proportional control $\mathbb G_\alpha$ through internal lines in $\calE_\textup{int}$ is scaled by the GLODF $K^{\calE_\textup{int}}$.

Note that if there are no tie line outages, i.e., $\calE_\textup{tie} = \emptyset$, then \eqref{ch:cf; subsec:bridgeoutage; eq:DeltaP.1} reduces to Theorem \ref{ch:cf; thm:LODF.3} for a non-cut set outage. 
If there are no internal lines $\calE_\textup{int} = \emptyset$, the first term of \eqref{ch:cf; subsec:bridgeoutage; eq:DeltaP.1} is zero and the second term simplifies, so that 
\[
\Delta \ff_{-\textup{int}}  = \sum_{\hat \ell\in \calE_\textup{tie}} f_{\hat \ell}  \sum_{k: \alpha_k>0} \alpha_k\, D_{-\textup{int}, k j(\hat \ell)}.
\]

\subsubsection{Bridge outage}
\label{sub:bridge}
Consider now the case of a cut set outage, in which the $\calE$ consist of a single line. This means that $\calE_\textup{int} = \emptyset$, $\calE_\textup{tie}=\{\hat \ell\}$, and $\hat \ell = (\hat i, \hat j)$ is a bridge. After the outage of line $\ell$, the power network gets disconnected into two islands, one containing node $\hat i$ and one containing node $\hat j$, which we denote by $\calI_{\hat i}$ and $\calI_{\hat j}$ respectively. 

If the pre-contingency power flow $f_{\hat \ell} = 0$ on the outaged line $\hat \ell$, then the operations of these islands, as modeled by the DC power flow equations, are decoupled pre-contingency and therefore remain unchanged post-contingency.
We henceforth consider only the case where $f_{\hat \ell}\neq 0$, which means that there is a nonzero power imbalance over each of the two islands post contingency. For example, if $f_{\hat \ell}>0$, i.e., the power flows from island $\calI_{\hat i}$ to island $\calI_{\hat j}$ pre contingency, then there is a power surplus of size $f_{\hat \ell}$ on $\calI_{\hat i}$ and a shortage of size $f_{\hat \ell}$ on $\calI_{\hat j}$ post contingency.
Like before, since the two islands operate independently post contingency, each according to DC power flow equations, we can focus only on one island, say $\calI_{\hat j}$. Let $\Delta \ff := (\Delta f_\ell, \ell\in E_{\calI_{\hat j}})$ be the vector of net power flows changes on all lines in island $\calI_{\hat j}$. Then, \eqref{ch:cf; subsec:bridgeoutage; eq:DeltaP.1} simplifies to
\[
	\Delta \ff = f_{\hat \ell}  \sum_{k\in \calI_{\hat j}} \alpha_k\, \Big( D_{E_{\calI_{\hat j}}, k \hat j} \Big).
\]
More explicitly, for any line $\ell = (i,j) \in E_{\calI_{\hat j}}$ we have
\[
\frac{\Delta f_\ell} {f_{\hat \ell}}
	=  \sum_{k: \alpha_k>0} \alpha_k \, D_{\ell, k\hat j}
	=  \sum_{k\in \calI_{\hat j}} \alpha_k \, b_\ell \left( L^\dagger_{ik} + L^\dagger_{j\hat j} - L^\dagger_{i\hat j} - L^\dagger_{jk} \right),
\]
where the last equality follows from Theorem \ref{ch:lm; prop:PTDF.1}. Therefore the power flow change on line $\ell$ is the superposition of two factors: (i) injecting $\alpha_k f_{\hat \ell}$ at participating nodes $k\in \calI_{\hat j}$ and (ii) withdrawing $f_{\hat \ell}$ at node $\hat j$.
Using the last identity, we can thus extend the definition of LODF $K_{\ell\hat \ell}$ to allow for outages of any bridge $\hat \ell = (\hat i, \hat j)$ with nonzero flow $f_{\hat \ell}\neq 0$ under the proportional control $\mathbb G_{\a}$ as follows
\begin{align}
K_{\ell\hat \ell} :=  \frac{ \Delta f_\ell } { f_{\hat \ell} } = \sum_{k ~:~ \alpha_k>0}  \alpha_k\, D_{\ell, k\hat j},\qquad \text{for all } \ell\in E_{\calI_{\hat j}}.
\label{ch:cf; subsec:bridgeoutage; eq:Klhatl.1}
\end{align}

\FloatBarrier

\section{Localization results}
\label{sec:localization}

In this section we show how the block and bridge-block decompositions of the power network $G$ fully captures the power redistribution effects after a contingency and give a very transparent view of the global network robustness against failures. More specifically, we characterize when LODF and GLODF are nonzero in terms of these network structures. The highlight is that an outage line has only a local impact which is confined in the (bridge-)block to which it belongs. We present the results for non-cut set and cut set outages separately in the next two subsections.

\subsection{Non-cut set outages}
The goal of this subsection is to show that the block decomposition of the network $G$ induces a block-diagonal structure for the PTDF and LODF matrices. The key idea is using the representation of these factors in terms of spanning forest and the intimate relation between the block decomposition and the cycles in the graph $G$. Lastly, we show that when the considered outage $\calE$ is a non-cut set, the GLODF matrix $K^\calE$ inherits the block-diagonal structure of the PTDF matrix.

The next lemma states the \textit{Simple Cycle Criterion}, which has been proved in~\cite[Theorem 8]{TPS1}. It provides a sufficient condition for the LODF to be zero expressed purely in topological terms, namely using the cycles of the power network $G$.
\begin{lem}[Simple Cycle Criterion] \label{lem:simple_cycle}
Consider a non-bridge line $\hat \ell \in E$ and any other line $\ell \in E$. If there is no simple cycle in $G$ that contains both $\ell$ and $\hat \ell$, then the LODF $K_{\ell \hat \ell}= 0$. 
\end{lem}
Used in combination with the network block decomposition introduced in Section~\ref{sub:basic}, this criterion can be used to show that the PTDFs and LODFs are always zero when they refer to two lines that belong to different blocks of the power network $G$.
\begin{prop}[LODF and PTDF characterization for non-bridge single line failure]
\label{prop:zeroLODFPTDF}
Consider a single non-bridge line $\hat \ell = (\hat i, \hat j)$ outage. 
For any surviving line $\ell = (i, j) \neq \hat\ell$ the following statements hold:
\begin{itemize}
\item[(i)] The PTDF $D_{\ell \hat\ell} = 0$ if and only if the LODF $K_{\ell \hat\ell} = 0$.
\item[(ii)] If $\ell$ and $\hat \ell$ are in different blocks of $G$, then $K_{\ell \hat\ell} = 0$ and $D_{\ell \hat\ell} = 0$.
\end{itemize}
\end{prop}

To formally state the next results for the distribution factor matrices when $\calE$ is a non-cut set outage, we first need some preliminary definitions to partition outage lines in $\calE$ and surviving lines in $-\calE$ depending on which block they belong to. More specifically, let $E = E_1 \cup \cdots \cup E_k$ be the block decomposition of (the edge set $E$ of) the network $G$. We partition accordingly into $k$ disjoint subsets both the outaged line set $\calE$ and the surviving lines set $-\calE$ as follows:
\begin{subequations}
\begin{align}
    \calE &= \bigcup_{j} \calE_j \quad \text{ with } \quad \calE_j := \calE\cap E_j, \qquad j=1, \dots, k,\\
    -\calE &= \bigcup_{j} \calE_{-j}  \quad \text{ with } \quad \calE_{-j} := -\calE\cap E_j = E_j \!\setminus\! \calE_j, \qquad j=1, \dots, k.
\end{align}
\label{eq:-EEdecomposition}
\end{subequations}
Without loss of generality we assume that the lines in $\calE$ are indexed so that the outaged lines are ordered depending on which subset $\calE_j$ they belong to. Similarly, the surviving lines in $-\calE$ are indexed so that the surviving lines are ordered depending on which subset $-\calE_j$ they belong to. In view of the partition in~\eqref{eq:-EEdecomposition}, the submatrices $B_{-\calE}$, $B_\calE$, $C_{-\calE}$ and $C_\calE$ appearing in~\eqref{ch:cf; subsec:LODF; eq:BCpartitions.1} can be further decomposed as
\begin{align}
B  &=  \begin{bmatrix} B_{-\calE} & 0 \\ 0 & B_\calE \end{bmatrix} \ \ =: \ \ 
	\begin{bmatrix}  \text{diag} \left( B_{-j}, j = 1, \dots, k \right) & 0 \\ 
		0 &  \text{diag} \left( B_{j}, j = 1, \dots, k \right) \end{bmatrix},
\\
C  &=  \begin{bmatrix} C_{-\calE} & C_{\calE} \end{bmatrix}  \ \ =: \ \ 
	\begin{bmatrix} C_{-1} & \cdots & C_{-k} & \ & C_{1} & \cdots & C_{k} \end{bmatrix}.
\end{align}
In particular, $C_{j}$ and $C_{-j}$ are the incidence matrices for the lines $\calE_j$ and $\calE_{-j}$, respectively.

\begin{prop}[Block diagonal structure of LODF and PTDF submatrices]
\label{prop:blockDF}
Let $\calE$ be non-cut set outage for the power network $G=(V,E,\bm{b})$. Then the two submatrices $D_{-\calE\calE}$ and $D_{\calE\calE}$ of the PTDF matrix defined in~\eqref{ch:cf; subset:glodf; eq:decomposeD.1} decompose as follows:
\begin{subequations}
\begin{equation}
    D_{-\calE\calE} =:  \begin{bmatrix} 
	D_{-1} & 0 & \dots & 0  \\  0 & D_{-2} & \dots & 0 \\ \vdots & \vdots & \ddots & \vdots \\ 0 & 0 & \dots & D_{-k} 
	\end{bmatrix}
	\qquad \text{ and } \qquad
	D_{\calE\calE}  =:  \begin{bmatrix} 
	D_1 & 0 & \dots & 0  \\  0 & D_2 & \dots & 0 \\ \vdots & \vdots & \ddots & \vdots \\ 0 & 0 & \dots & D_k 
	\end{bmatrix}
\label{ch:cf; sec:localization; eq:K-FF.1a}
\end{equation}
where, for each $j = 1, \dots, k$, $D_{-j}$ is the $|\calE_{-j}|\times |\calE_j|$ matrix given by $D_{-j}  := B_{-j} C_{-j}^T L^\dag C_j$, and $D_{j}$ is the $|\calE_{j}|\times |\calE_j|$ matrix given by $D_{-j}  := B_{j} C_{j}^T L^\dag C_j$. 

Furthermore, also the LODF matrix $K_{-\calE\calE}$ defined as in~\eqref{ch:cf; sec:lodf; eq:K_{-FF}D_{-FF}} has a block-diagonal structure, namely
\begin{equation}
K_{-\calE\calE}   =:  \begin{bmatrix} 
	K_1 & 0 & \dots & 0  \\  0 & K_2 & \dots & 0 \\ \vdots & \vdots & \ddots & \vdots \\ 0 & 0 & \dots & K_k 
	\end{bmatrix},
\label{ch:cf; sec:localization; eq:K-FF.1c}
\end{equation}
where $K_j$ is the $|\calE_{-j}|\times |\calE_j|$ matrix given by $K_{j} :=  D_{-j} \left( I - \mathrm{diag}(D_j) \right)^{-1}$, for $j=1,\dots,k$.
\label{ch:cf; sec:localization; eq:K-FF.1}
\end{subequations}
\end{prop}
\begin{proof}
Recall from \eqref{ch:cf; subset:glodf; eq:decomposeD.1} that $D_{-\calE\calE} = B_{-\calE}C_{-\calE}^T L^\dag C_\calE$ and $D_{\calE\calE} = B_\calE C_\calE^T L^\dag C_\calE$. Proposition~\ref{prop:zeroLODFPTDF} then implies that the PTDF submatrices $D_{-\calE\calE}$ and $D_{\calE\calE}$ decompose into a block-diagonal structure corresponding to the blocks of $G$, obtaining~\eqref{ch:cf; sec:localization; eq:K-FF.1a}.
Since $K_{-\calE\calE} = D_{-\calE\calE} ( I - \textrm{diag}(D_{\calE\calE}) )^{-1}$ in view of identity~\eqref{ch:cf; sec:lodf; eq:K_{-FF}D_{-FF}}, the LODF submatrix $K_{-\calE\calE}$ has the same block diagonal structure as $D_{-\calE\calE}$. The invertibility of $I - \mathrm{diag}(D_j) = \mathrm{diag}(1-D_{\hat\ell \hat\ell}, \ell \in \calE_j)$ follows from Proposition~\ref{ch:lm; prop:PTDF.2}(ii) and Corollary~\ref{ch:lm; corollary:PTDF.3}, since every line $\hat\ell \in \calE_j \subset \calE$ is not a bridge and thus $D_{\ell\ell}<1$.
\end{proof}

The block decomposition of the network $G$ is crucial for failure localization also when a non-cut set $\calE$ of lines are disconnected simultaneously. More specifically, the next result shows that even in this scenario the impacts of line outages are still contained in their own blocks.
\begin{prop}[GLODF characterization for non-cut set outage]
\label{prop:glodf_non_cut}
Let $\calE$ be non-cut set outage for the power network $G=(V,E,\bm{b})$. Then, the rectangular $|-\calE| \times |\calE|$ GLODF matrix $K^\calE$ has a block-diagonal structure, i.e.,
\begin{subequations}
\begin{equation}
K^\calE  =:  \begin{bmatrix} 
	K^\calE_1 & 0 & \dots & 0  \\  0 & K^\calE_2 & \dots & 0 \\ \vdots & \vdots & \ddots & \vdots \\ 0 & 0 & \dots & K^\calE_k
	\end{bmatrix},
\label{ch:cf; sec:localization; eq:K^E.1a}
\end{equation}
where each $K^\calE_j$ is $|\calE_{-j}|\times |\calE_j|$ and involves lines only in block $E_j$ of $G$, given by:
\begin{equation}
K^\calE_j  :=  B_{-j} C_{-j}^T L^\dag C_j \left( I-B_j C_j^T L^\dag C_j \right)^{-1}, \qquad j = 1, \dots, k,
\label{ch:cf; sec:localization; eq:K^E.1c}
\end{equation}
where the submatrices $B_j, B_{-j}, C_j$ and $C_{-j}$ have been obtained by partitioning the matrices $B$ and $C$ according to the block structures of both $-\calE, \calE$. Equivalently, we have
\begin{equation}
K^\calE_j  =  D_{-j} (I - D_j)^{-1} = K_{j} \left(I - \textup{diag}(D_j) \right) \left( I-D_j \right)^{-1}, \qquad j = 1, \dots, k.
\label{ch:cf; sec:localization; eq:K^E.1b}%
\end{equation}%
\label{ch:cf; sec:localization; eq:K^E.1}%
\end{subequations}%
where $D_j, D_{-j}, K_j$, $j=1,\dots,k$, are the submatrices defined in Proposition~\ref{prop:blockDF}.
Furthermore, for every $\ell \in -\calE$ and $\hat\ell \in \calE$, if $\ell$ and $\hat\ell$ are in different blocks of $G$, then the GLODF $K^\calE_{\ell\hat \ell} =0$.
\end{prop}
\begin{proof}
The proof of~\eqref{ch:cf; sec:localization; eq:K^E.1} readily follows from Theorem~\ref{ch:cf; thm:LODF.3} after rewriting all the various submatrices resulting from the block decomposition and using Proposition~\ref{prop:blockDF}. The last statement is an immediate consequence of~\eqref{ch:cf; sec:localization; eq:K^E.1}.
\end{proof}


We remark that, differently for the case of a single non-bridge failure in Proposition~\ref{prop:zeroLODFPTDF}, the PTDFs and GLODFs are not always simultaneously zero. Specifically, when multiple lines fail simultaneously, the fact that GLODF $K_{\ell \hat\ell}^\calE = 0$ does not imply that the PTDF $D_{\ell\hat\ell}=0$ or vice versa. Consider the counterexample in Fig.~\ref{fig:counterexample} where two dashed lines $\ell_1$ and $\ell_2$ fail simultaneously, i.e., $\calE = \{\ell_1, \ell_2\}$. Assuming all the lines have unit susceptances, the resulting absolute values of PTDF and GLODF are computed as in Table.~\ref{table:counterexample}. In particular, $K_{\ell_6,\ell_1}^\calE=K_{\ell_3,\ell_2}^\calE = K_{\ell_4,\ell_2}^\calE = 0 $, while the corresponding PTDFs are non-zero. We remark that such counterexamples may happen when the outage set $\calE$ is such that the block decomposition of the surviving network has more blocks than the original network (in our example, they are $4$ and $1$, respectively). 

\begin{table}[!ht]
  \begin{minipage}[b]{0.5\textwidth}
    \centering
    \includegraphics[width=0.6\textwidth]{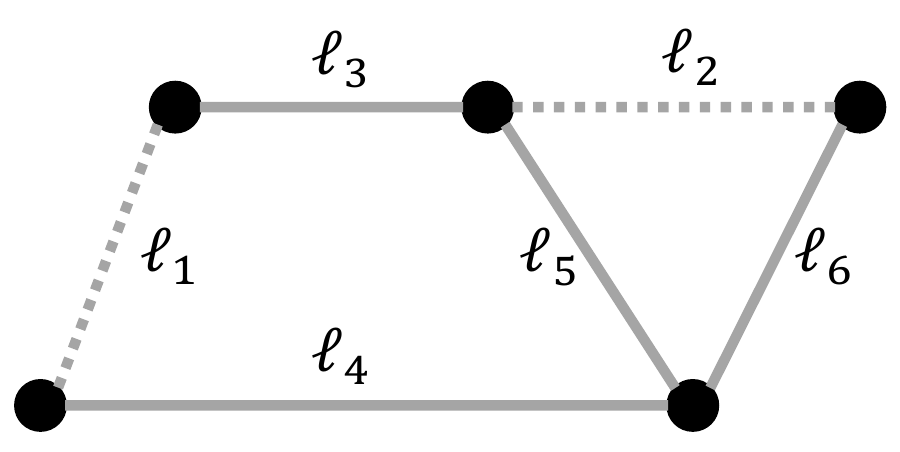}
    \captionof{figure}{Counterexample with the lines in the non-cut\\outage set $\calE = \{\ell_1, \ell_2\}$ displayed as dashed.}
    \label{fig:counterexample}
  \end{minipage}
  \hfill
  \begin{minipage}[b]{0.5\linewidth}
    \centering
    \begin{tabular}{c|c|c|c|c}
      \toprule
      \diagbox{$\hat\ell$}{$\ell$}& $\ell_3$ & $\ell_4$ & $\ell_5$ & $\ell_6$ \\
      \hline \hline
      $\ell_1$ & 3/11, 1 & 3/11, 1 & 2/11, 1& 1/11, 0 \\
      $\ell_2$ & 1/11, 0 & 1/11, 0 & 3/11, 1 & 4/11, 1 \\
      \bottomrule
    \end{tabular}
    \caption{The distribution factors for all the lines reported as pairs (PTDF $D_{\ell,\hat\ell}$, GLODF $K^\calE_{\ell,\hat\ell}$).}
    \label{table:counterexample}
  \end{minipage}
\end{table}
\FloatBarrier

Nonetheless, the following proposition lists sufficient conditions for both implications to hold.
\begin{prop}[Zero GLODF and PTDF characterization]
\label{prop:glodf_ptdf}
Consider a non-cut set outage $\calE$ and a surviving line $\ell=(i,j)\notin \calE$. The following statements hold:
\begin{itemize}
    \item[(i)] If the PTDF $D_{\ell \hat\ell} = 0$ holds for any set of susceptances $B$, then the GLODF $K_{\ell\hat\ell}^\calE = 0$.
    \item[(ii)] If the GLODF $K_{\ell\hat \ell}^\calE=0$, then either the PTDF $D_{\ell\hat\ell}=0$ or every cycle in pre-contingency network $G$ that containing $\ell$ and $\hat \ell$ pass through at least one outage line different from $\hat \ell$.
\end{itemize}
\end{prop}


\subsubsection{Visualizing LODFs using influence graphs}
\label{sub:influence}

The analysis of the LODFs corroborates a well-known fact in power systems, namely how after a particular line outage the next line to fail may be very distant topologically and geographically, as illustrated by~\cite{Dobson2016}. Classical contagion model in which failures spread locally from a node to its immediate neighbors are thus not appropriate to describe cascading failures on power networks (also due to the fact that line failures are an order of magnitude more likely than node outages).

A better way to capture the interdependencies between (single) line failures in power networks is by means of a specific type of ``dual graph'', usually referred to as~\textit{(failure) influence graph}~\cite{Hines2013,Hines2017} or \textit{interaction graph}~\cite{Nakarmi2018,Nakarmi2019}. There are many possible ways to define such a graph, but the high-level idea is that an influence graph $\mathcal \calL(G)$ is derived from the original power network $G$ as follows: the vertex set of $\mathcal \calL(G)$ is in one-to-one correspondence with the set $E$ of the lines of the power network $G$ and the pair of vertices corresponding to the lines $\ell, \hat\ell$ are connected in $\mathcal \calL(G)$ if and only if the failure of line $\hat\ell$ has a severe impact and/or causes and/or is correlated the failure of line $\ell$. 
The edges (or the edge weights) of the influence graph $\mathcal \calL(G)$ can thus be inferred by historical data or by simulating cascades. In this paper, we consider a version of the influence graph defined using LODFs very similar to the one in~\cite{Panciatici2018}, which offers a powerful way to visualize the potential risk of failure propagation. More specifically, the vertices corresponding to the lines $\ell, \hat\ell$ are connected in $\mathcal \calL(G)$ if and only if LODF is larger than a desired threshold, i.e., $|K_{\ell \hat\ell}| > K_\mathrm{min}$. 

\begin{figure}[!ht]
\centering
\subfloat[][The influence graph (in gray) obtained using the threshold $K_\mathrm{min}=0.005$. Due to the presence of a block, the influence graph has two clearly visible disconnected components, since all the corresponding LODFs are zero.]{\makebox[\width]{\includegraphics[width=.95\textwidth]{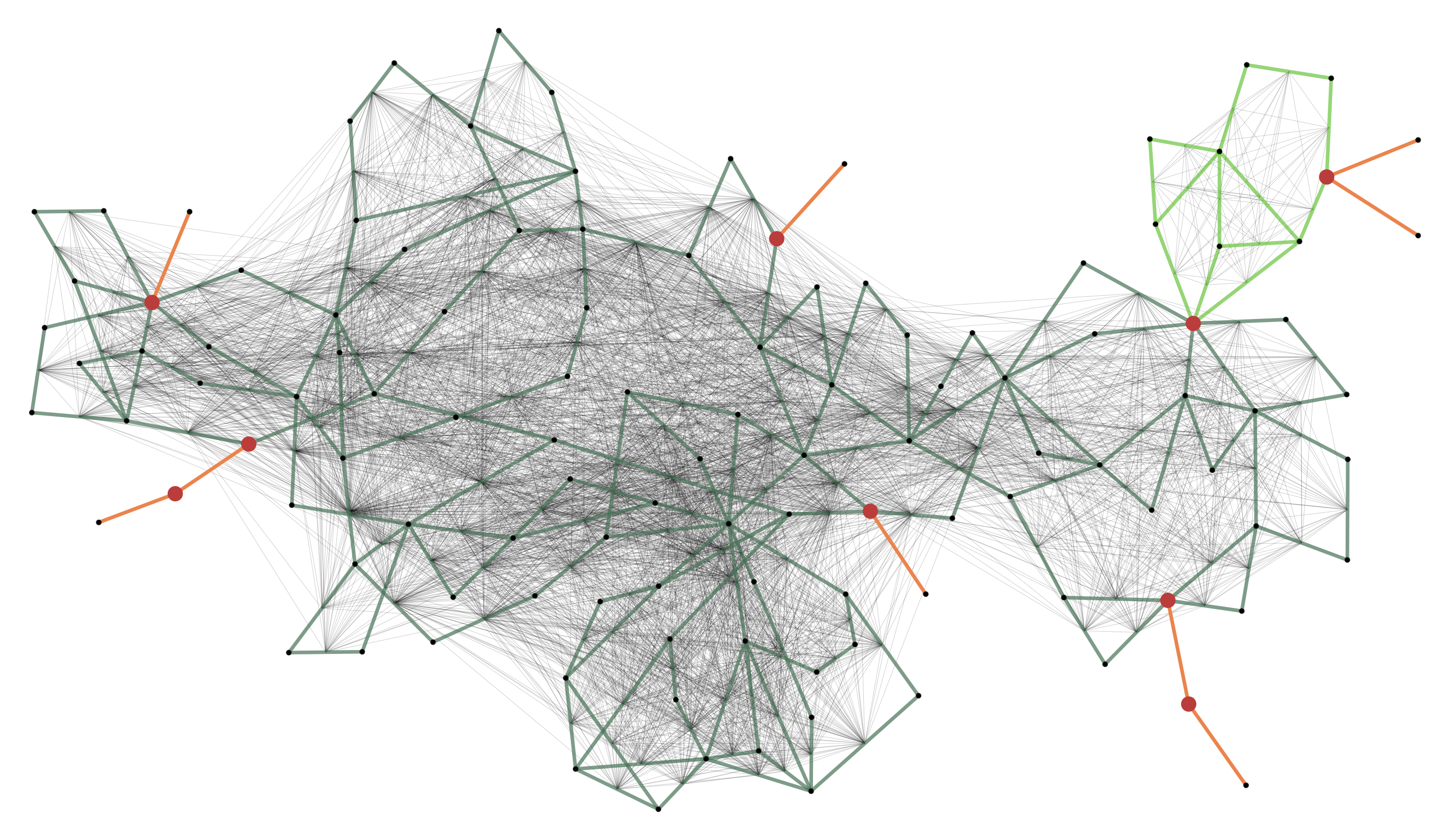}}}\\
\subfloat[][Schematic representation of the block and bridge-block decompositions of the original IEEE 118-node test network. The nontrivial bridge-block is further decomposed by a cut vertex into two blocks (in two different shades of green)]{\makebox[1.9\width]{\includegraphics[width=.5\textwidth]{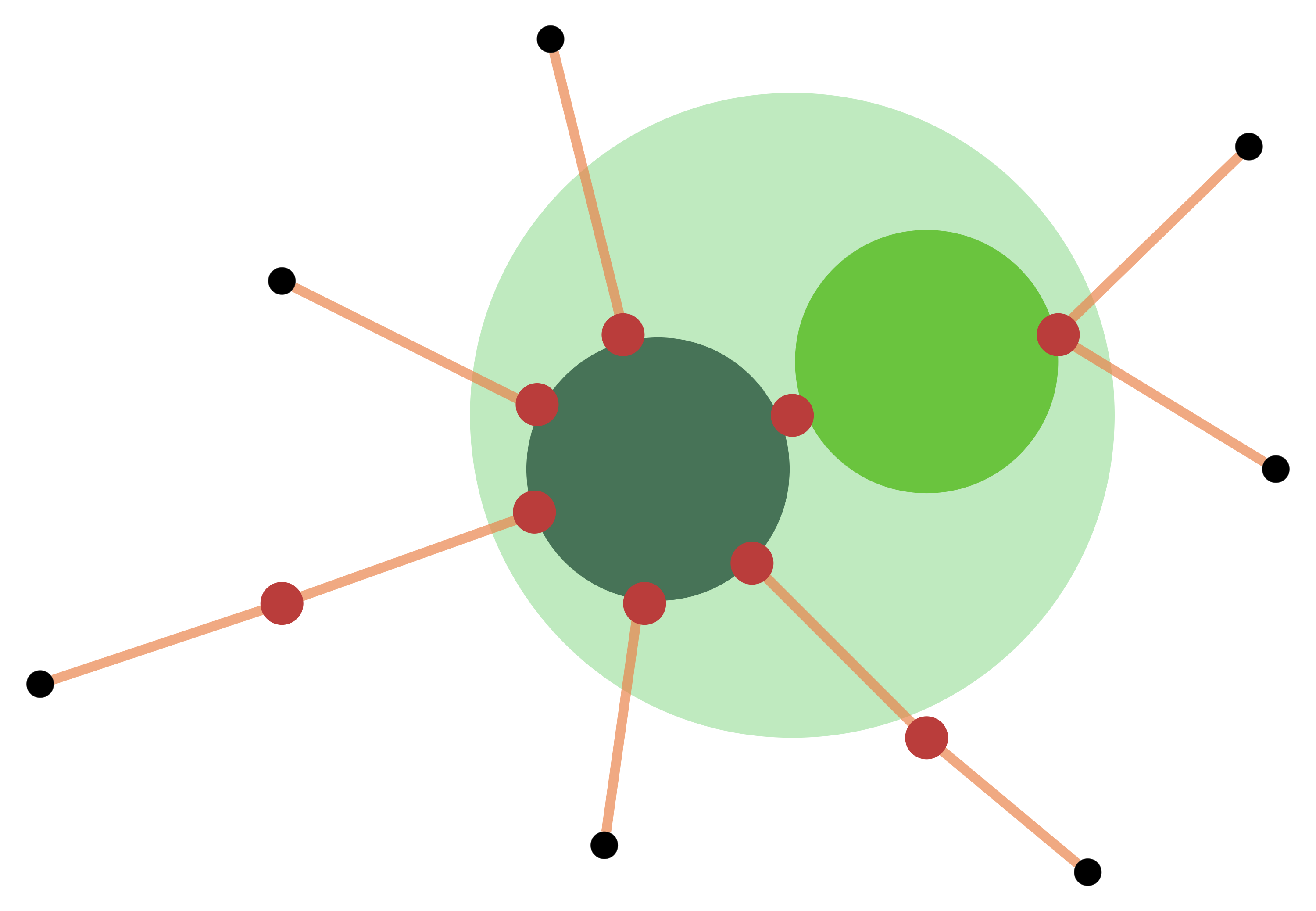}}}
\caption{Original IEEE 118-node test network. It has 9 bridges (in orange), one nontrivial bridge-block (in green), and 9 trivial bridge-blocks. Furthermore, it has 9 cut vertices (in red), one of which decomposes the nontrivial bridge-block into 2 blocks (in different shades of green).} \label{fig:influencegraph_pre}
\end{figure}

As illustrated for the IEEE 118-node test network in Fig.~\ref{fig:influencegraph_pre}(a), the influence graph of power network tend to be very dense and connecting many lines that are topologically far away, corroborating the non-local propagation of line failures.
Fig.~\ref{fig:influencegraph_pre}(b) shows the block and bridge-block decomposition of the same network. This network has a trivial bridge-block decomposition, consisting of a single nontrivial bridge-block and $9$ trivial ones. This bridge-block consists of two nontrivial blocks, separated by a cut vertex. As predicted by the theory (cf. Proposition~\ref{prop:zeroLODFPTDF}), the influence graph is not connected, since all the LODFs corresponding to pairs of lines in the two blocks are identically zero. 

By removing three lines from IEEE 118-node test network, we can make both the block and bridge-block decompositions of this network finer, as illustrated in Fig.~\ref{fig:influencegraph_post}. The influence graph of the pruned network is now less dense and have more disconnected components, showing a higher localization capability.
Fig.~\ref{fig:influencegraph_post}(b) shows the finer block and bridge-block decomposition of the pruned network. It has two nontrivial bridge-blocks, one of which further decomposes into four nontrivial blocks, separated by cut vertices.

\begin{figure}[!ht]
\centering
\subfloat[][The influence graph (in gray) of the pruned IEEE 118-node test network obtained using the threshold $K_\mathrm{min}=0.005$. Due to the presence of four blocks (in blue and three shades of green), the influence graph has four clearly visible disconnected components, since all the corresponding LODFs are zero.]{\makebox[\width]{\includegraphics[width=.95\textwidth]{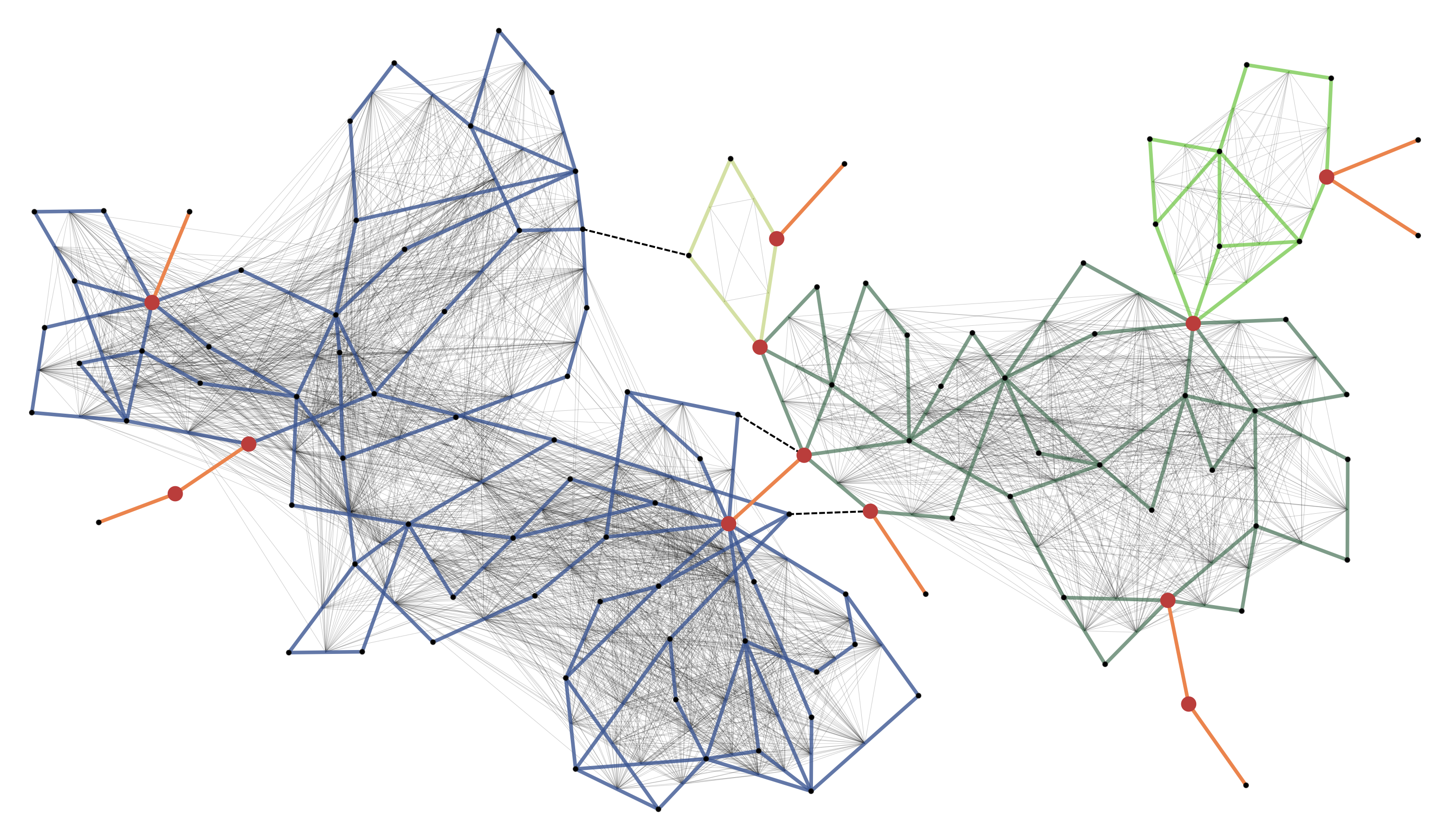}}}\\
\subfloat[][Schematic representation of the block and bridge-block decomposition of the pruned IEEE 118-node test network. The green bridge-block further is further decomposed by cut vertices into three blocks (in different shades of green). ]{\makebox[1.9\width]{\includegraphics[width=.5\textwidth]{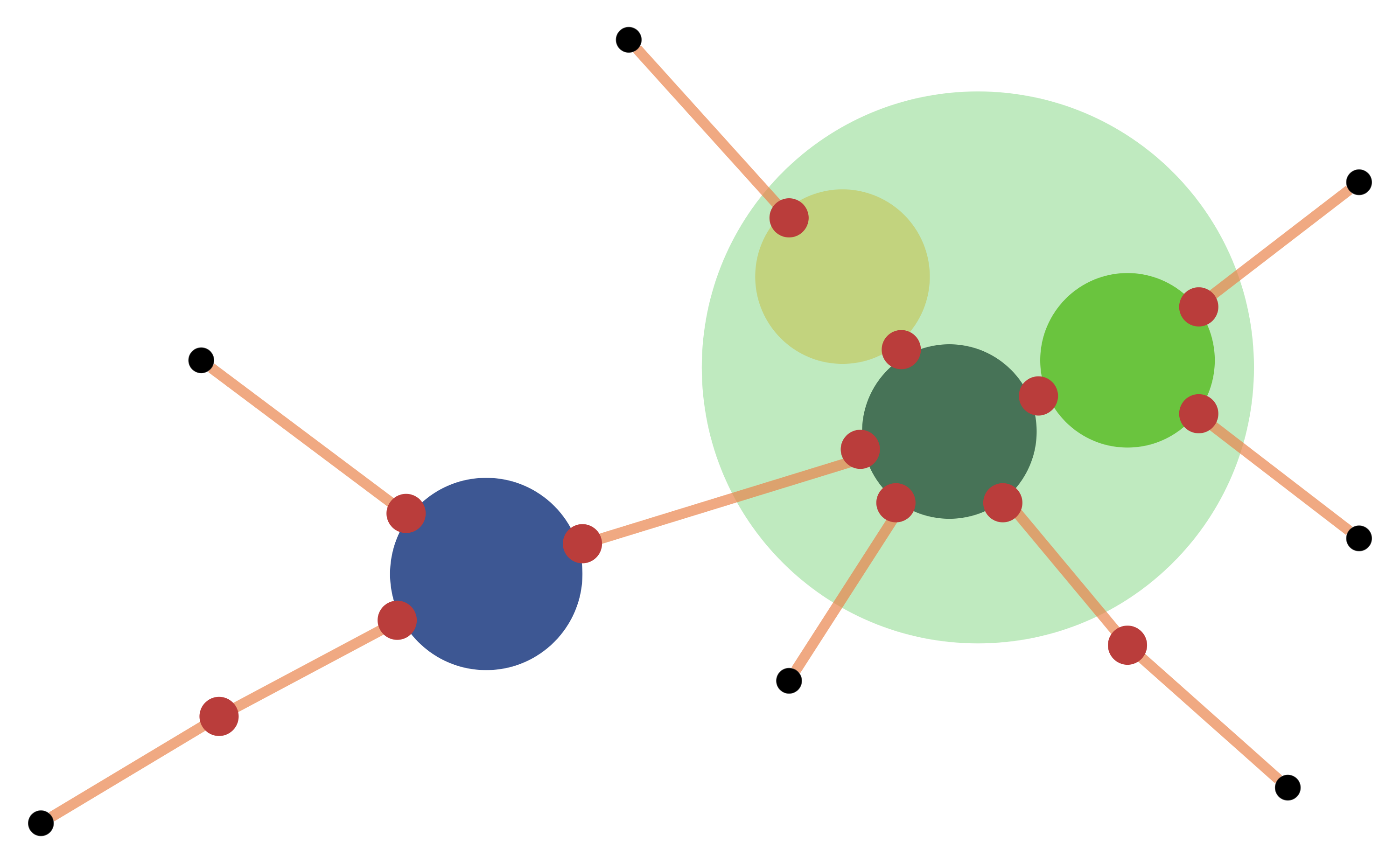}}}
\caption{The network obtained from the IEEE 118-node test network removing three lines (shown as dashed black edges). It has 10 bridges (in orange), two nontrivial bridge-block (in blue and green), and 9 trivial bridge-blocks. Furthermore, it has 13 cut vertices (in red), two of which decomposes the nontrivial green bridge-block into 3 blocks (in three different shades of green).} \label{fig:influencegraph_post}
\end{figure}
\FloatBarrier


\subsection{Cut set outages: Power rebalance and redistribution}




We now discuss the impact of cut set outages. As mentioned in Subsection~\ref{sub:cut}, we assume the post-contingency power balancing rule is decoupled for each island so that we can study the power redistribution in each of them separately. In particular, we consider the proportional control $\mathbb G_{\alpha}$ that the post-contingency power injection is adjusted proportionally to compensate the power imbalance caused by the cut set outages.

Consider an island $\calI$ and let $E_\calI:=\set{(i,j)\in E ~:~ i\in\calI \text{ and } j\in\calI }$ denote the set of edges whose both end points belong to $\calI$ before contingency. As outlined in Subsection~\ref{sub:cut}, we partition the failed lines into tie line failures $\calE_\textrm{tie}$ and internal line failures $\calE_\textrm{int}$.
Rewriting~\eqref{ch:cf; subsec:bridgeoutage; eq:DeltaP.1} for a specific surviving line $\ell \in E_\calI \setminus \calE_\textrm{int}$ in island $\calI$ yields
\begin{equation}
\Delta f_{\ell} = 
	\sum_{\hat \ell\in \calE_\textup{int}} K^{\calE_\textup{int}}_{\ell\hat \ell} \, f_{\hat \ell}   \ \ + \ 
	\sum_{\hat \ell\in \calE_\textup{tie}} f_{\hat \ell}  \sum_{k: \alpha_k>0} \alpha_k
		\Big( D_{\ell, k j(\hat \ell)} + \sum_{\hat \ell'\in \calE_\textup{int}} K^{\calE_\textup{int}}_{\ell\hat \ell'} D_{\hat \ell', k j(\hat \ell)} \Big).
\label{ch:cf; subsec:bridgeoutage; eq:DeltaP.l}
\end{equation}

It is observed that the impact of cut set outages consists of two terms, one capturing the impact of the internal line failures in $\calI$ and the other capturing the impact of the tie line failures, as shown in~\eqref{ch:cf; subsec:bridgeoutage; eq:DeltaP.l}. In order to describe the impact boundary of such cut set outages, we will study the conditions so that \eqref{ch:cf; subsec:bridgeoutage; eq:DeltaP.l} is zero. 

Decompose the pre-contingency subnetwork corresponding to the island $\calI$ into blocks. For any surviving line $\ell$ and failed internal line $\hat \ell \in \calE_\textrm{int}$ pair, Proposition~\ref{prop:glodf_non_cut} indicates that $K_{\ell \hat \ell}^{\calE_\textrm{int}} = 0$ if $\ell$ and $\hat\ell$ are in different blocks of $\calG_\calI$. Therefore, if a surviving line $\ell$ is not in the same block as any outaged internal lines, \eqref{ch:cf; subsec:bridgeoutage; eq:DeltaP.l} simplifies to \[
\Delta f_\ell = \sum_{\hat \ell\in \calE_\textup{tie}} f_{\hat \ell}  \sum_{k: \alpha_k>0} \alpha_k
		 D_{\ell, k j(\hat \ell)} 
\]
This suggests that the impact of multiple outaged tie lines is separable. 
In the following lemma we state the \textit{Simple Path Criterion}, cf.~\cite[Theorem 1]{TPS2}, which provides a sufficient condition for the PTDF introduced in Subsection~\ref{sub:PTDF} to be zero in purely topological terms.
\begin{lem}[Simple Path Criterion] \label{lem:simple_path}
For any line $\ell\in E$ and two nodes $j,k\in V$, if $\ell$ there is no simple path in the power network $G=(V,E,\bm{b})$ connecting nodes $j$ and $k$ that contains line $\ell$, then the PTDF $D_{\ell,jk}=0$.
\end{lem}
Leveraging this criterion and the notion of participating bus introduced in Subsection~\ref{sub:cut}, we can derive a sufficient condition, summarized in the following proposition~\cite[Corollary 6]{TPS2}, for a surviving line $\ell$ not to be impacted by a cut set outage. 


\begin{prop}[Characterization for cut set failures]
\label{prop:glodf_cut}
Consider a cut set $\calE$ outage that disconnects the power network into two or more islands with decoupled power balancing rules. Let $\calI$ denote one such island under the proportional control $\mathbb G_{\alpha}$, and let $\calE_\calI = \calE_\textrm{tie} \cup \calE_\textrm{int}$ be the set of outaged lines with one or both endpoints in this island. For any surviving line $\ell$, if $\ell$ is not in the same block as any outaged internal lines $\hat \ell \in \calE_\textrm{int}$ and $\ell$ is not on any simple path connecting the endpoint of an outaged tie line $j(\hat \ell)$ and a participating bus $k$ with $\alpha_k > 0$ for any $\hat \ell \in \calE_\textrm{tie}$, then the post-contingency flow on line $\ell$ remains unchanged.
\end{prop}

We can reformulate the latter result in even simpler terms in the case of a single bridge line outage, obtaining the following simplified characterization for the surviving lines whose post-contingency flow remains unchanged, see~\cite[Corollary 2]{TPS2}.

\begin{cor}[Characterization for bridge failures]
Consider the outage of a bridge line $\hat \ell=(\hat i, \hat j)$ with nonzero flow $f_{\hat \ell}\neq 0$ that disconnects the power network into two islands. Let $\calI$ be the island that contains node $\hat j$.
Under the proportional control $\mathbb G_{\alpha}$, for any line $\ell$ in $\calI$ if there is no simple path in $\calI$ that contains $\ell$ from the node $\hat j$ to a participating bus $k\in \calI$ with $\alpha_{k}>0$, then the post-contingency flow on line $\ell$ remains unchanged.
\end{cor}

\FloatBarrier

\section{Improving network reliability by bridge-block decomposition refinement}
\label{sec:optimization}
In this section we outline a new method to refine the bridge-block decomposition of a given power network by means of \textit{line switching actions}. 
Even though we usually think of transmission lines as static assets, network operators have the ability to remotely control circuit breakers that open and close them, effectively removing them from service and thus modifying the network topology. This procedure, that goes under the name of \textit{transmission line switching}, is a mature power system technology that is already commonly used by network operators. 
A large body of literature has been devoted to the optimization problems that stem from this network flexibility, see our discussion about Optimal Transmission Switching (OTS) in Section~\ref{sub:literature}.

In this paper we want to propose a new way to improve the network reliability leveraging the flexibility offered by these transmission line switches. Our strategy is inspired by Lemma~\ref{lem:finer}, which suggests that switching off a subset of the currently active lines is a way to refine the bridge-block decomposition of a network. In Figs.~\ref{fig:influencegraph_pre} and~\ref{fig:influencegraph_post} we briefly illustrated how a significant bridge-block decomposition refinement is possible for the IEEE 118-bus network by switching off only three lines.
Such switching actions can be taken both proactively, as semi-permanent network design decision, but also as temporary response to a contingency or an unusual load/generation profile. As summarized by our results in Section~\ref{sec:localization}, a finer bridge-block decomposition greatly improves the network robustness against failure propagation, since it automatically avoids long-distance propagation of line failures. More specifically, an immediate consequence of Proposition~\ref{prop:glodf_non_cut} is that any non-cut set line outage does not propagate outside the bridge-bridge block(s) to which the disconnected lines belong. 

Our approach is particularly promising in view of the fact that most power grids have a trivial bridge-block decomposition, making a global propagation of line failures very likely in these networks. Indeed, all the considered IEEE test networks have a single massive bridge-block and possibly a few additional blocks consisting of a single vertex, see Fig.~\ref{fig:influencegraph_pre}, Table~\ref{tab:networkstatistics} in Section~\ref{sec:numerics}, and Fig.~\ref{fig:oneshotalgorithm}(a). 

The crucial step is thus, given a power network $G$, identifying a subset of lines $\calE \subset E$ to be temporarily removed from service by means of switching actions, so that the newly obtained network $G^\calE$ has a ``good'' bridge-block decomposition $\calBB(G^\calE)$ while still remaining connected. In this section, we formulate an optimization problem to find the best subset of lines to switch off that would guarantee a desired network bridge-block decomposition without creating congestion in the remaining lines.
However, any line switching action causes a power flow redistribution on the remaining lines (exactly like line outages do) and may cause overloads. It is then crucial to carefully assess the impact of each set of switching actions on the network flows and avoid creating any congested line. A brute-force search to find such best switching actions and recalculate the line flows for each instance is infeasible, since there are exponentially many such subsets of lines $\calE \subset E$ to consider.


Aiming to reduce the number of possible subsets of lines to consider, our method takes a different approach and ``works backwards''. Indeed, we first identify a ``good'' network partition $\calP$ whose clusters we want to become its bridge-blocks and then, using these clusters, we choose which of the cross-edges between them to remove so that $\calP$ actually becomes the network bridge-block decomposition. We thus propose an algorithm that consists of two steps, each one tackling one of these two sub-problems. 
This section is structured as follows:
\begin{itemize}
    \item In Subsection~\ref{sub:mod} we introduce the \textit{Optimal Bridge-Block Identification} (OBI) problem, which aims at identifying a good network partition based on the current network flow configuration, 
    \item In Subsection~\ref{sub:obs} we tackle the second sub-problem by formulating the \textit{Optimal Bridge Selection} (OBS) problem, which identifies the best subset of lines to be switched off that refines the network bridge-block decomposition and simultaneously minimizes the network congestion level, 
    \item In Subsection~\ref{sub:pra} we discuss some practical considerations regarding the algorithm implementation and present a faster recursive variant.
\end{itemize}

\subsection{Identifying candidate bridge-blocks using network modularity}
\label{sub:mod}
The bridge-block decomposition refinement should ideally require a minimum number of line switching actions. For this reason, it is thus crucial to leverage the structure of the existing network as much as possible when identifying the candidate bridge-blocks. In this subsection we describe in detail how we tackle this first sub-problem using the notion of network modularity.

Consider a power network $G=(V,E,\bm{b})$ with a fixed power injection configuration $\pp$. Aiming to find a network partition $\calP = \{\calV_1,\dots,\calV_{|\calP|}\}$ into clusters with the most suitable properties to become bridge-blocks via line switching, we introduce the following \textit{Optimal Bridge-Blocks Identification (OBI)} problem:
\begin{equation}
\label{eq:maxmodprob}
\max_{\calP \in \Pi(G)} \frac {1}{2 M} \sum_{r=1}^{|\calP|} \sum \limits _{i,j \in \calV_r} \left ( |f_{ij}|-\frac {F_{i} F_{j}}{2 M} \right ),  \qquad \textrm{(OBI)}
\end{equation}
where $\ff=(f_{ij})_{(i,j)\in E}$ is the current power flow pattern determined by the injections $\pp$ as prescribed by~\eqref{eq:flowsdc_model}, $F_{i} :=\sum_{j \in V} |f_{ij}|$ is the sum of absolute line flow incident to vertex $i$ and $M := \frac{1}{2} \sum_{i \in V} F_i = \frac{1}{2} \sum_{i,j \in V} |f_{ij}|$ is the total absolute line flows over the network.
When we restrict the subset of feasible partitions to $\Pi_{k}(G)$, we refer to problem~\eqref{eq:maxmodprob} as the OBI-$k$ problem. 

Having in mind the optimal line switching that we intend to perform in the second step of our algorithm, it is important to remember that each line that we decide to switch off will cause a power flow redistribution on the remaining lines. Aiming to prevent any of these lines to become congested, it is thus crucial to identify in this first step a ``good'' partition into clusters such that (i) a smaller number of switching actions to create the desired bridge-block decomposition and that (ii) the power flow redistribution ensuing the switching actions has a minimal impact on the rest of the network and does not create congestion.

The OBI problem~\eqref{eq:maxmodprob} is formulated precisely to find a partition with these features. Indeed, its objective function indirectly favors partitions that have:
\begin{itemize}
    \item[(F1)] very few (and/or low-weight) cross-edges between clusters;
    \item[(F2)] clusters are balanced in terms of total net power and with very modest power flows from/towards the neighboring clusters.
\end{itemize}
These features follows from the fact that the objective function of~\eqref{eq:maxmodprob} is a power-flow-weighted version of the \textit{network modularity} introduced by~\cite{Newman2004b,Newman2004}. Network modularity is a key notion that has been introduced to tackle a classical problem in complex networks theory, namely unveiling the underlying community structure in a given network. Informally, the proposed optimal block identification problem is a community detection problem on a transmission network topology while using in a crucial way its electrical properties. The network modularity $Q(\calP)$ associated with a partition $\calP$ is a scalar between $-1/2$ and $1$~\cite{Brandes2008} that quantifies the quality of such a partition: a large value indicates a stronger ``community structure'', in the sense that there are more (and in the weighted case high-weight) edges within each cluster in $\calP$ than across them.

Broadly speaking, a partition that yields a high network modularity is one where the connectivity within each cluster is larger than the one expected in a random network with the same nodal properties. More specifically, the modularity of a network partition measures the total weight of the internal edges within each cluster against the same quantity in a network with the same number of nodes and nodal properties, but in which the edges are placed randomly. Indeed, the term $\frac {F_i F_j}{2 M}$ should be interpreted as the likelihood of finding edge $(i,j)$ in a random graph conditioned on having the node generalized degree sequence $F_1,\dots, F_n$.

For the OBI problem we take as the objective function the network modularity where each edge weight is equal to absolute power flow on the corresponding line. In view of the definition of network modularity, the optimal partitions for the OBI problem are those with few lines between clusters and modest power flowing on them. We also remark that using the absolute power flows as weights in the OBI problem implicitly makes the resulting bridge-block decomposition adaptive to the possibly changing power injection and flow patterns. Such a feature would not be present if we used an unweighted network modularity as objective function or if we used other static network quantities, e.g. the line susceptances $\{\b_\ell\}_{\ell \in E}$, as weights. 

The notion of network modularity has been extended to deal with negative edge weights in~\cite{Gomez2009}: in this extension the high-modularity partitions tend to assign the two nodes at the endpoints of an edge with a negative weight to different clusters. In power networks, however, the flow signs are just an artifact of the (arbitrary) directions chosen for the network edges -- in fact, one could always choose edge directions so that all the flows are positive. Our goal to identify the network partitions with modest flows across different clusters \textit{regardless of their directions} makes the \textit{absolute values} of the flows the sensible choice for the edge weights.



\subsubsection{Related classical clustering problems}
One might argue that a simpler approach for bridge-block identification could have been solving a weighted \textit{minimum $k$-cut problem}
\begin{equation}
\label{eq:mincut}
	\arg\min_{\calP \in \Pi(G)} \sum_{r=1}^{|\calP|} \mathrm{Cut}(\calV_r),
\end{equation}
where $\mathrm{Cut}(\calV_r) := \sum_{i \in \calV_r, \, j \in \calV_r^c} |f_{ij}|$. This combinatorial optimization problem can be solved efficiently, especially in the case $k=2$, as illustrated in~\cite{StoerWagner1997}. However, such a min-cut solution would often not be a satisfactory partition for the power systems application at hand. Indeed, in many instances the min-cut solution separates one or very few vertices from the rest of the network and such trivial partitions will yield a rather unbalanced bridge-block decomposition with limited failure localization potential. On the contrary, it is very unlikely that the OBI problem yields trivial partitions. Indeed, using the weighted network modularity~\eqref{eq:maxmodprob} as objective function favors partitions with ``balanced clusters'', which in our setting means with small total net power $\sum_{i \in \calV_r} p_i$, $r=1,\dots,|\calP|$. This feature of high modularity partition has been proved rigorously: as shown by~\cite{Reichardt2006} and~\cite{Bertozzi2018}, for a fixed number $k$ of clusters the OBI problem~\eqref{eq:maxmodprob}
has the same solution of the following \textit{balanced-cut problem}
\begin{equation}
\label{eq:balancedcut}
	\arg\min_{\calP \in \Pi_k(G)} \sum_{r=1}^k \left ( \mathrm{Cut}(\calV_r) + \frac {\mathrm{Vol}(\calV_r)^2}{2 M} \right ),
\end{equation}
where we used the classical notion of \textit{volume} of a cluster, which in our context is equal to $\mathrm{Vol}(\calV_r)= \sum_{i \in \calV_r} F_i$. This equivalence suggests that an optimal partition that maximizes modularity favors sparsely interconnected clusters with balanced volumes. This is particularly evident in the case of $k=2$ target clusters, in which~\eqref{eq:balancedcut} is equivalent to
\begin{equation}
\label{eq:modvol}
	\arg\min_{\calV \subset V} \left ( \mathrm{Cut}(\calV) + \frac {\mathrm{Vol}(\calV) \mathrm{Vol}(\calV^c)}{2 M} \right ).
\end{equation}
The second term $\mathrm{Vol}(\calV) \mathrm{Vol}(\calV^c)$ is maximal when $\mathrm{Vol}(\calV)= \mathrm{Vol}(\calV^c)=M$ and thus favors network partition with two clusters of roughly equal volume. In contrast, the first term favors a partition of the graph in which few edges with small weights are removed. This is reminiscent of the NCut problem introduced in~\cite{ShiMalik2000}, in which the objective is to minimize the ratio
\begin{equation}
\label{eq:ncut}
	\textrm{NCut}(\calP)= \frac{\mathrm{Cut}(\calV,\calV^c)}{\mathrm{Vol}(\calV)} +  \frac{\mathrm{Cut}(\calV,\calV^c)}{\mathrm{Vol}(\calV^c)} 
	\propto \frac{\mathrm{Cut}(\calV,\calV^c)}{\mathrm{Vol}(\calV) \mathrm{Vol}(\calV^c)}
\end{equation}
over all possible bipartitions $\calP=\{\calV,\calV^c\}$.
Whenever a cluster-balancing condition is introduced in the objective function either in the form of~\eqref{eq:balancedcut}, \eqref{eq:modvol} or~\eqref{eq:ncut} the min-cut problem become NP-hard, see~\cite{WagnerWagner1993} for details on NCut and~\cite{Brandes2006} for details on network modularity. Efficient algorithms are thus necessary to find good locally optimal network partitions with reasonable computational costs. Numerous methods have been proposed in the literature, for a review see~\cite{Fortunato2010,Fortunato2016,Porter2009,Chen2014}. These include greedy algorithms~\cite{Blondel2008,Clauset2004,Newman2004a,Wakita2007}, extrenal optimization~\cite{Boettcher2002,Duch2005}, simulated annealing~\cite{Guimera2005,Massen2005,Medus2005}, spectral methods 
~\cite{White2005,Newman2006a,Richardson2009}, sampling techniques~\cite{Sales-Pardo2007} and mathematical programming~\cite{Agarwal2008}.
In the present work we focus on two of these state-of-the-art algorithms, namely the ``fastgreedy'' method of~\cite{Clauset2004} and the spectral clustering method~\cite{Newman2006a,Newman2006b}, which we briefly review later in Section~\ref{sec:numerics}. These two methods are particularly suitable for our purposes because they allow for recursive partitioning. In the next section we use this feature in a crucial way, by devising an algorithmic procedure that progressively refines the block decomposition by splitting existing blocks in smaller ones, thus allowing to take switching actions in a sequential fashion.

\subsubsection{Intentional controlled islanding}

In this subsection, we briefly give an overview of another failure mitigation strategy that has some similarities with the candidate bridge-blocks identification described earlier in this section.

In the midst of a cascading failure, it is sometimes recommended to avoid catastrophic wide area blackouts by means of \textit{intentional controlled islanding} (ICI)~\cite{Zhao2003,you2004}. In other words, by means of line switching actions, they temporarily partition the network into islands to prevent the cascading failure from propagating further and/or creating very imbalanced grid sections. Moreover, such an intentional islanding can enhance faster restoration of the full grid later and avoid transient stability issues during reconnection~\cite{you2003self,you2004}. 

The network operator ideally would want each of these islands to have:
\begin{itemize}
\item a net power as close as possible to zero, in the sense that power supply and demand should be almost equal, so that only minimal load shedding or generators adjustment are needed;
\item enough capacity to safely sustain the power flows within the island.
\end{itemize}
A lot of research has been devoted to the problem of finding a network partition with such properties, using either efficient search algorithms~\cite{Sun2003,Zhao2003}, mixed integer programming~\cite{Fan2012,Pahwa2013}, graph theory~\cite{Soltan2017}, or spectral clustering techniques~\cite{Esmaeilian2017,Bialek2014}. Similar methods have also been used in~\cite{Liu2009,Vittal2003,Vittal1998} to group together synchronized generators in case of frequency disturbance events. 

These island design principles are clearly very much aligned with those we outlined for bridge-blocks, see (F1) and (F2) in the previous section. In view of this similarity, we expect an optimal islands design based on weighted network modularity to be very effective, but a detailed exploration of this other application is out of the scope of the present paper.

Dynamically creating a bridge-block decomposition is proposed in~\cite{Bialek2021}  as a less drastic emergency measure in alternative to controlled islanding. The authors also classical spectral $k$-way clustering techniques to identify the optimal clusters to be transformed into bridge-blocks.

\subsection{Building the bridge-block decomposition via optimal switching actions}
\label{sub:obs}

In the previous subsection we formulated the OBI problem, which for a given power network identifies a partition with desirable features that we want to transform into the ``backbone'' of the network bridge-block decomposition. In this subsection we focus on the second sub-problem, namely how to, given such a partition $\calP$, reconfigure the network topology so that the resulting network (i) has a bridge-block decomposition at least as fine as $\calP$ and (ii) has minimum line congestion. 

The feasibility and effectiveness of such a network reconfiguration crucially depends on the ``quality'' of the chosen partition $\calP$. Nonetheless, all the results of this section do not rely on any specific features of $\calP$ and they are valid for a general partition of $G$. The only additional assumption that we tacitly make in the rest of the section is the following:
\begin{itemize}
    \item[(A1)] for every cluster $\calV \in \calP$, the subgraph of $G$ induced by $\calV$ is connected.
\end{itemize}
Note that this is not really a stringent assumption: indeed, for any partition $\calP$ that does not satisfy it, we can always consider the finer partition $\calP'$ in which all the connected components of all its subgraphs are considered as separate clusters. Assumption (A1) is useful as it links the connectedness of the network with that of its reduced graph, as illustrated by the following lemma. 

\begin{lem}[Connectedness of the reduced graph] \label{lem:connected}
If $G$ is a connected graph, then the reduced graph $G_\calP$ is connected for any partition $\calP$. The converse is true only when the partition $\calP$ satisfies (A1). 
\end{lem}

To any partition $\calP$ is associated a subset $E_c(\calP) \subset E$ of cross-edges, i.e., the edges whose endpoints belong to different clusters of $\calP$ (cf.~Subsection~\ref{sub:basic}). Given a partition $\calP$ of $G$, we are interested in finding a subset $\calE \subset E_c(\calP)$ of cross-edges that, if removed, yield a network $G^\calE$ whose bridge-block decomposition is at least as fine as $\calP$, i.e.,
\begin{equation}
\label{eq:BBGE}
    \calBB(G^\calE) \succeq \calP.
\end{equation}
Each of these subsets of cross-edges corresponds to a different set of transmission lines that we could temporarily disconnect by means of switching actions, which yields a transmission network topology with the desired bridge-block decomposition.
The next proposition shows that the condition~\eqref{eq:BBGE} is equivalent to requiring that the reduced graph induced by the partition $\calP$ on the post-switching network $G^\calE$ is a tree.
\begin{prop}[Reduced graph characterization]
\label{prop:treecharacterization}
Let $\calP$ be a vertex partition of a connected graph $G$. Then, $\calBB(G) \succeq \calP$ if and only if the reduced graph $G_\calP$ is a tree.
\end{prop}
\begin{proof}
Assume by contradiction that $G_\calP$ is not a tree and thus there exists a non-trivial cycle in $G_\calP$. Since we assume all the clusters of $\calP$ induce connected subgraphs, there exists also a non-trivial cycle $c$ in $G$ that includes vertices from at least two different clusters of $\calP$. Recall that by definition of bridge-block decomposition $\calP_{\mathrm{circuit}}(G) = \calBB(G)$ (cf. Subsection~\ref{sub:basic}), all the vertices on the cycle $c$ must belong to the same bridge-block, say $B \in  \calBB(G)$. But then $\calBB(G)$ cannot be finer than $\calP$ since the bridge block $B$ is not contained in any cluster in $\calP$, since it has vertices from at least two different clusters of $\calP$ and thus $B \not\subseteq V$ for every $V \in \calP$.

The reverse implication can be proved along the same lines, assuming by contradiction that $\calBB(G^\calE) \not\succeq \calP$ and leveraging again the connection between the circuits in $G$ and those in the reduced graph $G_\calP$.
\end{proof}

In view of Proposition~\ref{prop:treecharacterization}, we can assume that $G_\calP$ is not a tree, otherwise is suffices to take $\calE = \emptyset$ and the target condition~\eqref{eq:BBGE} is automatically satisfied. When the reduced graph $G_\calP$ is a not a tree (possibly with/because of multiple edges), there are many valid subsets of cross-edges $\calE \subset E_c(\calP)$ that satisfy the target property~\eqref{eq:BBGE}. Indeed, there exists a suitable subset of cross-edges for each spanning forest of the reduced graph $G_\calP$. Aiming to have the least impact possible on the network and, in particular, avoiding its disconnection, we exclude all the subsets $\calE \subset E_c(\calP)$ that are cut sets for the graph $G$ and thus we can focus on the spanning trees of the reduced graph $G_\calP$.

Among all these possible choices for the set $\calE$, we want to pick the one that has the least impact on the surviving network in terms of line congestion. We define as \textit{congestion} on a transmission line $\ell \in E$ the ratio between the absolute power flow on that line and its (thermal) capacity $C_\ell$, i.e., $|f_\ell|/C_\ell \geq 0$.

Assuming the power injections $\pp$ do not change, after the removal of the lines in any such subset $\calE \subset E_c(\calP)$, the power redistributes on the remaining $|E \setminus \calE|$ lines as described in Subsection~\ref{sub:noncut}, yielding a new vector of new power flows $\ff^\calE$. We define the \textit{network congestion level} $\gamma(\calE)$ as the maximum line congestion in the surviving graph $G^\calE$, i.e.,
\begin{equation}
\label{eq:congestionlevel}
	\gamma(\calE):=\max_{\ell \in E\setminus \calE} |f^\calE_\ell| / C_\ell.
\end{equation}
With a minor abuse of notation, we use $\gamma(\emptyset)$ to denote to the network congestion level of the original network where no lines have been removed yet, i.e., $\calE = \emptyset$. Power networks must be operated so that each line flow is strictly below the capacity of the corresponding line, meaning that the network congestion level should not be larger than one.

The problem of finding the best set $\calE$ of switching actions that yields a network $G^\calE$ with minimum network congestion and bridge-block decomposition as fine as $\calP$ can be formalized as the following \textit{Optimal Bridge Selection} (OBS) problem:
\begin{subequations}\label{eq:optswitch}
\begin{align}
	 \min_{\calE \subset E_c(\calP)} \quad & \gamma(\calE) \qquad \qquad \qquad \textrm{(OBS)} \label{eq:optswitch_basic}\\
	 \textrm{ s.t. } \quad &(G^\calE)_\calP \text{ is a tree}. \label{eq:optswitch_constraint2}
\end{align}
\end{subequations}

The constraint~\eqref{eq:optswitch_constraint2}, by requiring the reduced graph $(G^\calE)_\calP$ of the post-switching network $G^\calE$ to be a tree (and not a forest), implicitly guarantees that $(G^\calE)_\calP$ is connected and therefore so is $G^\calE$ by virtue of Lemma~\ref{lem:connected}. In other words, for any feasible set $\calE$ for~\eqref{eq:optswitch} the resulting post-switching network $G^\calE$ is still connected.

In view of Proposition~\ref{prop:treecharacterization}, constraint~\eqref{eq:optswitch_constraint2} ensures that the bridge-block decomposition of the surviving network $G^\calE$ is at least as fine as $\calP$, but it can possible be finer. This happens when the removal of the cross-edges in $\calE$ accidentally creates additional bridges in $G^\calE$ that the original network did not have. By construction these new bridges were not cross-edges and thus must have been \textit{inside} one or more clusters of $\calP$. This means that the switching actions $\calE$ involuntarily split each of these clusters into two or more smaller bridge-blocks, and thus the resulting bridge-block decomposition $\calBB(G^\calE)$ is finer than $\calP$. This happens also in the example that we illustrate in the next subsection, see Fig.~\ref{fig:oneshotalgorithm}.
If $G_\calP$ is not a tree, the OBS problem~\eqref{eq:optswitch} has always at least one feasible nontrivial solution $\calE \neq \emptyset$ for any target partition $\calP$. Indeed, $G_\calP$ has at least one spanning tree as long as it is connected, property guaranteed by assumption (A1) and by Lemma~\ref{lem:connected}.

It could be the case that even for the optimal set $\calE^*$ of switching actions, at least one line gets congested, i.e.,
$$\gamma(\calE) \geq 1, \quad \text{ for every feasible subset } \calE.$$
This is a rare occurrence and, in fact, an optimal solution $\calE^*$ with congestion level $\gamma(\calE^*) < 1$ exists in most cases (see Section~\ref{sec:numerics} for our numerical results for IEEE test networks), especially if the original network was not heavily congested, i.e., $\gamma (\emptyset) \not\approx 1$. Furthermore, as we illustrate in the same section, it is often the case that the network congestion level \textit{decreases} after taking the optimal switching actions, i.e.,
\[
	\gamma(\calE^*) < \gamma (\emptyset).
\]
A possible explanation for this counter-intuitive fact is the presence of loop-flows involving edges in two or more clusters, which are then removed by the switching actions, hence possibly reducing the congestion on the newly created bridges or inside the clusters. 

Since the optimal solution of the combinatorial OBS problem~\eqref{eq:optswitch} may not be unique, in which case we heuristically select the one that minimizes the total number of congested lines in the post-switching network. 

Lastly, we remark that an optimization problem similar to OBS is proposed in~\cite{Bialek2021} to refine the bridge-block decomposition as emergency measure. In their formulation, however, the objective is to minimize the total weight of the spanning tree $(G^\calE)_\calP$ of the reduced graph rather than minimizing the congestion in the post-switching network.

\subsection{Full algorithm and its recursive variant}
\label{sub:pra}
We now present in full our proposed algorithm by combining the two steps described in the previous subsections. Furthermore, we discuss some practical considerations regarding the implementation of the OBS problem. In particular, we show how the computational complexity of~\eqref{eq:optswitch} can be dramatically reduced by adopting a strategy where the bridge-block decomposition is progressively made finer by recursively splitting the existing bridge-blocks in smaller ones.


By combining the OBI and OBS problems described in the previous subsections we obtain the following procedure to refine the bridge-block decomposition of a given power network $G$ with power injections $\pp$. This algorithm also requires in input the target number of bridge-blocks $b \in \mathbb{N}$.

\begin{algorithm}[!ht]
\SetAlgoLined
\SetKwInOut{Input}{Input}\SetKwInOut{Output}{Output}
\Input{The power network $G=(V,E,\bm{b})$, the set of power injections $\pp$, the target number of bridge-blocks $b$}
\BlankLine
\DontPrintSemicolon
Calculate the current line flows $\ff$ on the network $G$ induced by the injections $\pp$;\\
Calculate the current bridge-block decomposition $\calBB(G)$;\\
Consider in isolation the largest bridge-block in size of the current bridge-block decomposition, say $\calV \in \calBB(G_i)$;\\
Solve a OBI-$b$ problem for $\calV$ in isolation to find its optimal $b$-partition $\calP^*$;\\
Solve the OBS problem~\eqref{eq:optswitch} for $\calV$ in isolation with target $b$-partition $\calP^*$, obtaining an optimal subset $\calE^*_i$;\\
Calculate the new congestion level $\gamma^*=\gamma(\calE^*)$;\\
Calculate the new set of active lines to be $E' := E \setminus \calE^*_{i}$, set $G'=(V,E',\bm{b}')$, and calculate its bridge-block decomposition $\calBB(G')$;\\
\Output{$G'$, $\calBB(G')$, $\gamma^*$}
\caption{One-shot bridge-block decomposition refinement}\label{alg:oneshotrefinement}
\end{algorithm}
Fig.~\ref{fig:oneshotalgorithm} illustrates the steps of Algorithm~\ref{alg:oneshotrefinement} with $b=4$ when used to refine the bridge-block decomposition of the IEEE 73-bus network.

\begin{figure}[!ht]
    \centering
    \vspace{-0.3cm}
    \subfloat[The graph $G$ corresponding to IEEE 73 network with its bridge-block decomposition $\calBB(G)$ consisting of three bridge-blocks, two of which are trivial.]{\makebox[1.2\width]{\includegraphics[scale=0.39]{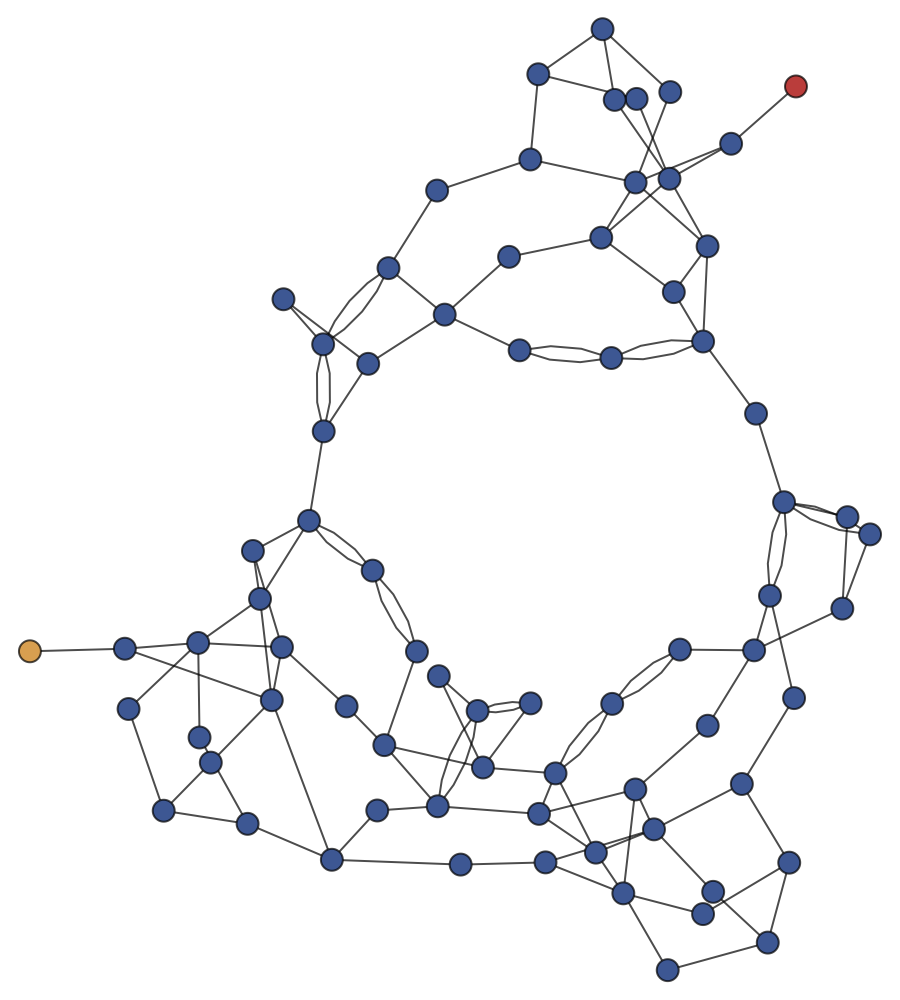}}}
    \hspace{2cm}
    \subfloat[The bridge-block tree of $G$, i.e., the reduced graph $G_{\calBB(G)}$ corresponding to the bridge-block decomposition of the network in (a)]{\makebox[1.2\width]{\includegraphics[scale=0.39]{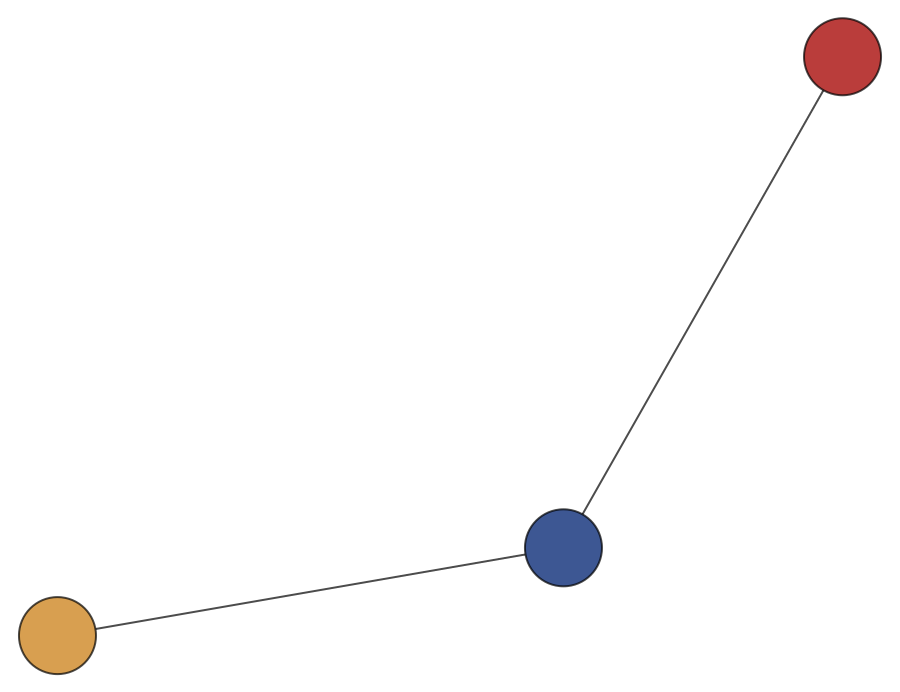}}}
    \vspace{-0.3cm}\\
    \subfloat[The biggest bridge-block of the network in (a) is split into $b=4$ clusters by solving an OBI-$4$ problem, obtaining a finer partition $\calP^*$. The cross-edges of $\calP^*$ inside the original bridge-block are highlighted in red.]{\makebox[1.2\width]{\includegraphics[scale=0.39]{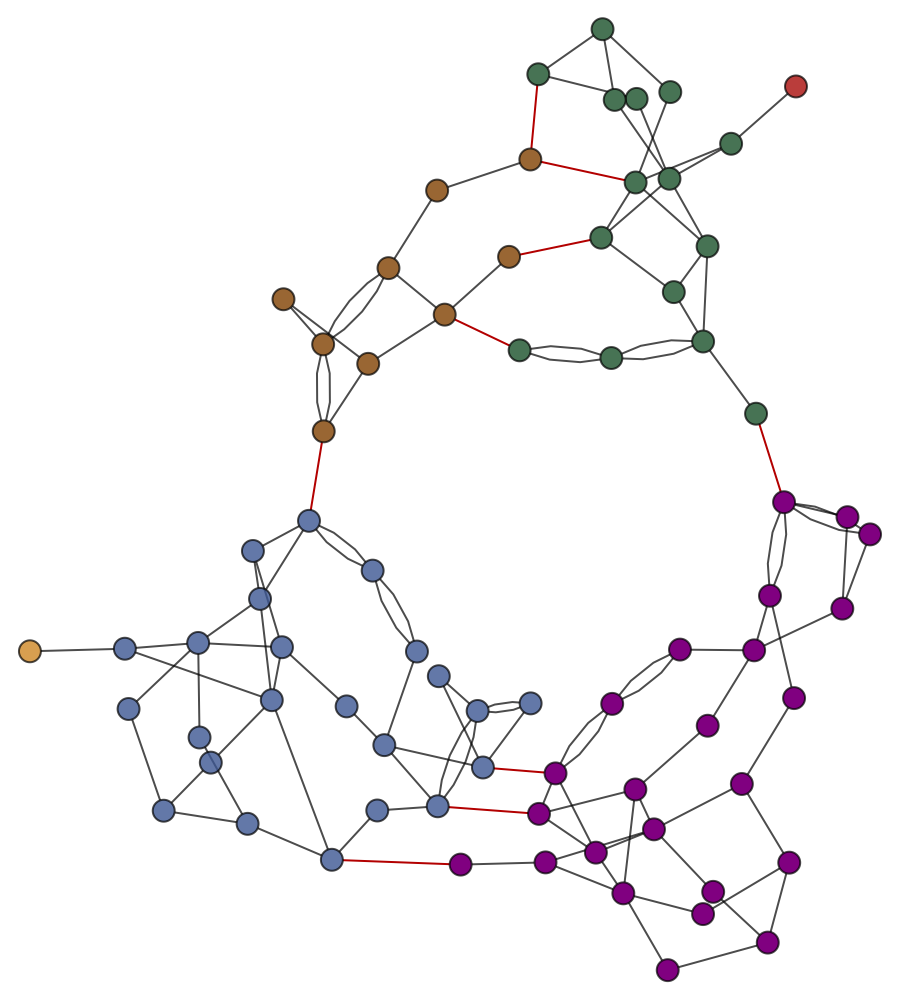}}}
    \hspace{2cm}
    \subfloat[The reduced graph $G_{\calP^*}$ corresponding to the optimal partition $\calP^*$ in (c), with the set $E_c(\calP^*)$ of its cross-edges highlighted in red. Each feasible set of switching actions corresponds to a spanning tree of $G_{\calP^*}$.]{\makebox[1.2\width]{\includegraphics[scale=0.39]{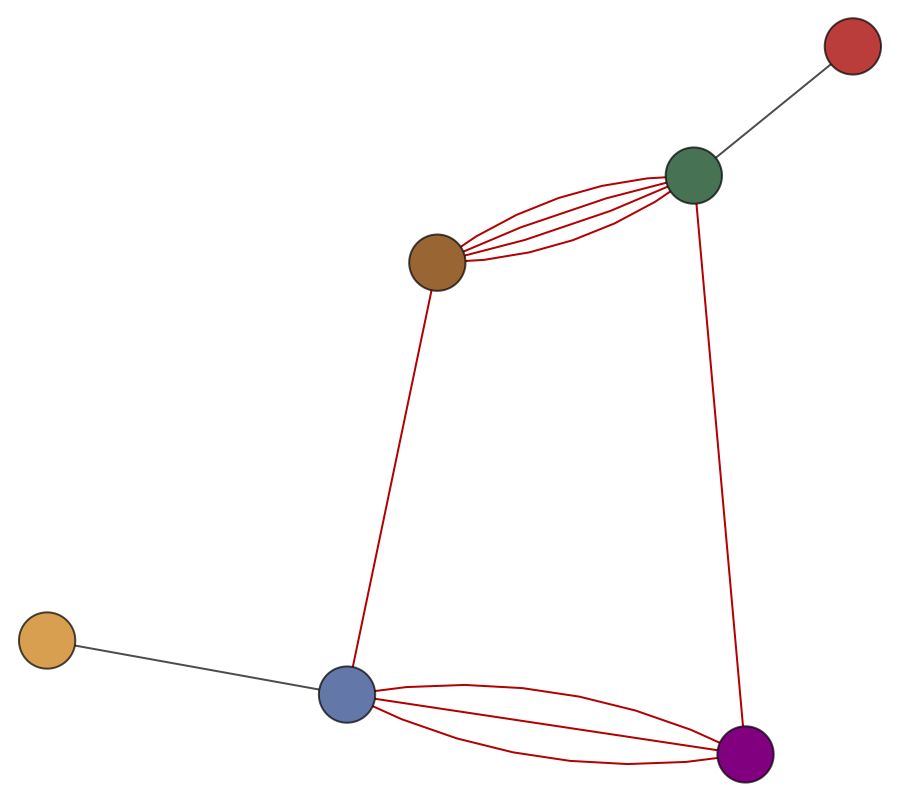}}}
    \vspace{-0.3cm}\\
    \subfloat[The network $G'$ obtained after switching off the optimal subset $\calE^*$ consisting of $6$ cross-edges. The vertex colors display its bridge-block decomposition $\calBB(G')$, which is finer than the original one $\calBB(G)$.]{\makebox[1.2\width]{\includegraphics[scale=0.39]{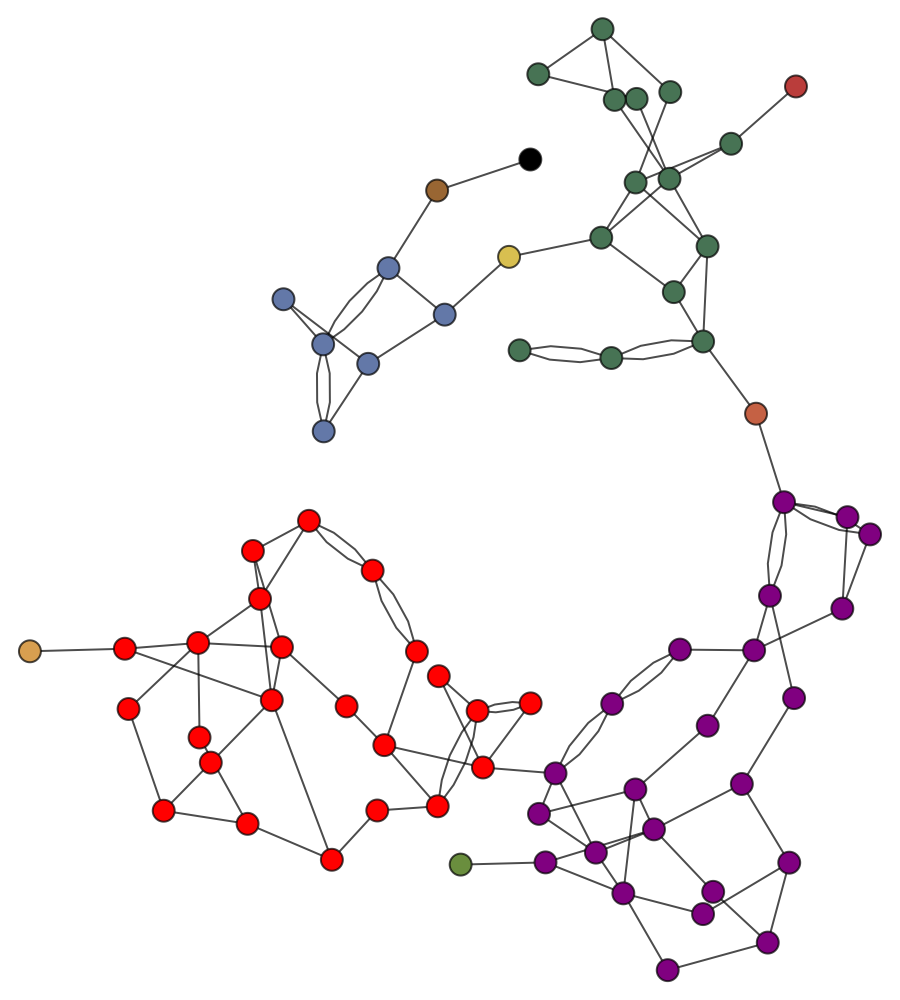}}}
    \hspace{2cm}
    \subfloat[The bridge-block tree of $G'$, i.e., the reduced graph $G'_{\calBB(G')}$ corresponding to the new bridge-block decomposition $\calBB(G')$ of the post-switching network $G'$ in (e).]{\makebox[1.2\width]{\includegraphics[scale=0.39]{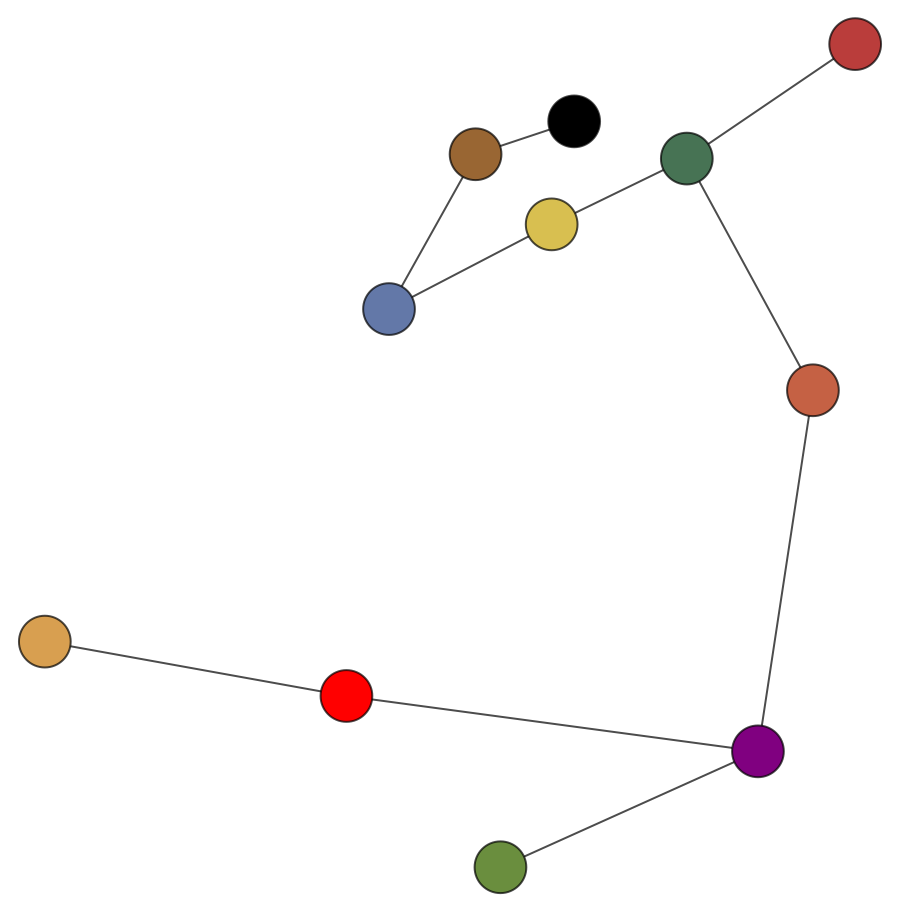}}}
    \caption{One-shot algorithm at work. Even if the goal was to split the biggest bridge-block into $b=4$ smaller ones, the algorithm returned a finer bridge-block decomposition consisting of $11$ bridge-blocks.}
    \label{fig:oneshotalgorithm}
\end{figure}
\FloatBarrier

Solving the one-shot bridge-block decomposition refinement presents several issues. Firstly, it is hard to determine a priori a reasonable target number $b$ of bridge-blocks, which is required as input of the algorithm: a small value would only yield a minor refinement of the network bridge-block decomposition, while a large one might yield only network with congested lines, i.e., $\gamma^* >1$.
Secondly, solving the OBS at the step 3 might be computationally very expensive. As mentioned earlier, the feasible sets $\calE \subset E_c(\calP)$ for the OBS problem~\eqref{eq:optswitch} are in one-to-one correspondence with the spanning trees of the reduced multi-graph $G_\calP$, which has at least $b$ nodes. 

Unfortunately, the OBS problem is not a minimum spanning tree problem, since the network congestion level appears in the objective function. For every spanning tree $T$ of $G_\calP$, we need to recalculate the power flows $\ff^\calE$ on the network $G^\calE$ obtained from the original one by removing the cross-edges that do not belong to that spanning tree, i.e., $\calE=E_c(\calP) \setminus T$. In view of this fact, the bottleneck of the above algorithm is precisely finding all such spanning trees of $G_\calP$. The time required to generate all spanning trees of a graph $G$ can be expressed as $O(g(b,|E_c(\calP)|) + \tau(G_\calP) h(b,|E_c(\calP)|)$ where $g,h$ are functions that are specific to the algorithm under consideration (see~\cite{Chakraborty2019} for a detailed overview), but it is largely dominated by $\tau(G_\calP)$, the number of spanning trees of $G_\calP$, that increases exponentially in the graph size $b$. 

This latter issue disappears if the selected target partition $\calP$ consists of only two clusters, say $\calV_1$ and $\calV_2$, as the reduced multi-graph $G_\calP$ has a number of spanning trees at most linear in the number of edges of the original network $G$. Indeed, in this case any spanning tree for $G_\calP$ consists of a single cross-edge in $E_c(\calP)$. Therefore, selecting a subset $\calE \subset E_c(\calP)$ of cross-edges between $\calV_1$ and $\calV_2$ feasible for the OBS problem is equivalent to choosing the unique cross-edge $e \in E_c(\calP)$ to keep, which become the unique bridge of the reduced graph $(G^\calE)_\calP$. Thus the number of feasible sets $\calE$ for the OBS problem is precisely equal to that of the cross-edges, and, trivially, $|E_c(\calP)| \leq |E|$.
Restricting ourselves to bi-partitions $\calP \in \Pi_2(G)$, the OBS problem~\eqref{eq:optswitch} becomes
\begin{equation}
	\min_{e \in E_c(\calP)} \gamma(E_c(\calP) \bs \{e\}) .
\label{eq:optswitch_tworegions}
\end{equation}
where the constraint~\eqref{eq:optswitch_constraint2} disappears because $(G^\calE)_\calP$ is always trivially a tree in this case. This optimization problem is much simpler and faster to solve than the original one: instead of optimizing over a complex combinatorial object, namely the family of spanning trees of the reduced graph $G_\calP$, the feasible set is just the collection of cross-edges $E_c(\calP)$.

This fact suggests the idea that it may be computationally more efficient to refine the bridge-block decomposition progressively, splitting the current largest bridge-block into two smaller bridge-blocks in a recursive fashion. We thus propose a second algorithm that leverages this idea. 
Besides drastically simplifying the structure of the OBS problem to be solved, this recursive approach embeds a natural stopping rule for the bridge-block decomposition refinement. As we hinted at already, it is impossible to estimate a priori how much the bridge-block decomposition of a power network can be refined without creating congestion on any line. In this respect, a procedure that increases the number of blocks by one at each step while monitoring the network congestion level can help find the best trade-off between the number of bridge-blocks and network congestion level. We introduce the following stopping rule. We keep refining the bridge-block decomposition splitting the largest bridge-block in a recursive fashion as long as:
\begin{itemize}
    \item[(i)] the current maximum number of splits is less than a predefined threshold $i_\mathrm{max} \in \N$, and
    \item[(ii)] the obtained network congestion level $\gamma^*$ is not larger than a threshold $\delta \in (0,1]$ describing the maximum tolerable congestion that the network operator can set.
\end{itemize}
The high-level structure of the proposed recursive procedure is shown in Algorithm~\ref{alg:recursiverefinement}.

\begin{algorithm}[!ht]
\SetAlgoLined
\SetKwInOut{Input}{Input}\SetKwInOut{Output}{Output}
\Input{A power network $G=(V,E,\bm{b})$, a maximum number of iterations $i_\mathrm{max}$, and a congestion threshold $\delta$}
\BlankLine
\DontPrintSemicolon
Set $i=0$;\\
Set $E_0=E$ and $G_0=(V,E_0)$;\\
Calculate the line flows $\ff_0$ and the current network congestion level $\gamma_0^*:=\gamma(\emptyset)$;\\
\While{$\gamma^*_{i} < \delta$ \textbf{\textup{and}} $i < i_\mathrm{max}$}{
Calculate the current bridge-block decomposition $\calBB(G_i)$;\\
Consider in isolation the largest bridge-block in size of the current bridge-block decomposition, say $\calV \in \calBB(G_i)$;\\
Solve a OBI-2 problem for $\calV$ in isolation to find its optimal bipartition $\calP^*_i$;\\
Solve the simplified OBS problem~\eqref{eq:optswitch_tworegions} for $\calV$ in isolation with target bipartition $\calP^*_i$, obtaining an optimal subset $\calE^*_i \subset E_c(\calP^*_i)$ and the new network congestion level $\gamma^*_i=\gamma(\calE^*_i)$;\\
Calculate the new set of active lines to be $E_{i+1} := E_{i} \setminus \calE^*_{i}$ and set $G_{i+1}=(V,E_{i+1})$;\\
Increase $i$ by $1$;\\
}
\Output{$i$, $G_i$, $\calBB(G_i)$, $\gamma^*_i$}
\caption{Recursive bridge-block decomposition refinement}\label{alg:recursiverefinement}
\end{algorithm}
Note that in Step 6 we select the current largest bridge-block to be further split, so that all the bridge-blocks progressively become smaller, indirectly ensuring that, when the algorithm stops, the sizes of the resulting bridge-blocks are similar. However, we could tweak the algorithm to select either (i) the bridge-block with the lowest congestion level (inside which is probably safer to switch off lines without creating overloads) or (ii) the one with the highest potential to be clustered in terms of modularity score. Both these approaches, however, could possibly lead to ineffective bridge-block decompositions, where there are still very large bridge-blocks and several small (trivial) ones.

The OPS problem needs to be solved several times in this recursive variant (and not only once as prescribed in the one-shot variant, Algorithm~\ref{alg:oneshotrefinement}), which means that Algorithm~\ref{alg:recursiverefinement} requires possibly many more power flows calculations. However, a remarkable feature of the recursive approach is that all these power flow redistribution calculations can be done ``locally'' (i.e., inside the bridge-block that is being split) and thus become progressively faster as the size of the largest bridge-block decreases. More specifically, let $\calV$ be the bridge-block selected to be further split in Step 6 of Algorithm~\ref{alg:recursiverefinement}. By virtue of Proposition~\ref{prop:glodf_non_cut}, any switching action (which is effectively equivalent to line failure in this context) inside the bridge-block $\calV$ does not affect the line flows in any other bridge-block $\calV' \neq \calV$ of the current bridge-block decomposition. Therefore, only the line flows inside $\calV_i$ needs to be recalculated and this can be done by solving the DC power flow equations~\eqref{eq:flowsdc_model} only for the subgraph induced by $\calV$ and the corresponding local power injections. The only caveat is that we need to adjust the net power at all the nodes of $\calV$ that are endpoints of bridges to account for the outgoing flow. 

Adopting this recursive splitting strategy, the bridge-blocks become smaller at every iteration, which in turn means that fewer flows (and congestion levels) need to be recalculated at each step when solving the OBS problem~\eqref{eq:optswitch_tworegions}, making it increasingly faster to solve.

The recursive structure also makes Step 10 of Algorithm~\ref{alg:recursiverefinement} faster: the bridge-block decomposition does not have to be recomputed globally, since all the bridge-blocks but the one considered in this iteration will be unchanged and we only need to identify the new bridge-blocks created locally. We remark that this step cannot be skipped as the optimal set of switch action may accidentally split the considered bridge-block into more than two bridge-blocks.

Figs.~\ref{fig:recursivealgorithm_1} and~\ref{fig:recursivealgorithm_2} visualize the various steps that  Algorithm~\ref{alg:recursiverefinement} takes when set with $i_\mathrm{max}=2$ and applied to the same network considered in Fig.~\ref{fig:oneshotalgorithm}, namely the IEEE 73-bus network. It is clear that the OBS problem is essentially trivial for the recursive algorithm in both steps, see~Figs.~\ref{fig:recursivealgorithm_1}(d) and~\ref{fig:recursivealgorithm_2}(b), while it was more involved for the one-shot algorithm, cf.~Fig.~\ref{fig:oneshotalgorithm}(d).

\begin{figure}[!ht]
    \centering
    \vspace{-0.5cm}
    \subfloat[The graph $G_0$ corresponding to IEEE 73 network with its bridge-block decomposition $\calBB(G_0)$ consisting of three bridge-blocks, two of which are trivial.]{\makebox[1.25\width]{\includegraphics[scale=0.39]{ieee73_1_original.png}}}
    \hspace{1cm}
    \subfloat[The bridge-block tree of $G_0$, i.e., the reduced graph $(G_0)_{\calBB(G_0)}$ corresponding to the bridge-block decomposition of the network in (a)]{\makebox[1.25\width]{\includegraphics[scale=0.39]{ieee73_1_bbd.png}}}
    \vspace{-0.2cm}\\
    \subfloat[The biggest bridge-block of the network in (a) is split into 2 clusters by solving a simplified OBI-$2$ problem, obtaining a finer partition $\calP^*_0$. The cross-edges of $\calP^*_0$ inside the original bridge-block are highlighted in red.]{\makebox[1.25\width]{\includegraphics[scale=0.39]{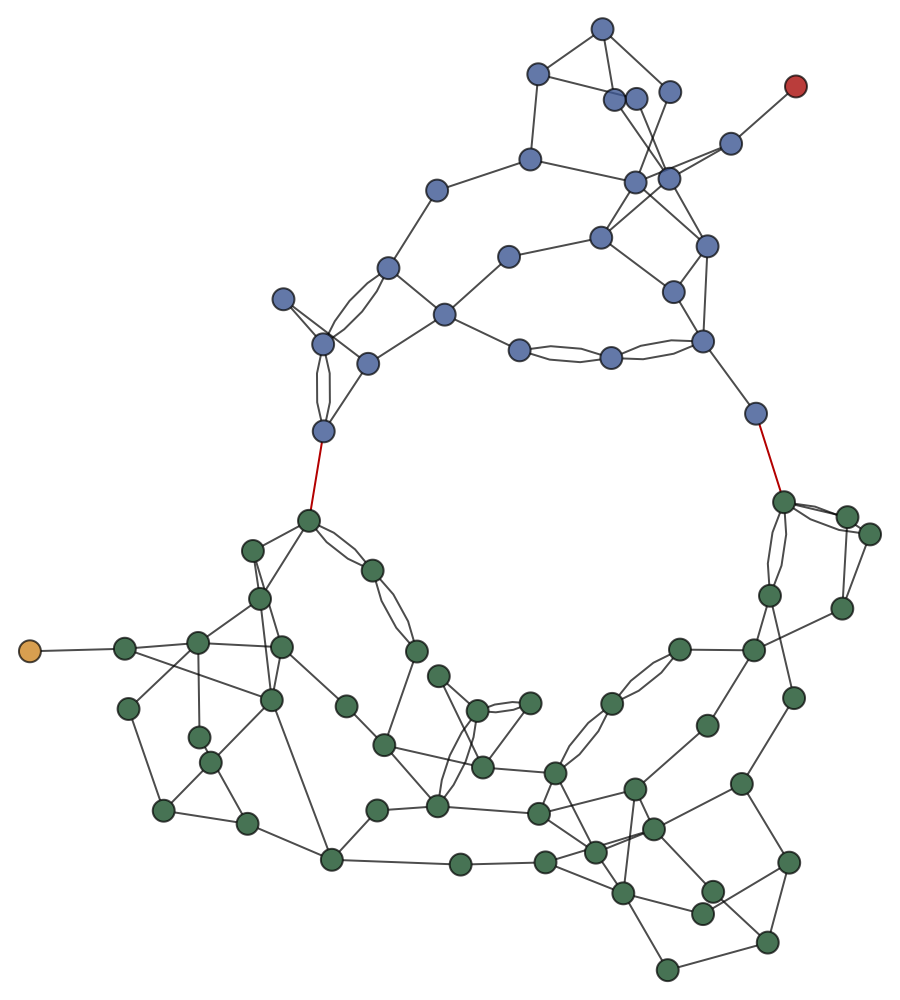}}}
    \hspace{1cm}
    \subfloat[The reduced graph $G_{\calP^*_0}$ corresponding to the optimal partition $\calP^*_0$ in (c), with the set $E_c(\calP^*_0)$ of its cross-edges highlighted in red. The spanning trees of $G_{\calP^*_0}$ (and thus the feasible switching actions) to consider do not grow exponentially with the size of $G_{\calP^*_0}$, being in one-to-one correspondence with $E_c(\calP^*_0)$.]{\makebox[1.25\width]{\includegraphics[scale=0.39]{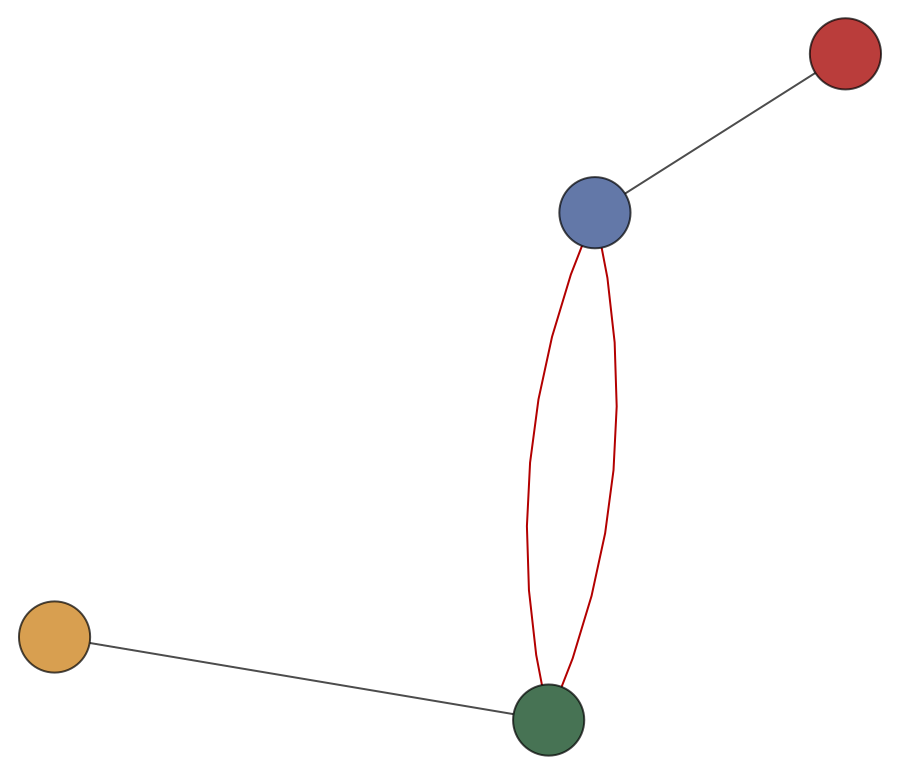}}}
    \vspace{-0.2cm}\\
    \subfloat[The network $G_1$ obtained after switching off the cross-edges in $\calE^*_0$. The vertex colors display its bridge-block decomposition $\calBB(G_1)$, which is finer than the original one $\calBB(G_0)$.]{\makebox[1.25\width]{\includegraphics[scale=0.39]{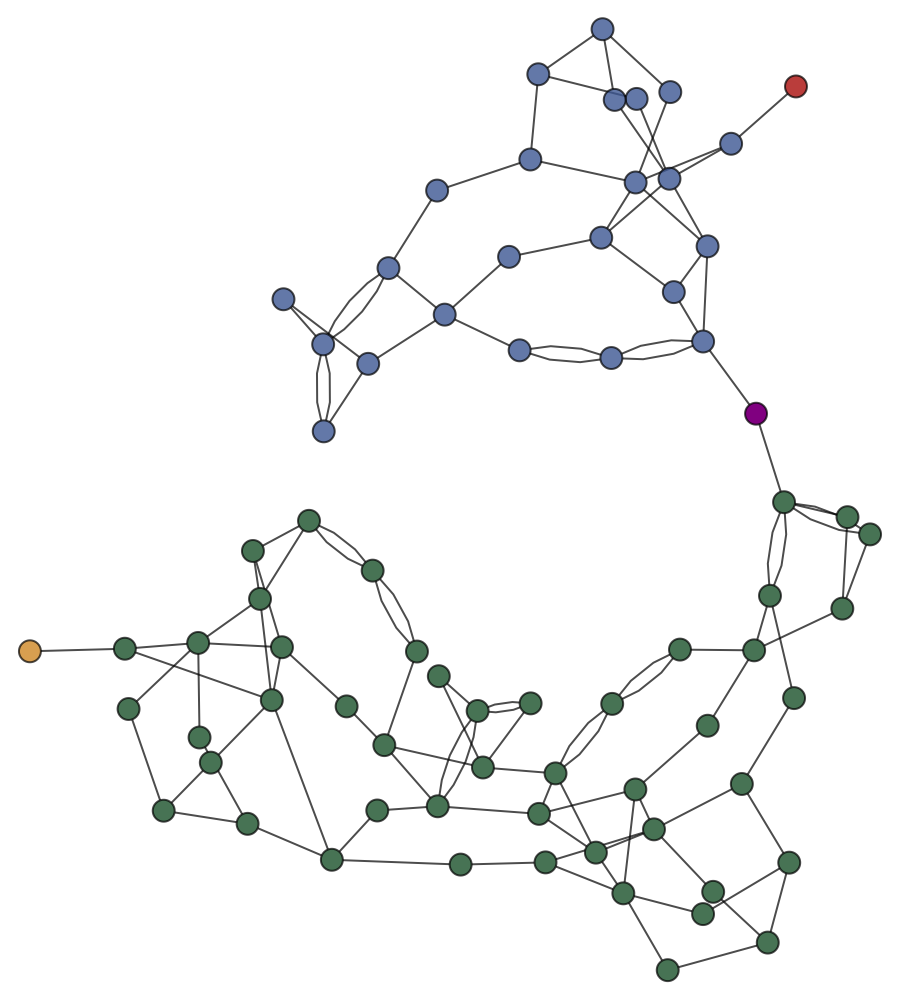}}}
    \hspace{1cm}
    \subfloat[The bridge-block tree of $G_1$, i.e., the reduced graph $(G_1)_{\calBB(G_1)}$ corresponding to the new bridge-block decomposition of the post-switching network $G_1$ in (e).]{\makebox[1.25\width]{\includegraphics[scale=0.39]{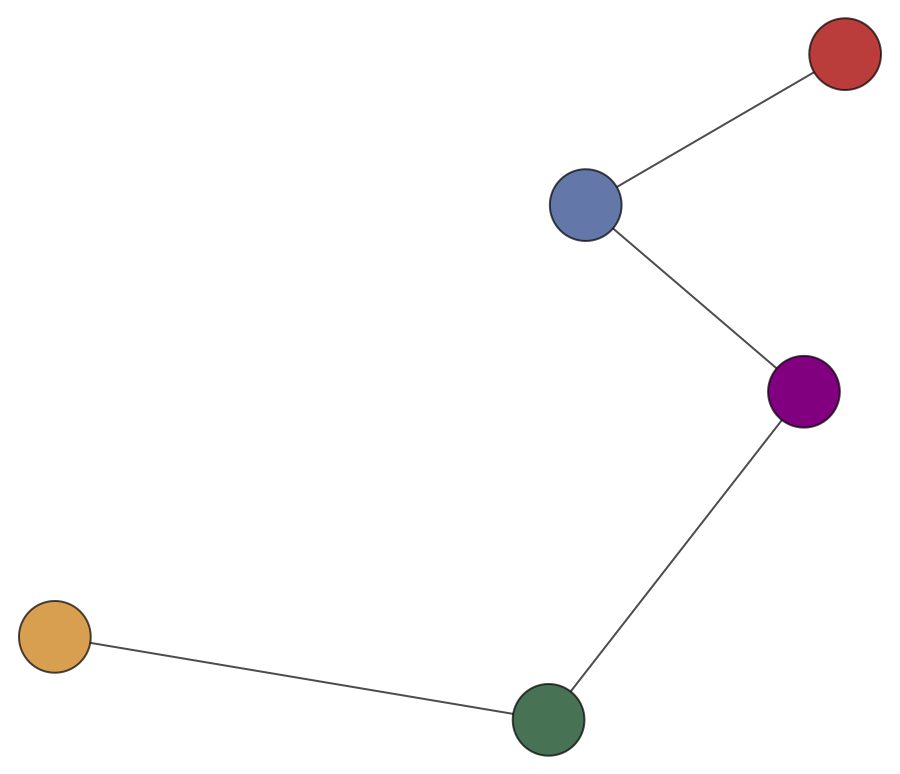}}}
    \caption{Recursive algorithm at work -- first iteration $i=1$}
    \label{fig:recursivealgorithm_1}
\end{figure}
\FloatBarrier
\begin{figure}[!h]
    \centering   
    \subfloat[The biggest bridge-block of the network $G_1$ in Fig.~\ref{fig:recursivealgorithm_1}(e) is split into 2 clusters by solving a simplified OBI-$2$ problem, obtaining a finer partition $\calP^*_1$. The cross-edges of $\calP^*_1$ inside the original bridge-block are highlighted in red.]{\makebox[1.25\width]{\includegraphics[scale=0.39]{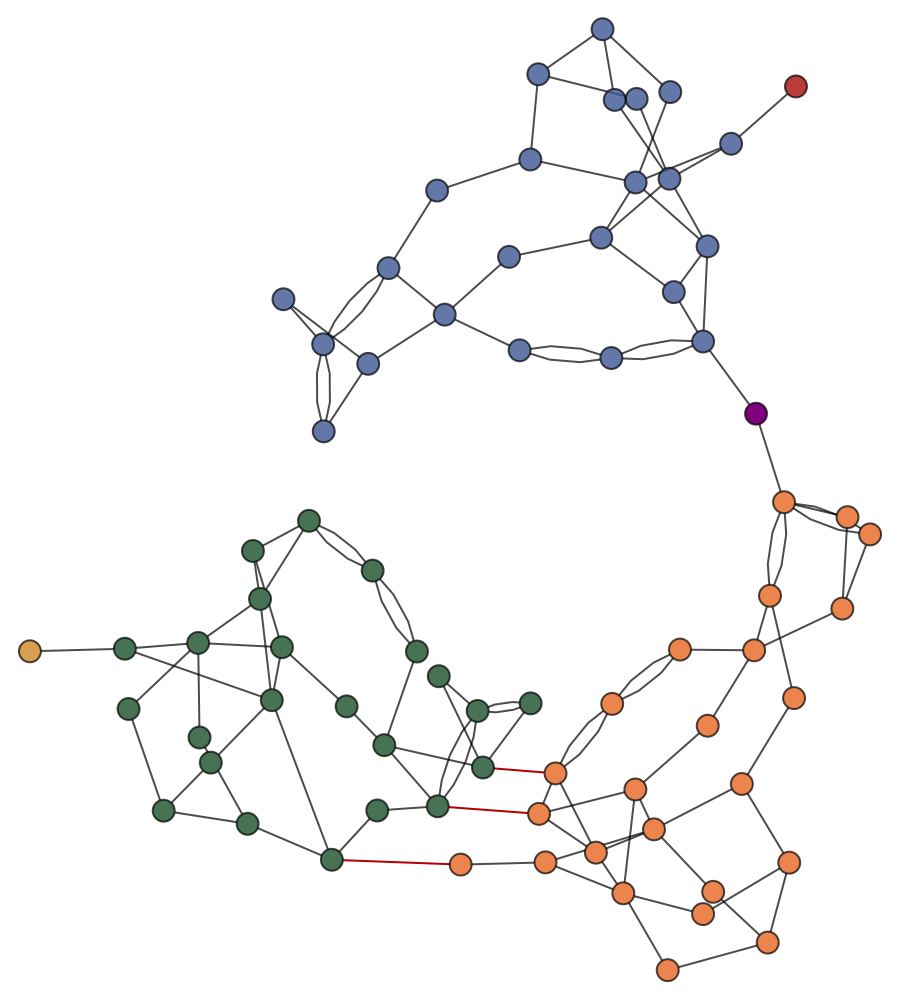}}}
    \hspace{1cm}
    \subfloat[The reduced graph $(G_1)_{\calBB(G_1)}$ corresponding to the bridge-block decomposition of the network in (a)]{\makebox[1.25\width]{\includegraphics[scale=0.39]{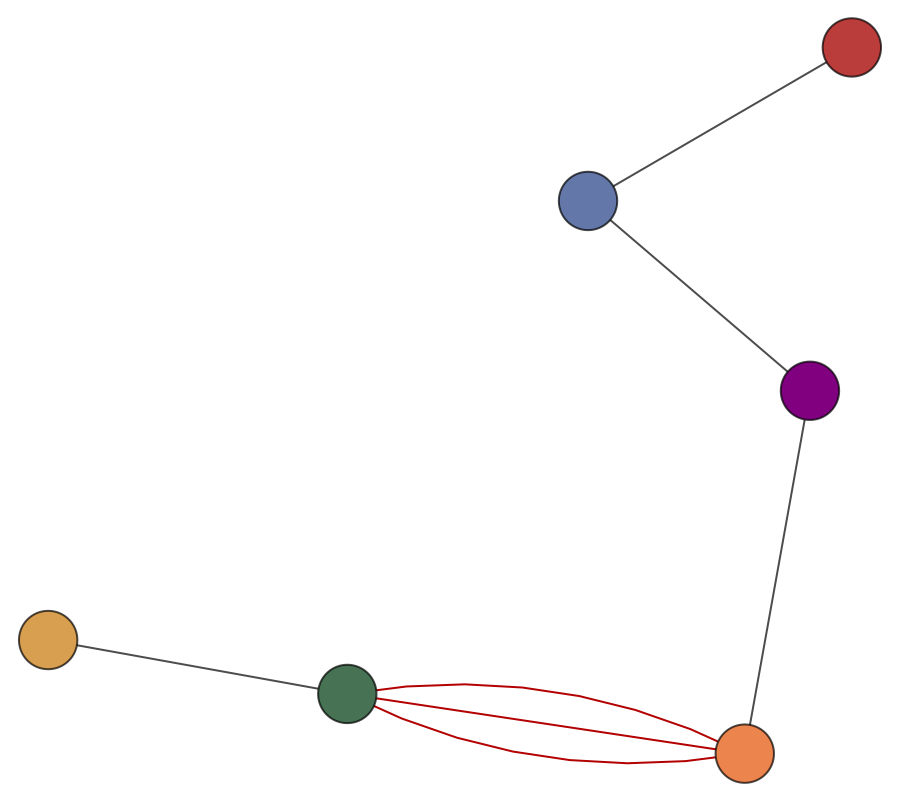}}}
    \\
    \subfloat[The network $G_2$ obtained after switching off the cross-edges in $\calE^*_1$. The vertex colors display its bridge-block decomposition $\calBB(G_2)$, which is finer than the previous one $\calBB(G_1)$]{\makebox[1.25\width]{\includegraphics[scale=0.39]{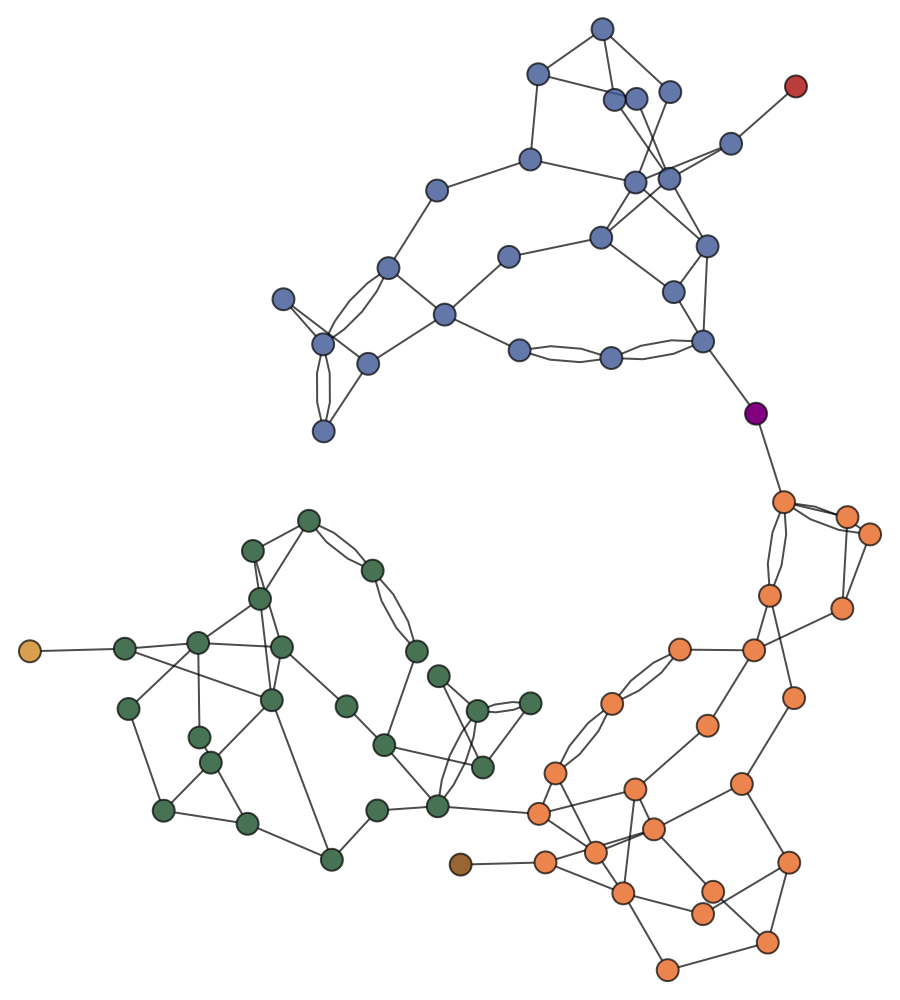}}}
    \hspace{1cm}
    \subfloat[The bridge-block tree of $G_2$, i.e., the reduced graph $(G_2)_{\calBB(G_2)}$ corresponding to the new bridge-block decomposition of the post-switching network $G_2$ in (c).]{\makebox[1.25\width]{\includegraphics[scale=0.39]{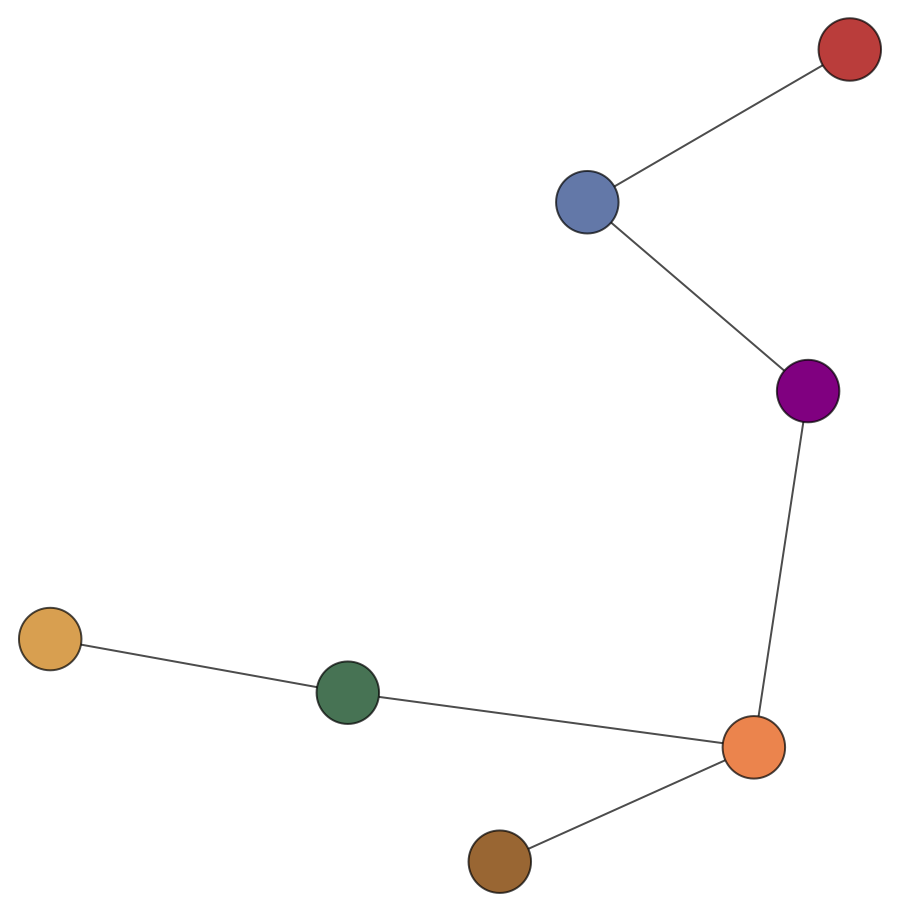}}}
    \caption{Recursive algorithm at work -- second iteration $i=2$}
    \label{fig:recursivealgorithm_2}
\end{figure}
\FloatBarrier
%
%
Both the one-shot and the recursive algorithms focus on the scenario where the power injections are fixed. However, in practical settings the network operator could \textit{jointly} reconfigure the network topology by means of switching actions \textit{and} adjust (some of) the power injections, either by shedding load or by adjusting generator levels. In this scenario, the OBS problem could potentially find a solution with an even lower network congestion level, but analyzing this joint optimization problem or proposing heuristics leveraging this additional injection flexibility is beyond the scope of this paper. 

We remark that one such possible strategy, in the spirit of~\cite{Liang2021}, is aiming for local power adjustments and only change the power injections at the endpoints of each line that is being switched off by an amount exactly equal to its original flow on that line. This strategy would guarantee that the power flows on the entire surviving network will be unchanged and no congestion will be created. However, there might not be enough flexibility in terms of injection adjustment in each pair of endpoints. Alternatively, one can solve a global OPF problem for the whole network after the switching actions, hence ensuring that no line will be overloaded. Even better results could be obtained when considering the line switching and power injection adjustment jointly in the same optimization problem. 
\FloatBarrier

\section{Numerical examples}
\label{sec:numerics}
This section is devoted to the implementation of the bridge-block decomposition refinement strategy introduced in Section~\ref{sec:optimization}. In Subsection~\ref{sub:networkanalysis} we first review  the structural properties of several power grids used in the literature as test networks, from which it becomes evident the need of finer and more homogeneous bridge-blocks for these networks. In Subsection~\ref{sub:modnum} we introduce two fast approximations algorithms to find locally optimal solutions for the OBI problem~\eqref{eq:maxmodprob} and compare their performance on IEEE test networks in terms of the clusters' size and cross-edges. Then, in Subsection~\ref{sub:switchingnum}, we compare the performances of the one-shot algorithm and of its recursive variant for the same test networks.


\subsection{Preliminary analysis of power systems topologies}
\label{sub:networkanalysis}
In order to show the potential benefits of applying our method for improving network reliability, it is important to first understand the relevant structural properties of typical power networks, in particular their bridge-block decomposition. In Table~\ref{tab:networkstatistics} below we report statistics about the original topologies of various standard test networks taken from~\cite{pglibopf,pglibopf_git}. For each of these networks, we find all its bridges and calculate its bridge-block decomposition. Modulo some minor variations, all the networks exhibit the same topological features, namely:
\begin{itemize}
    \item They have only one giant bridge-block that comprises a vast majority of the network nodes;
    \item They have substantial number of edges that are bridges (on average roughly $18\%$), each of them connecting a trivial bridge-block (consisting of one or two nodes) to the aforementioned giant bridge-block.
\end{itemize}
These features are also easily recognizable for the IEEE 118-bus network displayed in Fig.~\ref{fig:influencegraph_pre}. From these statistics it is clear that most power networks have an essentially trivial bridge-block decomposition and would thus perform very poorly in terms of failure localization. In particular, a line failure inside the unique giant bridge-block could propagate to roughly $80\%$ of the other lines (i.e., those in the same bridge-block).
Moreover, we note that all these networks already have a surprisingly large fraction (on average $\sim 18\%$) of lines that are bridges, which the switching actions prescribed by our method would increase only marginally.

\begin{table}[!ht]
    \centering
\begin{tabular}{c|r|r|C{2.35cm}|r|C{3.1cm}}
Test network case & \# of edges & \# of bridges & \% of edges that are bridges & \# of bridge-blocks &  Size(s) of non-trivial bridge-blocks \\
\hline \hline
case14\_ieee             &     20 &        1 &  5.00    &               2 &             13 \\
case30\_ieee             &     41 &        3 &  7.32    &               4 &             27 \\
case39\_epri             &     46 &       11 &  23.91   &               12 &             28 \\
case57\_ieee             &     80 &        1 &  1.25    &               2 &             56 \\
case73\_ieee\_rts        &    120 &        2 &  1.67    &               3 &             71 \\
case89\_pegase           &    210 &       16 &  7.62    &               17 &             73 \\
case118\_ieee            &    186 &        9 &  4.84    &               10 &             109 \\
case162\_ieee\_dtc       &    284 &       12 &  4.22    &             13 &             150 \\
case179\_goc             &    263 &       43 &  16.35   &             44 &             136 \\
case200\_activ           &    245 &       72 &  29.39   &             73 &             128 \\
case240\_pserc           &    448 &       58 &  12.95   &             59 &             182 \\
case300\_ieee            &    411 &       89 &  21.65   &             90 &             206, 3, 3 \\
case588\_sdet            &    686 &      229 &  33.38   &            230 &             357 \\
case793\_goc             &    913 &      290 &  31.76   &            291 &             500 \\
case1354\_pegase         &   1991 &      561 &  28.18   &            562 &             791 \\
case1888\_rte            &   2531 &      964 &  38.09   &            965 &             918, 5 \\
case2000\_goc            &   3633 &      445 &  12.25   &          446 &             1555 \\
case2736sp\_k            &   3273 &      627 &  19.16   &            628 &             2109 \\
case2737sop\_k           &   3274 &      627 &  19.15   &            628 &             2110 \\
case2746wp\_k            &   3281 &      637 &  19.41   &          638 &             2109 \\
case2746wop\_k           &   3309 &      607 &  18.34   &           608 &             2139 \\
case2848\_rte            &   3776 &     1410 &  37.34   &           1411 &             1421, 7, 5, 3 \\
case2869\_pegase         &   4582 &      778 &  16.98   &            779 &             2088 \\
case3120sp\_k            &   3693 &      731 &  19.79   &            732 &             2382, 8 \\
case3375wp\_k            &   4161 &      826 &  19.85   &            827 &             2536, 3 \\
case9241\_pegase         &  16049 &     1665 &  10.37   &           1666 &             7558, 7, 5, 3
\end{tabular}
    \caption{Summary of bridge and bridge-block statistics for various test networks. A bridge-block is considered to be non-trivial is it consists of strictly more than two nodes.}
    \label{tab:networkstatistics}
\end{table}

\subsection{Bridge-block candidate identification}
\label{sub:modnum}
The OBI problem that we introduced in Section~\ref{sub:mod} to find good candidate bridge-blocks is a variant of the well-known weighted network modularity maximization. However, this optimization problem cannot be solved exactly, being NP-hard, especially for large power networks and we need to resort to locally optimal partition. Before presenting our numerical findings, we first briefly review the two approximation algorithms that we implemented, namely the \textit{fastgreedy method} and the \textit{spectral clustering}.\\

\textbf{Fastgreedy.} A very fast approximation algorithm for optimizing modularity on large networks was proposed by~\cite{Clauset2004}. It uses a bottom-up hierarchical approach that tries to optimize the modularity in a greedy manner. Initially, every vertex belongs to a separate community, and communities are merged iteratively such that each merger is locally optimal (i.e., yields the largest increase in the current value of modularity). We use the Python library $\mathrm{igraph}$ and the built-in implementation of $\mathrm{fastgreedy}$ modularity optimization algorithm described in~\cite{Clauset2004}, which also implements some improvements proposed by~\cite{Wakita2007}.\\

\textbf{Spectral clustering.} Spectral clustering was first used in the context of modularity maximization by~\cite{Newman2006a,Newman2006b} and has become one of the most popular modern clustering algorithms, being simple to implement and very efficient to solve using standard linear algebra software. In our implementation of this method we use two matrices that simultaneously capture the network structure and flow pattern, the modularity matrix and the weighted Laplacian matrix. They can be both calculated from the graph $G=(V,E)$ with edge weights equal the absolute power flows $\{|f_{ij}|\}_{(i,j)\in \calE}$ as follows:
\begin{itemize}
    \item the \textit{modularity matrix} $B$, whose $(i,j)$ entry is $B_{i,j}= |f_{ij}|-\frac {F_{i} F_{j}}{2 M}$, and
    \item the \textit{weighted Laplacian matrix} $L$, whose entries are $L_{i,j}= -|f_{ij}|$ if $i\neq j$ and $L_{i,i}=-\sum_{j \neq i} L_{i,j}$.
\end{itemize} 
In both cases we choose to work with their normalized counterparts, namely
$$B_n := W^{-1/2} B W^{-1/2} \qquad \text{ and } \qquad L_n := W^{-1/2} L W^{-1/2},$$
where $W=\mathrm{diag}(F_1,\dots,F_n)$, since, especially in graphs of large size or with very heterogeneous weights, it has been argued that the spectral clustering method is more effective and computationally more stable when using the normalized version of these matrices, see~\cite{Bolla2011,vonLuxburg2007}.

The spectral clustering performed on the normalized Laplacian matrix $L_n$ approximates the optimal solution of the NP-hard problem known as \textit{Normalized Cut (NCut)}. Working with the normalized modularity matrix $B_n$ means that the spectral clustering method approximates the partition that maximizes a renormalized version $Q_n$ of the network modularity, namely
\[
	Q_n(\calP) := \frac {1}{2 M} \sum_{r=1}^{|\calP|} \frac{1}{\sum_{k \in \calV_r} |F_{k}|} \sum \limits _{i,j \in \calV_r}  \left ( |f_{ij}| - \frac {F_{i} F_{j}}{2 M} \right ).
\]
If we fix the target number $b$ of clusters the optimal normalized modularity clustering is equal to the normalized cut clustering, in the sense that $\arg\max_\calP Q_n(\calP)= \arg\min_\calP \mathrm{NCut}(\calP)$. This follows from the identity
\[
	Q_n(\calP) = \frac{b-1}{2M} - \frac{1}{2M} \mathrm{NCut}(\calP),
\]
proved in~\cite{Bolla2011,YuDing2010} under the assumption that the target number $b$ of clusters is fixed. This means that the spectral clustering method performed on either matrix $B_n$ and $L_n$ solve approximately the same optimization problem. However, having different eigensystems, they often yield to slightly different results. For this reason, we consider both of these two variants of the spectral clustering approximation method in the following subsection and label them as \textit{Spectral $L_n$} and \textit{Spectral $B_n$}.

We remark that the authors in~\cite{Bialek2014} also adopt the normalized-Laplacian-based spectral clustering to power systems, proposing several choices for the edge weights, including the one we adopted in this paper, i.e., the absolute value of power flows. In the same paper, some heuristics are outlined to decide the optimal number of clusters in which to partition the network, but these leverage the eigensystem structure of the Laplacian matrix, rather than being informed by a specific target number of clusters $b$ like in our method.

\subsubsection{Performance comparison}

We compare the performance of the three aforementioned clustering methods (FastGreedy, Spectral $L_n$ and Spectral $B_n$) to optimally partition some standard IEEE test cases available in the pglib-opf library~\cite{pglibopf,pglibopf_git}. We consider four networks of various sizes, ranging from small (39 nodes) to large (2737 nodes). 

We first consider the quality of the network partitions returned by three methods by looking exclusively at graph-related quantities. In the following tables, we report for different target number $b$ of clusters the normalized modularity score $Q_n$ (the higher the better) of the partition returned by each algorithm, as well as the sizes of the $b$ clusters and the fraction of edges that are cross-edges between them. We further report the number of cross-edges to be switched off to obtained a bridge-block decomposition with the identified clusters as bridge-blocks. Recall that, as we argued in Section~\ref{sub:mod}, it is highly desirable to have a small number of cross-edges between the identified clusters, since it dramatically reduces the feasible sets of switching actions for the OBS problem.

\begin{table}[!ht]
\centering
\begin{tabular}{c|C{1.1cm}|r|C{0.8cm}|C{1.7cm}|C{2cm}|C{3.8cm}}
\toprule
Clustering method & runtime (sec) & $b$ & $Q_n$ & fraction of cross-edges & \#lines to be switched off & resulting partition of the largest bridge-block\\
\hline \hline
Spectral $L_n$  &   0.0192 &  2 &   0.379 &               0.0869 &                    3 &                   [10, 18] \\
Spectral $B_n$ &   0.0186 &  2 &    0.406 &               0.0652 &                    2 &                   [11, 17] \\
FastGreedy  &  0.0016 &  2 &    0.406 &               0.0652 &                    2 &                   [11, 17] \\
\hline
Spectral $L_n$  &   0.0247&  3 &   0.486 &                0.152 &                    5 &                 [9, 9, 10] \\
Spectral $B_n$ &    0.0210 &  3 &   0.501 &                0.130 &                    4 &                 [8, 9, 11] \\
FastGreedy  &  0.0018 &  3 &   0.524 &                0.109 &                    3 &                [7, 10, 11] \\
\hline
Spectral $L_n$  &   0.0306 &  4 &   0.513 &                0.173 &                    5 &               [4, 7, 8, 9] \\
Spectral $B_n$ &   0.0291 &  4 &   0.527 &                0.196 &                    5 &               [4, 7, 8, 9] \\
FastGreedy  &  0.0027 &  4 &   0.535 &                0.152 &                    4 &              [4, 7, 7, 10] \\
\end{tabular}
\caption{39 EPRI}
\end{table}

\begin{table}[!ht]
\centering
\begin{tabular}{c|C{1.1cm}|r|C{0.8cm}|C{1.7cm}|C{2cm}|C{3.8cm}}
Clustering method & runtime (sec) & $b$ & $Q_n$ & fraction of cross-edges & \#lines to be switched off & resulting partition of largest bridge-block\\
\hline \hline
Spectral $L_n$  &    0.0255 &  2 &   0.220 &               0.032 &                    5 &                   [28, 81] \\
Spectral $B_n$ &   0.0294 &  2 &   0.390 &               0.032 &                    5 &                   [40, 69] \\
FastGreedy  &  0.0032 &  2 &   0.427 &               0.075 &                   13 &                   [50, 59] \\
\hline
Spectral $L_n$  &    0.0338 &  3 &   0.490 &               0.053 &                    8 &               [27, 33, 49] \\
Spectral $B_n$ &   0.0341 &  3 &   0.518 &               0.059 &                    9 &               [29, 38, 42] \\
FastGreedy  &  0.0061 &  3 &   0.527 &                0.102 &                   17 &               [16, 34, 59] \\
\hline
Spectral $L_n$  &  0.0366 &  4 &    0.498 &               0.075 &                   11 &            [9, 19, 33, 48] \\
Spectral $B_n$ &  0.0477 &  4 &   0.540 &                0.102 &                   16 &           [18, 21, 29, 41] \\
FastGreedy  &   0.0381 &  4 &   0.620 &                0.140 &                   23 &           [16, 27, 32, 34] \\
\end{tabular}
\caption{IEEE 118}
\end{table}

\begin{table}[!ht]
\centering
\begin{tabular}{c|C{1.1cm}|r|C{0.8cm}|C{1.7cm}|C{2cm}|C{3.8cm}}
Clustering method & runtime (sec) & $b$ & $Q_n$ & fraction of cross-edges & \#lines to be switched off & resulting partition of largest bridge-block\\
\hline \hline
Spectral $L_n$  &   0.912 &  2 & 0 &             0.0008 &                    1 &                   [1, 917] \\
Spectral $B_n$ &  0.844 &  2 &     0.483 &               0.0107 &                   26 &                 [391, 527] \\
FastGreedy  &  0.443 &  2 &      0.489&               0.0130 &                   32 &                 [415, 503] \\
\hline
Spectral $L_n$  &   7.536 &  3 &   0.610 &               0.0217 &                   51 &            [168, 356, 394] \\
Spectral $B_n$ &   1.201 &  3 &   0.648 &               0.0194 &                   46 &            [273, 308, 337] \\
FastGreedy  &  0.461 &  3 &   0.617 &                0.0198 &                   48 &            [205, 298, 415] \\
\hline
Spectral $L_n$  &  8.167 &  4 &     0.709 &               0.0221 &                   52 &       [187, 189, 222, 320] \\
Spectral $B_n$ &  5.693 &  4 &   0.709 &               0.0205 &                   48 &       [157, 204, 261, 296] \\
FastGreedy  &  1.028 &  4 &   0.724 &               0.0217 &                   52 &       [153, 205, 262, 298] \\
\end{tabular}
\caption{1888rte}
\end{table}

\begin{table}[!ht]
\centering
\begin{tabular}{c|C{1.1cm}|r|C{0.8cm}|C{1.7cm}|C{2cm}|C{3.8cm}}
Clustering method & runtime (sec) & $b$ & $Q_n$ & fraction of cross-edges & \#lines to be removed & resulting partition of largest bridge-block\\
\hline \hline
Spectral $L_n$  &   3.949 &  2 &   0.375 &              0.0076 &                   24 &                [674, 1436] \\
Spectral $B_n$ &    2.702 &  2 &    0.379 &               0.0058 &                   18 &                [690, 1420] \\
FastGreedy  &  0.606 &  2 &   0.461 &               0.0489 &                  159 &               [1000, 1110] \\
\hline
Spectral $L_n$  &   4.394 &  3 &   0.640 &               0.0150 &                   47 &            [659, 698, 753] \\
Spectral $B_n$ &   2.837 &  3 &   0.640 &               0.0159 &                   50 &            [667, 702, 741] \\
FastGreedy  &  0.987 &  3 &   0.632 &               0.0516 &                  167 &           [410, 590, 1110] \\
\end{tabular}
\caption{Polish 2737sop}
\end{table}
\FloatBarrier

In the vast majority of the considered cases, the two spectral clustering methods yield partitions with smaller values of the normalized modularity $Q_n$ with respect to the FastGreedy method, which, however, consistently returns partitions with a larger number of cross-edges than the other two methods. In our experiments, the FastGreedy method is considerably (2x-10x times) faster than the two spectral clustering methods. 

FastGreedy has the advantage that all the partitions for various $b$ are calculated jointly because the algorithm is nested in that each partition of size $b$ is obtained from the corresponding one of size $b-1$ by optimally dividing one of the clusters into two. This recursive nature of the FastGreedy method makes it faster, but at times it means that partitions with unbalanced clusters sizes could be returned. On the other hand, the two spectral clustering methods almost always identify partitions with clusters of balanced size.

In the cases considered above, all methods always yield partitions with exactly $b$ connected clusters. It is possible that the actual number of identified \textit{connected} clusters could be strictly larger than $b$ because some of them induce network subgraphs that are disconnected, violating our assumption (A1). This occurs very rarely as the OBI objective function severely penalizes such partitions, but even so it would not be an issue as it would lead to a finer bridge-block decomposition. 


\subsection{Performance of one-shot and recursive algorithms}
\label{sub:switchingnum}

We first illustrate the numerical results for the one-shot algorithm, i.e., Algorithm~\ref{alg:oneshotrefinement}, for various test networks. 
More specifically, for each network we solve the OBS problem~\eqref{eq:optswitch_basic} three times, using the (possibly different) optimal partition $\calP$ into $b$ clusters returned by each of the three different clustering algorithms (cf.~Subsection~\ref{sub:modnum}). 
Recall that there is a one-to-one correspondence between the feasible subsets of cross-edges $\calE \subset E_c(\calP)$ and the spanning trees of the reduced multi-graph $G_\calP$. This means that the one-shot algorithm cannot be solve exactly already for partitions into $b \geq 5$ clusters, since the spanning trees of the reduced multi-graph grow exponentially with $b$. However, if we consider a small number of target clusters, say $b\leq 4$, we can still solve the OBS exactly by considering all possible spanning tree of $G_\calP$.

As pre-processing step, for each of the considered networks we first run a DC-OPF to obtain the power flow configuration and calculate the initial network congestion level $\gamma(\emptyset)$  (which is possibly equal to $1$ if one or more line limits are active for the OPF solution). We then run Algorithm~\ref{alg:oneshotrefinement} with $b=4$ using the three different clustering methods discussed in the previous section. Table~\ref{tab:one-shot_results} below reports for each network both the initial congestion level $\gamma(\emptyset)$ and the results of the one-shot algorithm ran with $b=4$ as target number of clusters. More specifically, for each network and each clustering algorithm, we report
\begin{itemize}
    \item the runtime,
    \item the number of spanning trees of $G_\calP$ considered,
    \item the optimal post-switching network congestion level $\gamma(\calE^*)$,
    \item the number of lines that to be switched off for the optimal set of switching actions $\calE^*$ and which fraction they represent on the total number of network lines.
\end{itemize}

\begin{table}[!ht]
    \centering
    \begin{tabular}{p{1.85cm}|p{2.1cm}|p{1.15cm}|p{1.6cm}|p{1cm}|p{1.85cm}|p{1.8cm}|p{1.8cm}}
network & partition method &  runtime (sec) & \# spanning trees &  $\gamma(\calE^*)$ & \# congested lines & \# of lines switched off & $\%$ of lines switched off\\
\hline \hline
\multirow{4}{*}{case39\_epri} 
& \multicolumn{2}{l}{initial network configuration} & - & 1.000 & 2 & - & -\\
\cline{2-8}
& Spectral $L_n$  &  0.238 &             10 &      0.960 &                0 &              3 &            6.52 \\
& Spectral $B_n$ &   0.270 &             10 &      0.960 &                0 &              3 &            6.52 \\
& \textbf{FastGreedy}  &  0.245 &             12 &  \textbf{0.833} &                \textbf{0} &              3 &            6.52 \\
\hline
\hline
\multirow{4}{*}{case57\_ieee} 
& \multicolumn{2}{l}{initial network configuration} & - & 0.938 & 0 & - & -\\ 
\cline{2-8}
& Spectral $L_n$  &   6.210 &            220 &   1.050 &                2 &             13 &            16.25 \\
& Spectral $B_n$ &   5.454 &            216 &   1.078 &                2 &             12 &            15.00 \\
& \textbf{FastGreedy}  &  6.829 &            256 &  \textbf{0.921} &                \textbf{0} &             14 &            17.50 \\
\hline
\hline
\multirow{4}{*}{case73\_ieee}
& \multicolumn{2}{l}{initial network configuration} & - & 0.632 & 0 & - & -\\
\cline{2-8}
& Spectral $L_n$  &   1.330 &             49 &   1.171 &                1 &              8 &           6.66 \\
& Spectral $B_n$ &   2.010 &             71 &   1.171 &                3 &             10 &           10.83 \\
& \textbf{FastGreedy}  &  0.970 &             31 &  \textbf{0.723} &                \textbf{0} &              6 &           5.00 \\
\hline
\hline
\multirow{4}{*}{case118\_ieee} 
& \multicolumn{2}{l}{initial network configuration} & - & 1.000 & 2 & - & - \\
\cline{2-8}
& \textbf{Spectral} $L_n$ &  1.912 &             80 &    \textbf{1.004} &                \textbf{2} &             10 &           5.38 \\
& Spectral $B_n$ &   2.259 &             80 &  1.016 &                3 &             10 &           5.38 \\
& FastGreedy  &  7.958 &            264 &  2.248 &                8 &             18 &           9.68 \\
\hline
\hline
\multirow{4}{*}{case179\_goc} 
& \multicolumn{2}{l}{initial network configuration} & - & 1.000 & 4 & - & -\\
\cline{2-8}
& Spectral $L_n$  &  1.699 &             54 &  1.243 &                7 &              7 &           2.66 \\
& Spectral $B_n$ &  10.358 &            312 &  1.035 &                2 &             13 &           4.94 \\
& \textbf{FastGreedy}  &  2.103 &             69 &  \textbf{1.000} &                \textbf{1} &              9 &            3.42 \\
\hline
\hline
\multirow{4}{*}{case200\_activ} 
& \multicolumn{2}{l}{initial network configuration} & - & 0.708 & 0 & - & -\\
\cline{2-8}
& Spectral $L_n$  &  6.541 &            196 &   0.591 &                0 &             15 &           6.12 \\
& Spectral $B_n$ &  9.146 &            312 &   0.591 &                0 &             16 &           6.53 \\
& \textbf{FastGreedy}  &  6.863 &            208 &   \textbf{0.591} &                \textbf{0} &             12 &           4.90 \\
\hline
\hline
\multirow{4}{*}{case300\_ieee} 
& \multicolumn{2}{l}{initial network configuration} & - & 1.000 & 11 & - & -\\
\cline{2-8}
& Spectral $L_n$  &  3.610 &            120 &  1.220 &                7 &             12 &           2.92 \\
& \textbf{Spectral} $B_n$ &  14.275 &            468 &  \textbf{1.058} &                \textbf{5} &             14 &           3.41 \\
& FastGreedy  &  37.212 &           1112 &  1.161 &                6 &             24 &           5.84 \\
\hline
\hline
    \end{tabular}
    \caption{Results obtained by the one-shot algorithm with $b=4$ target clusters using three different methods for the candidate bridge-blocks identification. We highlight in bold the best method and the resulting post-switching congestion $\gamma(\calE^*)$ for each network.}
    \label{tab:one-shot_results}
\end{table}
\FloatBarrier
From Table~\ref{tab:one-shot_results} we deduce that the resulting post-switching network congestion level varies a lot depending on the test network and partition method. In particular, in a few instances the maximum congestion is strictly bigger than $1$, but when considering all three methods, there is always at least one solution that does not increase the congestion significantly. Furthermore, in many instances the post-switching network congestion level even decreases with respect to the initial configuration and so does the number of congested lines. 

We did not run the one-shot algorithm on larger networks, as it would have required a substantially (exponentially) larger amount of time. This is not a surprise, in view of the discussion in Section~\ref{sub:pra} regarding the feasible region of the OBS problem. Indeed, the runtime of Algorithm~\ref{alg:oneshotrefinement} does not depend on the network size per-se, but rather on the the number of spanning trees of the reduced graph $G_\calP$ corresponding to the optimal partition $\calP$ identified by each of the three methods. This makes the runtime of the one-shot algorithm rather unpredictable, discouraging its usage in real time.


Let us now focus on the performance of the recursive algorithm, i.e.~Algorithm~\ref{alg:recursiverefinement}. In Table~\ref{tab:recursive_results} below, we report the results for the same test networks considered in Table~\ref{tab:one-shot_results} s well as a few more networks of larger size. Indeed, the running time of the recursive algorithm depends solely on the number of spanning trees that needs to be considered in the OBS problem. By design, for the recursive algorithm this number is in most cases very small even for large networks, as we can deduce from the second last column of Table~\ref{tab:recursive_results}. Indeed for the recursive algorithm the number of considered spanning trees is equal to the number of cross-edges of the selected bipartition, which is in turn equal to one plus the number of lines switched off.

For each test network, we run the recursive algorithm with $i_\mathrm{max}=3$, reporting the runtime of each of these three iterations. This choice for $i_\mathrm{max}$ allows us to make a fair comparison with the one-shot algorithm with a target of $b=4$ clusters, since we expect the same number of resulting bridge-blocks (modulo ``spontaneous'' splits in smaller bridge-blocks).

\begin{table}[!ht]
    \centering
    \begin{tabular}{p{1.9cm}|p{3.2cm}|p{1.25cm}|p{2cm}|p{1.85cm}|p{1.8cm}|p{1.8cm}}
network & iteration &  runtime (sec) &  $\gamma(\calE^*)$ & \# congested lines & \# of lines switched off & $\%$ of lines switched off\\
\hline \hline
\multirow{5}{*}{case39\_epri} 
& initial configuration & - &  $\gamma(\emptyset) =1.000$ & 2 & - & -\\
\cline{2-7}
& 1 &   0.091 &                0.794 &                0 &              2 &                4.35 \\
& 2 &   0.069 &                0.833 &                0 &              1 &                2.17 \\
& 3 &  0.082 &                0.833 &                0 &              1 &                2.17 \\
\cline{2-7}
& final &  0.242 &  \textbf{0.833} & \textbf{0} & \textbf{4} & \textbf{8.70} \\
\hline
\hline
\multirow{4}{*}{case57\_ieee} 
& initial configuration & - &  $\gamma(\emptyset) =0.938$ & 0 & - & -\\ 
\cline{2-7}
& 1 &   0.371 &                1.038 &                2 &             10 &                     12.5 \\
& 2 &   0.142 &                1.038 &                2 &              3 &                     3.75 \\
& 3 &  0.085 &                 1.038 &                2 &              1 &                     1.25 \\
\cline{2-7}
& final &  0.598 &   \textbf{1.038} &  \textbf{2} & \textbf{14} & \textbf{17.5} \\
\hline
\hline
\multirow{4}{*}{case73\_ieee}
& initial configuration & - &  $\gamma(\emptyset) = 0.632$ & 0 & - & -\\
\cline{2-7}
& 1 &  0.076 &                0.778 &                0 &              1 &                 0.83 \\
& 2 &   0.099 &                0.772 &                0 &              2 &                  1.67 \\
& 3 &   0.156 &                0.694 &                0 &              4 &                  3.33 \\
\cline{2-7}
& final &  0.242 &   \textbf{0.694} & \textbf{7} & \textbf{6} &  \textbf{5.83} \\
\hline
\hline
\multirow{4}{*}{case118\_ieee} 
& initial configuration & - &  $\gamma(\emptyset) = 1.000$ & 2 & - & - \\
\cline{2-7}
& 1 &   0.278 &                1.011 &                2 &              5 &                  2.69 \\
& 2 &  0.168 &                1.045 &                2 &              4 &                  2.15 \\
& 3 &  0.135 &                1.045 &                2 &              3 &                   1.61 \\
\cline{2-7}
& final &  0.581 &  \textbf{1.045} & \textbf{2} &  \textbf{12} & \textbf{6.45} \\
\hline
\hline
\multirow{4}{*}{case179\_goc} 
& initial configuration & - &  $\gamma(\emptyset) =1.000$ & 4 & - & -\\
\cline{2-7}
& 1 &   0.148 &                1.382 &                5 &              3 &                  1.14 \\
& 2 &   0.154 &                1.382 &                5 &              2 &                 0.76 \\
& 3 &   0.251 &                1.382 &                5 &              6 &                  2.28 \\
\cline{2-7}
& final &  0.553 &  \textbf{1.382} & \textbf{5} & \textbf{11} & \textbf{4.18} \\
\hline
\hline
\multirow{4}{*}{case200\_activ} 
& initial configuration & - &  $\gamma(\emptyset) =0.708$ & 0 & - & -\\
\cline{2-7}
& 1 &   0.221 &                 0.591 &                0 &              5 &                  2.04 \\
& 2 &   0.281 &                 0.591 &                0 &              8 &                  3.26 \\
& 3 &   0.141 &                0.605 &                0 &              3 &                  1.22 \\
\cline{2-7}
& final &  0.643 &  \textbf{0.605} &  \textbf{0} & \textbf{16} & \textbf{6.52} \\
\hline
\hline
\multirow{4}{*}{case300\_ieee} 
& initial configuration & - &  $\gamma(\emptyset) =1.000$ & 11 & - & -\\
\cline{2-7}
& 1 &  0.509 &               1.161 &                5 &             12 &                  2.91 \\
& 2 &  0.223 &               1.161 &                5 &              3 &                 0.73 \\
& 3 &  0.386 &               1.197 &                7 &              8 &                  1.95\\
\cline{2-7}
& final & 1.118 &   \textbf{1.197} & \textbf{7} & \textbf{23} & \textbf{5.59} \\
\hline
\hline
\multirow{4}{*}{case1888\_rte} 
& initial configuration & - &  $\gamma(\emptyset) =1.000$ & 14 & - & -\\
\cline{2-7}
& 1 &     4.567 &              0.869 &                0 &             48 &                  1.89 \\
& 2 &    1.994 &               0.869 &                0 &              9 &                 0.36 \\
& 3 &    2.694 &               0.869 &                0 &             24 &                 0.95 \\
\cline{2-7}
& final &  9.225 &  \textbf{0.869} & \textbf{0} & \textbf{81} & \textbf{3.20} \\
\hline
\hline
\multirow{4}{*}{case2735\_sop} 
& initial configuration & - &  $\gamma(\emptyset) =1.000$ & 1 & - & -\\
\cline{2-7}
& 1 &   12.57 &              0.964 &                0 &            116 &                  3.54 \\
& 2 &   1.317 &              0.964 &                0 &             19 &                  0.58 \\
& 3 &   2.145 &              2.637 &               13 &             33 &                  1.00\\
\cline{2-7}
& final &  16.032 &  \textbf{2.637} & \textbf{13} & \textbf{168} & \textbf{5.12} \\
\hline
\hline
    \end{tabular}
    \caption{Results obtained by the recursive algorithm stopped after $3$ iterations using the FastGreedy method for the candidate bridge-blocks identification.}
    \label{tab:recursive_results}
\end{table}
\FloatBarrier

By comparing Tables~\ref{tab:one-shot_results} and~\ref{tab:recursive_results} it is clear that among the two algorithms, the one-shot algorithm is the one with slightly better performance, as it yields lower network congestion level for almost all the test networks considered. 
On the other hand, the recursive algorithm still yields good results (even when using only the FastGreedy method for the first subproblem), but it is orders of magnitude faster.

\FloatBarrier

\section{Concluding remarks \& future directions}
\label{sec:conclusion}

In this paper, we proposed a novel spectral representation of power systems which captures the topological and physical interactions within a power grid. Particularly, we illustrated that the graph Laplacian is of importance to characterize the power distributions. This spectral perspective further provides a simple, exact and clear representation of the power redistribution after line failures, where we introduced the network block decomposition to precisely characterize the boundary of the region a line failure can impact.

A finer block decomposition of the power grid creates more, smaller management regions whose operations can be isolated from each other in various aspects, bringing benefits such as disruption localizability, optimization parallelization, and more to the control and management of power systems. Therefore, we introduced a two-stage procedure to optimally modify the network topology for a finer block decomposition. Specifically, we first identify the candidate power network blocks by solving a weighted version of the network modularity problem and then identify the line switching action that has the least impact on the surviving network with respect to line congestion. This two-stage procedure is validated with numerical experiments. 

In the rest of this section, we present a collection of examples on how network block decomposition offers powerful analytical guarantees and benefits in different power system applications beyond the setting of this paper, and how we can leverage these properties to improve the current industry practices in these applications.

\paragraph{Real-time failure localization and mitigation.}
Reliability is a crucial part for the development of sustainable modern power systems. Current industry practice relies on simulations based contingency analysis which is usually constrained by computational power~\cite{song2016dynamic}. Since the cascading failures exhibit a complicated and non-local propagation pattern, such analysis tends to omit the structural properties of failures and may fail to control and mitigate the failure in real time. There are topological models adopted from complex network analysis, however, such models are relatively simple thus cannot reveal physical properties~\cite{Hines2010}. 

As introduced in Sections~\ref{sec:spectral} and \ref{sec:localization}, network block decomposition leads to clear and well-defined boundaries that can localize the impact of failures. This suggests a new technique for grid reliability by switching off a small number of transmission lines so that finer-grained network blocks are created. In~\cite{Guo2019}, a real-time distributed control strategy utilizing the localizability from network block decomposition is proposed. Such control strategy is guaranteed to (i) always prevent successive failures from happening; and (ii) localize the impact of initial failures as much as possible. Simulation results show that the controller significantly improves the overall grid reliability. 

\paragraph{Security constrained OPF.}   
Security constrained OPF (SC OPF) is another approach to improve grid reliability that works by adding a set of contingencies and requiring the operating point to be safe under any of such contingency. These additional constraints greatly increase the size of the problem~\cite{capitanescu2011state}. In practice, the SC OPF is solved with contingency filtering techniques \cite{capitanescu2008new} to reduce the problem size and/or iterative methods such as Alternating Direction Method of Multipliers (ADMM)~\cite{chakrabarti2014security} for distributed computations. 

With a finer network block decomposition, both approaches can be significantly accelerated. The block induces a natural way to decompose and decouple the SC OPF, reducing the number of constraints and variables in each region. In addition, the messages necessary to pass across are reduced, leading to a fast convergence for the local optimizations. Therefore, there is huge potential in accelerating the computation of SC OPF when the network can be decomposed into comparable blocks.

\paragraph{Distributed AC OPF.}
Optimal Power Flow (OPF) underlies the efficient and stable operation of power grids, reducing the generation cost and transmission loss while maintaining the system stability in steady state. AC OPF problems are challenging because of their highly non-convex nature stemming from the nonlinear relationship between voltages and the complex powers. To overcome this challenge, semidefinite programming (SDP) relaxation of the AC OPF problem, where the optimization variable is transformed to a matrix $\mathbf{V}$ that needs to be semidefinite positive, i.e. $\mathbf{V}\succeq 0$, has been proposed; see~\cite{MolzahnHiskens2019,lavaei2011zero} for surveys to a considerable literature.
However, the standard SDP solvers do not scale well with the network size, necessitating a distributed solution where multiple agents communicate with each other and solve the relaxed problem together. This has proved difficult because the $\mathbf{V}\succeq 0$ constraint is globally coupled, rendering ADMM to be inapplicable. Nevertheless, it is shown in \cite{dall2013distributed} that if the power grid exhibits a block decomposition and each block has an nonempty interior, then the partial matrix with known entries specified by these blocks is chordal
and thus can be completed to a feasible global solution $\mathbf{V}\succeq 0$ (see. e.g.,~\cite{6815671}). This enables the application of ADMM if the SDP optimization is decomposed in accordance with the network block structure and allows the optimization to be distributed to each block. Therefore, a finer network block decomposition helps reduce the overall operation cost by allowing AC OPF to be solved more efficiently.

\paragraph{Monitoring of price manipulation.}
Aggregators are playing an increasingly important role in the construction of a morden power grid, especially with the higher penetration level of renewable energy resources. The aggregators install, monitor and manage household energy devices such as rooftop solar panels and electric vehicle charging stations, and coordinate with utilities and Idependent System Operators (ISOs) to provide renewable generation and demand response. In addition to the benefits of aggregators, however, aggregators have the potential to manipulate prices through strategic curtailment for more profit~\cite{ruhi2017opportunities}. 
To quantify the potential for aggregators to exercise market power, in~\cite{ruhi2017opportunities} an algorithmic framework is given to approximately solve a bilevel quadratic optimization problem. Their method relies on decomposing the optimization problem into local parts and solve approximately with dynamic programming, thus only works for radial networks. It can be shown that similar bound holds in general networks if the block size of the network is upper bounded. Therefore, it is possible to extend the analysis from radial networks to general networks with proper block decomposition.


\FloatBarrier


\section*{Acknowledgements}
This research has been supported by NWO Rubicon grant 680.50.1529, Resnick Fellowship, Linde Institute Research Award, NSF grants through PFI:AIR-TT award 1602119, EPCN 1619352, CNS 1545096, CCF 1637598, ECCS 1619352, CNS 1518941, CPS 154471, AitF 1637598, ARPA-E grant through award DE-AR0000699 (NODES) and GRID DATA, DTRA through grant HDTRA 1-15-1-0003 and Skoltech through collaboration agreement 1075-MRA.


\appendix
\section{Appendix}

\subsection{Proof of Proposition~\ref{ch:lm; eq:Laplacian-properties}}

\begin{itemize}
    \item[(i)] From its definition~\eqref{eq:Laplaciandef}, it is clear that we can rewrite the Laplacian matrix $L$ as
    \[
        L = \sum_{(i,j) \in E} b_{ij} (\bm{e}_i \bm{e}_i^T + \bm{e}_j \bm{e}_j^T - \bm{e}_i \bm{e}_j^T - \bm{e}_j \bm{e}_i^T).
    \]
    Therefore, for any $\bm{x} \in \R^n$ we have 
    \begin{align*}
    \bm{x}^T L \bm{x} &= \bm{x}^T \Big(\sum_{(i,j) \in E} b_{ij} (\bm{e}_i \bm{e}_i^T + \bm{e}_j \bm{e}_j^T - \bm{e}_i \bm{e}_j^T - \bm{e}_j \bm{e}_i^T) \Big) \bm{x}
    = \sum_{(i,j) \in E} b_{ij} \, \bm{x}^T  (\bm{e}_i \bm{e}_i^T + \bm{e}_j \bm{e}_j^T - \bm{e}_i \bm{e}_j^T - \bm{e}_j \bm{e}_i^T) \bm{x}\\
    &= \sum_{(i,j)\in E} b_{ij} \left( x_i^2 - 2 x_i x_j + x_j^2 \right) =  \sum_{(i,j)\in E} b_{ij} (x_i - x_j)^2 \geq 0
    \end{align*}

\item[(ii)] In (i) we rewrote $\bm{x}^T L \bm{x}$ as the quadratic form $\sum_{(i,j)\in E} b_{ij} (x_i - x_j)^2$. This, in combination with the fact that $b_{ij} >0$ for every line $(i,j) \in E$ suffices to conclude that $L$ is positive semidefinite.
	
\item[(iii)] The row and column property immediately follows from the way the matrix $L$ is defined, cf.~\eqref{eq:Laplaciandef}. Since the kernel of $L$ contains at least the nontrivial vector $\bm{1}$, the matrix is singular.

\item[(iv)] We claim that a vector $\bm{v}$ is in the null space of $L$ if and only if $v_i = v_j$ for all nodes $i, j$ in the same connected component. For any such pair, there exists a path $\ell_1,\dots,\ell_r$ connecting them and for all the nodes traversed by this path (including $i$ and $j$) the corresponding entries of the vector $v$ must be equal, in view of the fact that $\bm{v}^T L \bm{v} = 0$ and that $b_{\ell_1},\dots,b_{\ell_r}>0$. The reverse implication is trivial. This claim implies that an orthonormal basis of the null space consists of $k$ orthogonal unit vectors $\bm{v}_j$, $j = 1, \dots, k$, whose entries are
\[
    (\bm{v}_j)_i := \frac{\bm{1}_{\calI_j}(i)}{\sqrt{|\calI_j|}}, \qquad i=1,\dots, n,
\]
where $\calI_1,\dots,\calI_k$ are the connected components of the graph $G$. Hence, the null space of $L$ has a dimension of $k$ and, therefore, $\mathrm{rank}(L) = n - \mathrm{dim}(\mathrm{ker}(L)) = n - k$.

Consider now the singular value decomposition of the symmetric real matrix $L$, that is $L = U\Sigma U^T$. The columns of $U$ are the orthonormal eigenvectors of $L$, namely $\frac{1}{\sqrt{|\calI_1|}} \bmo_{\calI_1},\dots,\frac{1}{\sqrt{|\calI_k|}} \bmo_{\calI_k}, \bm{v}_{k+1},\dots, \bm{v}_n$. Then the pseudo-inverse $L^\dag$ is equal to $L^\dag = U \Sigma^\dag U^T$, where $\Sigma^\dag$ is the diagonal $n\times n$ matrix of rank $n-k$ obtained from $\Sigma$ by replacing each the nonzero eigenvalues $\lambda_i$ of $L$ by its reciprocal $1/\lambda_i$.
The first identity 
\begin{equation}
\label{eq:Ldag}
    L^\dag = \sum_{j=k+1}^{n} \frac{1}{\l_j} \, \bm{v}_j \bm{v}_j^T
\end{equation}
readily follows from $L^\dag = U \Sigma^\dag U^T$ in view of the orthogonality of the eigenvectors and the structure of the null space described above. For the second identity, it is enough to show that 
\[
    \left( L + \sum_{j=1}^k \frac{1}{|\calI_j|} \bmo_{\calI_j} \bmo_{\calI_j}^T \right)^{-1}  = L^\dag + \sum_{j=1}^k \frac{1}{|\calI_j|} \bmo_{\calI_j} \bmo_{\calI_j}^T = \sum_{j=k+1}^n \frac{1}{\lambda_j} \bm{v}_j \bm{v}_j^T + \sum_{j=1}^k \frac{1}{|\calI_j|} \bmo_{\calI_j} \bmo_{\calI_j}^T.    
\]
We prove that the RHS of the latter equation is the inverse matrix by direct calculation
\[
    \left( L + \sum_{j=1}^k \frac{1}{|\calI_j|} \bmo_{\calI_j} \bmo_{\calI_j}^T \right)  \left(\sum_{j=k+1}^n \frac{1}{\lambda_j} \bm{v}_j \bm{v}_j^T + \sum_{j=1}^k \frac{1}{|\calI_j|} \bmo_{\calI_j} \bmo_{\calI_j}^T \right) = U U^T = I,
\]
and the conclusion follows by uniqueness of the matrix inverse. In the special case in which $G$ is connected, then $k=1$, $V = \calI_1$, and the above expression simplifies in view of the fact that $\frac{1}{|\calI_1|} \bmo_{\calI_1} \bmo_{\calI_1}^T = \frac{1}{n} \bmo \bmo^T$. 

\item[(v)] All the listed properties of $L^\dag$ can be easily deduced from (ii) and identity~\eqref{eq:Ldag}. \qedhere
\end{itemize}



\subsection{Proof of Proposition~\ref{ch:lm; prop:PTDF.2}}

(i) We can directly inspect that the $(\ell,\hat \ell)$-th element of the PTDF matrix $D$, i.e. $D_{\ell \ell}$, coincides with the $(\ell,\hat \ell)$-th element of the matrix $B C^T  L^\dag  C$. More specifically, using the first equation of Proposition~\ref{ch:lm; prop:PTDF.1} and choosing $\hat \ell=(s,t) \in E$, we have
\[
    (B C^T  L^\dag  C)_{\ell \hat\ell}
    = \mathbf{e}_\ell^T B C^T  L^\dag  C \mathbf{e}_{\hat\ell} 
    = b_{\ell} e_\ell^T C^T  L^\dag  C \mathbf{e}_{\hat\ell} 
    = b_{\ell} (\mathbf{e}_i - \mathbf{e}_j)^T L^\dag (\mathbf{e}_s - \mathbf{e}_t)
    = b_{\ell} \left( L^\dag_{i\hi} +  L^\dag_{j\hj} -  L^\dag_{i\hj} -  L^\dag_{j\hi} \right) = D_{\ell \hat\ell}.
\]

(ii) Let us consider the second expression given in Proposition~\ref{ch:lm; prop:PTDF.1} for $D_{\ell,\ell}$ in the special case where $\hat\ell = (s,t) = (i,j) = \ell$. We obtain
\begin{equation} \label{eq:Dll}
    D_{\ell, \ell} = b_{\ell} \ \frac{\sum_{\calE\in T(\{i,i\}, \{j,j\})}\, \beta(\calE) \ - \ \sum_{\calE\in T(\{i,j\}, \{j,i\})}\, \beta(\calE)}{\sum_{\calE\in T}\, \beta(\calE)} = b_{\ell} \ \frac{\sum_{\calE\in T(\{i\},\{j\})}\, \beta(\calE)}{\sum_{\calE\in T}\, \beta(\calE)},
\end{equation}
where the last identity follows from the fact that $T(\{i,j\}, \{j,i\})=\emptyset$. The collection $T=T_E$ of all spanning trees of the graph $G$ can be partitioned into two subsets, depending on whether the edge $\ell=(i,j)$ belongs to it or not. All the spanning trees that do contain $\ell$ are in one-to-one correspondence with the spanning forests in $T(\{i\},\{j\})$. Indeed, each such spanning tree, once the edge $\ell$ is removed, yields a spanning forest consisting of two trees, one containing node $i$ and the other one node $j$. The other subset consisting of all the spanning trees that do not contain $\ell$ is $T \!\setminus\! \{\ell\}$ by definition. In view of these considerations, we can rewrite the sum of all spanning tree weights as
\[
    \sum_{\calE\in T}\, \beta(\calE) = b_\ell \sum_{\calE\in T(\{i\},\{j\})}\, \beta(\calE) \ + \  \sum_{\calE\in T\!\setminus\! \{\ell\}}\, \beta(\calE).
\]
Using this identity to rewrite the denominator appearing in the RHS of~\eqref{eq:Dll} we obtain
\[
\frac{b_{\ell} \sum_{\calE\in T(\{i\},\{j\})}\, \beta(\calE)} {\sum_{\calE\in T}\, \beta(\calE)} 
=  \frac{b_{\ell} \sum_{\calE\in T(\{i\},\{j\})}\, \beta(\calE)} {b_\ell \sum_{\calE\in T(\{i\},\{j\})}\, \beta(\calE) \ + \  \sum_{\calE\in T\!\setminus\! \{\ell\}}\, \beta(\calE) } 
= 1 \ - \  \frac{ \sum_{\calE\in T\!\setminus\! \{\ell\} }\, \beta(\calE)} {\sum_{\calE\in T}\, \beta(\calE)}.
\]
Since the collection $T\!\setminus\! \{\ell\}$ could be empty but it is always strictly included in $T$, we immediately deduce that $0 < D_{\ell\ell} \leq 1$.

(iii) The two identities readily follow from Proposition~\ref{ch:lm; prop:PTDF.1} in the special case $\ell=(s,t)$ combined with the definition of effective resistance given in~\eqref{eq:effr}. \qedhere

\subsection{Proof of Theorem~\ref{thm:zeroPTDF}}
We prove the result by contradiction. If $D_{\ell \hat\ell} \neq 0$, then at least one of the two sets $T(\{i,s\}, \{j,t\})$ and $T(\{i,t\}, \{j,s\})$ of spanning forests appearing in the numerator of the RHS of~\eqref{eq:Dlst} is nonempty. Suppose $T(\{i,s\}, \{j,t\})$ is nonempty and take any spanning forest $F \in T(\{i,s\}, \{j,t\})$. The tree in the forest $F$ that contains buses $i$ and $s$ defines a path from $i$ to $s$, and the other tree in $F$ that contains $j$ and $t$ defines a path from $j$ to $t$. Clearly these two paths do not share any vertices. If we add the lines $\ell = (i, j)$ and $\hat \ell = (s, t)$ to these two vertex-disjoint paths we obtain a simple cycle that contains both $\ell$ and $\hat\ell$. \qedhere

\end{document}